\documentclass[a4paper,reqno,11pt]{amsart}
\usepackage[margin=1.3in]{geometry}
\usepackage[utf8]{inputenc}	
\usepackage{indentfirst} %Per avere il rientro nella prima riga di ogni capitolo.
\usepackage{amsmath,amsfonts,amssymb,,mathrsfs}
\usepackage{bbm}
\usepackage[colorlinks=true]{hyperref}
\usepackage{enumerate}
\usepackage{graphicx}
\usepackage{xcolor}
\usepackage{enumitem}
\newcommand{\restr}{%
  \,\raisebox{-.127ex}{\reflectbox{\rotatebox[origin=br]{-90}{$\lnot$}}}\,%
}

\allowdisplaybreaks

\theoremstyle{definition}
\newtheorem{definition}{Definition}[section]
\theoremstyle{remark}
\newtheorem{remark}[definition]{Remark}
\theoremstyle{theorem}
\newtheorem{theorem}[definition]{Theorem}
\theoremstyle{theorem}
\newtheorem{lemma}[definition]{Lemma}
\theoremstyle{remark}

\theoremstyle{theorem}
\newtheorem{corollary}[definition]{Corollary}
\theoremstyle{theorem}
\newtheorem{proposition}[definition]{Proposition}
\theoremstyle{theorem}

\numberwithin{equation}{section}
\allowdisplaybreaks

\def\Xint#1{\mathchoice
    {\XXint\displaystyle\textstyle{#1}}%
    {\XXint\textstyle\scriptstyle{#1}}%
    {\XXint\scriptstyle\scriptscriptstyle{#1}}%
    {\XXint\scriptscriptstyle\scriptscriptstyle{#1}}%
    \!\int}
\def\XXint#1#2#3{{\setbox0=\hbox{$#1{#2#3}{\int}$}
      \vcenter{\hbox{$#2#3$}}\kern-.5\wd0}}

\def\mint{\Xint-}

\definecolor{ao}{rgb}{0.0, 0.5, 0.0}

\DeclareMathOperator*{\aplim}{ap-\lim}
\DeclareMathOperator*{\aplims}{ap-\limsup}
\DeclareMathOperator*{\aplimi}{ap-\liminf}

\newcommand{\Om}{\Omega}
\newcommand{\R}{\mathbb{R}}
\newcommand{\C}{\mathbb{C}}

\newcommand{\HH}{\mathcal{H}}
\newcommand{\di}{\mathrm{d}}

\newcommand{\e}{\mathrm{E}}

\newcommand{\EEE}{\color{black}}
\newcommand{\RRR}{\color{black}}

\def\namedlabel#1#2{\begingroup
    #2%
    \def\@currentlabel{#2}%
    \phantomsection\label{#1}\endgroup
}

\makeatletter
\@namedef{subjclassname@2020}{%
 \textup{2020} Mathematics Subject Classification}
\makeatother

\title[Generalized bounded deformation in non-Euclidean settings]{Generalized bounded deformation \\in non-Euclidean settings}

\author[S. Almi]{Stefano Almi}
\address[Stefano Almi]{Universit\'a di Napoli Federico II, Dipartimento di Matematica e Applicazioni R. Caccioppoli, via Cintia, Monte S. Angelo, 80126 Naples, Italy. }
\email{stefano.almi@unina.it}

\author[E. Tasso]{Emanuele Tasso}
\address[Emanuele Tasso]{Institute of Analysis and Scientific Computing, TU Wien, Wiedner-Hauptstrasse 8-10, 1040 Vienna, Austria}
\email{emanuele.tasso@tuwien.ac.at}



\begin{document}

\begin{abstract}
We introduce a new space of generalized functions of bounded deformation $GBD_{F}$, made of functions $u$ whose one-dimensional slice $u(\gamma) {\, \cdot\,} \dot{\gamma}$ has bounded variation in a generalized sense for all curves $\gamma$ solution of the second order ODE $\ddot{\gamma} = F(\gamma, \dot{\gamma})$ for a fixed field $F$. For $u \in GBD_{F}$ we study the structure of the jump set in connection its slices and prove the existence of a curvilinear approximate symmetric gradient. With a particular choice of $F$ in terms of the Christoffel symbols of a Riemannian manifold~${\rm M}$, we are able to define and recover similar properties for a space of $1$-forms on~$\rm M$ which have generalized bounded deformation in a suitable sense.
\end{abstract}

\maketitle

\section{Introduction}
\label{s:intro}

In the last decades the study of Free Discontinuity functionals has lead to the development of different notions of functions with bounded variation. We recall here the original $BV$ and $BD$ spaces~\cite{MR1480240, afp, MR711964} and the unifying approach~\cite{arr, MR4024554} of functions of bounded $\mathcal{A}$-variation~$BV^{\mathcal{A}}$. In a more applied framework, the spaces~$GBV$ and~$GBD$ have been introduced in~\cite{afp, dal} to supply to the lack of integrability of the field $u$ and of the jump~$[u]$ with respect to the $(n-1)$-dimensional Hausdorff measure~$\mathcal{H}^{n-1}$ restricted to the jump set~$J_{u}$. The spaces~$GBD$ and~$GSBD$, in particular, have found applications in the study of functionals of the form
\begin{equation}
\label{e:intro0}
\int_{\Om} | e(u)|^{2} \, \di x + \HH^{n-1}(J_{u})\,,
\end{equation}
where $e(u)$ denotes the approximate symmetric gradient of~$u$. In this respect, we mention results on compactness and lower semicontinuity~\cite{AT_22, CCS_22, Cha-Cri_23, MR4210722, MR4205185, MR4076067, MR4079210}, Ambrosio-Tortorelli approximations~\cite{MR4215197, MR3780140, MR3928751, MR3249854, MR3247391}, dimension reduction, homogenization, atomistic derivation, and nonlocal approximations~\cite{MR4216047, ARS_22, AT_21, MR3593536, KreFriZem22, Fri-Per-Sol_hom, GG_22, MarSol22, MR3669836, Sch-Zem2, Sch-Zem}, linearization in elasticity~\cite{ADF_22, MR3634030, MR4139444}, and modeling of fracture, epitaxially strained films, and stress-driven rearragnement instabilities~\cite{MR4097334, MR4332457, MR3739927, MR4121137}.

The common feature of the above mentioned works is that the underlying ambient space is of euclidean type. Nevertheless, there are number of interesting applications in which the reference configuration is represented by a Riemannian manifold $\rm (M,g)$. For instance, when deriving model of brittle fractures on linearly elastic shells embedded in $\mathbb{R}^3$ via dimension reduction, the limiting energy and function spaces have to take into account the geometry of the shell. This was noticed, e.g., in \cite{MR4216047, LeDret1996} where the limit displacements are expressed in curvilinear coordinates, hinging on the fact that one still has a control of the full gradient displacements. However, in the linear elastic case, the lack of a control on the full gradient requires the definition of an intrinsic $BD$-like space which is sensible of the geometry of the reference manifold. Indeed, the classical approach to study the structure of $BD$-functions in $\mathbb{R}^n$ relies on slicing techniques, where the vector fields are restricted to and projected onto lines, which are the \emph{geodesics} of $\mathbb{R}^n$. The equivalent procedure on $\rm M$ consists in replacing lines with the geodesics given by the metric $\rm g$. Furthermore, following the approach in~\cite{MR1757535}, differently from the euclidean case in which displacements are modelled as vector fields, it is more convenient to use the representation as one-forms $\omega$ on $\rm M$. This leads us to a local description of $\omega$ in terms of a vector field $u$ in $\mathbb{R}^n$ whose entries are its contravariant components. If $\omega$ is smooth, its symmetric gradient writes in curvilinear coordinates as the matrix
\cite[Section~1.2]{MR1757535}
\begin{equation}
\label{e:intro3}
(E(u))_{ij} = \frac{1}{2}  ( \partial_{i} u_{j} + \partial_{j} u_{i})  - \sum_{\ell = 1}^{n} \Gamma^{\ell}_{ij} u_{\ell} \qquad i, j= 1, \ldots, n\,,
\end{equation}
where $\partial_{i}$ denotes the derivative with respect to~$x_{i}$ and $\Gamma^{\ell}_{ij}$ are the Christoffel symbols. Moreover if $\gamma$ is a geodesic of $\rm M$ expressed in coordinates, for $t \in \R$ it holds true
\begin{align}
\label{e:intro2}
 \frac{\di}{\di \tau}\Big|_{\tau = t} u (\gamma(\tau)) \cdot \dot{\gamma} (\tau) &  =   \nabla{u} (\gamma(t)) \dot{\gamma} (t) \cdot \dot{\gamma} (t)  - \!\!\!  \sum_{\ell, i, j = 1}^{n} \Gamma^{\ell}_{ij} (\gamma(t)) u_{\ell} (\gamma(t)) \dot{\gamma}_{i} (t) \dot{\gamma}_{j} (t) 
 \\
 &
   =    Eu (\gamma(t))  \dot{\gamma} (t) \cdot  \dot{\gamma} (t)\,, \nonumber
\end{align}
where the first equality follows from the geodesics equation $\ddot{\gamma}_\ell = - \sum_{i, j=1}^{n} \Gamma^{\ell}_{ij} (\gamma) \dot{\gamma}_{i}  \dot{\gamma}_{j}$.

% The main feature characterizing all of the above spaces is that their fine properties can be studied with slicing techniques (see, e.g.,~\cite[Section~3.11]{afp}). In order to better describe the space we are going to introduce, let us focus on the $BD$-setting. In this case, the symmetric gradient~$Eu$ of~$u$ can be reconstructed by looking at suitable one-dimensional slices. This is justified by the following elementary consideration: given a smooth function $u \colon \R^{n} \to \R^{n}$ and a direction $\xi \in \mathbb{S}^{n-1}$, we obtain that 
% \begin{equation}
% \label{e:intro1}
% Eu (x) \xi \cdot \xi = {\rm D}_{\xi} u (x)  \cdot \xi = \frac{\di}{\di \tau}\Big|_{\tau = t} u(y + \tau \xi) \cdot \xi  
% \end{equation}
% for $x \in \R^{n}$, $y \in \xi^{\bot}$, and $t \in \R$ such that $x = y + t\xi$. This basic idea can be extended to the case of a $BD$-function~$u$, namely, when $Eu$ is a finite Radon measure. Notice that the Euclidean structure of the ambient space influences the choice of the slices, which are indeed taken along the geodesics of~$\R^{n}$.

Motivated by~\eqref{e:intro2} and \cite{dal}, we consider a space of functions whose one-dimensional slices~$t \mapsto u (\gamma(t)) \cdot \dot{\gamma} (t)$ have bounded variation when computed on solutions~$\gamma$ of suitable second order ODEs. More precisely, we fix a smooth field $F \colon \R^{n} \times \R^{n} \to \R^{n}$ which is a quadratic form in the second variable (cf.~\eqref{hp:F}) and a class of {\em curvilinear projections} $(P_{\xi})_{\xi \in \mathbb{S}^{n-1}}$ from $\Om$ to $\xi^{\bot}$ satisfying the {\em transversality} condition of Definition~\ref{d:transversal} (see also \cite[Definition~2.4]{hov}) and whose level sets $P_{\xi}^{-1}(y)$ are   the images of the map $t \mapsto \varphi_{\xi} (y + t\xi)$ solution of $\ddot{\gamma} = F(\gamma, \dot \gamma)$. Moreover, for $E \subseteq \Om$, $u \colon \Om \to \R^{n}$, $\xi \in \mathbb{S}^{n-1}$, and $y \in \xi^{\bot}
 $ we define the slices 
\begin{align}
\nonumber
E^{\xi}_{y} & := \{ t \in \R: \varphi_{\xi}(y + t\xi) \in E\}\,,
\\
\label{e:introslice}
\hat{u}^{\xi}_{y} (t) & := u(\varphi_{\xi} (y + t\xi)) \cdot \dot{\varphi}_{\xi} (y + t\xi) \qquad \text{for $t \in \Om^{\xi}_{y}$}\,.
\end{align}
Then, we say that a measurable function $u \colon \Om \to \R^{n}$ belongs to $GBD_{F}(\Om)$ if there exists a positive bounded Radon measure $\lambda$ on~$\Om$ such that for every $\xi\in \mathbb{S}^{n-1}$, for $\mathcal{H}^{n-1}$-a.e.~$y \in \xi^{\bot}$ we have that $\hat{u}^{\xi}_{y} \in BV_{loc} (\Om^{\xi}_{y})$ and 
\begin{align}
\int_{\xi^{\bot}} | \mathrm{D} (\tau (\hat{u}^{\xi}_{y})) | ( \Om^{\xi}_{y}) \, \di \mathcal{H}^{n-1} (y) \leq \| \dot{\varphi}_{\xi}\|_{\infty} \, {\rm Lip} (P_{\xi})^{n-1} \lambda(\Om)\,,
\end{align}
whenever $\tau \in C^{1} (\R)$ is such that $-\frac{1}{2} \leq \tau \leq \frac{1}{2}$ and $0 \leq \tau'\leq 1$. We refer to Section~\ref{s:definition} for the precise definition. We remark that if $F=0$ and the family $(P_{\xi})_{\xi \in \mathbb{S}^{n-1}}$ is the family of orthogonal projections~$\pi_{\xi}\colon \R^{n} \to \xi^{\bot}$, the space $GBD_{F}(\Om)$ coincides with the space~$GBD(\Om)$ introduced in~\cite{dal}. 

By means of the space $GBD_F(\Omega)$, with the choice
\begin{align}
\label{e:choiceF}
F_{\ell} (x, \zeta) = - \sum_{i, j=1}^{n} \Gamma^{\ell}_{ij} (x) \zeta_{i}  \zeta_{j} \qquad \text{for $\ell= 1, \ldots, n$, $x \in \Om$, and $\zeta \in \R^{n}$},
\end{align}
it is possible to define an intrinsic space of measurable one-forms on $\rm M$ having generalised bounded deformation, which we denote by $GBD(\rm{M})$. We refer for all details to Sections \ref{s:GBD-M} and \ref{s:equivalence}.

In this paper we mainly focus on the structure of the jump set in relation with one dimensional slices \eqref{e:introslice} and on the existence of an approximate symmetric gradient. In this regard we make use of the recent result \cite{del} ensuring that the jump set of any measurable function is countably $(n-1)$-rectifiable. However, it is of fundamental importance to establish a precise relation between the slices of the jump set and the jump sets of the slices. In this regard, the non-linear nature of our setting leads to a crucial point, which is the lack of symmetry required to exploit the parallelogram law \cite{MR1480240} (see also \cite[Formula 7.1]{dal} and \cite[Proposition 4.4]{arr}) and the possibility to perform codimension-one slicing on which the $BD$ and $BV^{\mathcal{A}}$ theories hinge. In order to overcome this difficulty we have to make a stronger assumption on the field~$F$. Namely, we suppose $F$ to satisfy a condition which we call \emph{Rigid Interpolation} \eqref{hp:F2}. Such a condition requires a (local) control on the $L^\infty$-norm of the curvilinear symmetric gradient, seen as an operator acting on smooth vector fields, in terms of a discrete semi-norm defined on the vertices of $n$-dimensional simplexes of $\mathbb{R}^n$ (see \eqref{e:rip5}). Appealing to the general slicing criterion developed in \cite{AT_22-preprint} we combine Lemma \ref{l:everyxi} to guarantee that, given a family of curvilinear projections $(P_\xi)_{\xi \in \mathbb{S}^{n-1}}$,  the jump set of $u \in GBD_F(\Omega)$ can be sliced by means of the jump sets of the one dimensional restrictions $\hat{u}^\xi_y$ (see Theorem~\ref{p:euju}). In addition, since the choice of the field $F$ dictated by \eqref{e:choiceF} does satisfy the Rigid Interpolation property, the above mentioned result obtained in the $GBD_F$-context can be easily transferred to the Riemannian case $GBD(\rm{M})$ (see Theorem \ref{t:jumpslicingform}). 

Eventually, we prove in Theorem~\ref{t:apsym} that every $u \in GBD_{F}(\Om)$ admits a.e.~in~$\Om$ an approximate symmetric gradient~$\tilde{e} (u) \in \mathbb{M}^{n}_{sym}$, where $\mathbb{M}^{n}_{sym}$ denotes the space of symmetric matrices of order~$n$. In this setting, we notice that $\tilde{e} (u)$ is not integrable in~$\Om$. Nevertheless, we may define a ``curvilinear'' approximate symmetric gradient $e(u) \colon  \Om \to \mathbb{M}^{n}_{sym}$ as
 \begin{equation*}
%\label{e:intro4}
e(u)(x)\zeta \cdot \zeta := \tilde{e}(u)(x) \zeta \cdot \zeta - u(x) \cdot F(x, \zeta) \qquad \text{for a.e.~$x \in \Om$ and every $\zeta \in \mathbb{R}^n$}.
\end{equation*}
 Then, it turns out that $e(u) \in L^{1}(\Om; \mathbb{M}^{n}_{sym})$, which is consistent with~\eqref{e:intro3}. Furthermore,~$e(u)$ can be reconstructed by means of the approximate gradients of the one-dimensional slices: for $\mathcal{H}^{n-1}$-a.e.~$\xi \in \mathbb{S}^{n-1}$ it holds
\begin{displaymath}
\nabla{\hat{u}}^{\xi}_{y} (t) = (e(u))^{\xi}_{y} \, \dot{\varphi}_{\xi} (y + t\xi) \cdot \dot{\varphi}_{\xi} (y + t\xi) \qquad \text{for $\HH^{n-1}$-a.e.~$y \in \xi^{\bot}$, for a.e.~$t \in \Om^{\xi}_{y}$.}
\end{displaymath}
In regards to the Riemmanian case, the existence of a curvilinear approximate symmetric gradient in the $GBD_F$-context can be used to prove that every one form $\omega \in GBD(\rm{M})$ admits an approximate symmetric gradient $e(\omega)(p)$ in the sense of~\eqref{e:apsymmanifold1} for a.e. $p \in \rm M$. We refer to Theorem \ref{t:apsymmanifold} for the relevant properties of $e(\omega)$.

\subsection*{Outlook.} In this paper we have introduced a new notion of functions with generalized bounded deformation by working on one-dimensional slices that are solutions of a second order ODE driven by a smooth field~$F$. In particular, we have shown that, with a suitable choice of~$F$, this leads to the definition of a space of functions with generalized bounded deformation on a Riemannian manifold~${\rm M}$. Possible applications of such space may be found in the modeling of brittle fracture in Riemannian setting (see, e.g.,~\cite{MR4216047, MR1757535, MR3574979}). For instance, Griffith's energy~\eqref{e:intro0} on $\rm M$ is well defined in terms of the space $GBD(\rm{M})$ and writes in the following form
\begin{equation*}
    \int_{\rm{M}} |e(\omega)|^2 \, d\mathcal{H}^n + \mathcal{H}^{n-1}(J_\omega),
\end{equation*}
where $J_\omega$ denotes the jump set of $\omega$ (see Definition \ref{d:jump-omega}).
%For thin structures, this amounts in rewriting~\eqref{e:intro0} in curvilinear coordinates, thus obtaining a new functional defined, in terms of the curvilinear symmetric gradient~$e(u)$ and on the jump set~$J_{u}$.

In the setting of dimension reduction problems for thin structures, such as the shallow shell~\cite{MR1151270, Mag-Mor} $S_{\rho} := \{ (x', \rho \theta(x')): x' \in U \}$ for $\rho>0$, $ U \subseteq \R^{2}$ open bounded set with Lipschitz boundary, and $\theta \in C^{\infty} (U)$, one is led to study the $\Gamma$-limit as $\rho\to 0$ of the energy functional
\begin{equation}
\label{e:shallow}
\int_{U \times (- 1, 1) } \C_{\rho} e_{\rho} (u) \cdot e_{\rho}(u) \,  \di x + \int_{J_{u}} \phi_{\rho} (\nu_{u}) \, \di \mathcal{H}^{n-1}\,.
\end{equation}
In \eqref{e:shallow}, $\C_{\rho}$ is a suitably rescaled elasticity tensor, $\phi_{\rho}$ is a positive definite quadratic form smoothly dependent on~$\rho$ and~$\theta$, and $e_{\rho}$ is of the form
\begin{align*}
& e_{\alpha\beta, \rho} (u) = \frac{1}{2} ( \partial_{\alpha} u_{\beta} + \partial_{\beta} u_{\alpha}) - u_{p} \Gamma^{p}_{\alpha\beta}(\rho)\,,\\
& e_{\alpha 3, \rho} (u) = \frac{1}{\rho} \Big( \frac{1}{2} ( \partial_{3} u_{\alpha} + \partial_{\alpha} u_{3}) - u_{\sigma} \Gamma^{\sigma}_{\alpha3} (\rho)\Big)\,,\\
& e_{33, \rho}(u) = \frac{1}{\rho^{2}} \partial_{3} u_{3}\,,
\end{align*}
where the (rescaled) Christoffel symbols $\Gamma^{p}_{ij}(\rho)$ behave as
\begin{align*}
&\Gamma^{\sigma}_{\alpha\beta} ( \rho) = O(\rho^{2})\,,\\
& \Gamma^{3}_{\alpha\beta} (\rho) = - \frac{ \partial^{2}_{\alpha\beta} \theta}{\sqrt{ 1 + \rho^{2} | \nabla \theta |^{2}}} + O(\rho)\,,\\
& \Gamma^{\sigma}_{\alpha3}(\rho) = O(\rho^{2})\,.
\end{align*}
As $\rho\to0$ we obtain the formal limit
\begin{equation*}
    \int_{U \times (-1, 1)} \hat{\C} e(u) {\, \cdot\,} \e(u) \, \di x + \int_{J_{u}} \phi(\nu_{u}) \, \di \mathcal{H}^{n-1}\,,
\end{equation*}
where $\hat{\C}$ and $\phi$ are limit of~$\C_{\rho}$ and of~$\phi_{\rho}$, respectively, while the displacement $u$ belongs to the space $GSBD_{F} ( U \times (-1, 1))$ for the choice
\begin{equation*}
F_{1}(x, \zeta) = F_{2} (x, \zeta) = 0 \,, \qquad  F_{3} (x, \zeta) = - \nabla^{2}\theta \, \zeta {\, \cdot\,} \zeta \qquad \text{for  $x \in U \times (-1, 1)$ and $\zeta \in \R^{3}$}.
\end{equation*}
Problems such as compactness and lower-semicontinuity in $GBD(\rm M)$ and in $GBD_{F} (\Om)$, the rigorous computation of the $\Gamma$-limit of~\eqref{e:shallow}, as well as the applications to more general reduction problems for brittle linearly elastic shells will be the subjects of future investigations.

\subsection*{Plan of the paper.}
In Sections~\ref{s:GBD-M} and \ref{s:definition} we give the main assumptions of the paper and present the notions of the  space $GBD({\rm M})$ of functions of generalized bounded deformation on a Riemannian manifold~${\rm M}$ and of $GBD_{F}(\Om)$, respectively. In particular, we show in Section~\ref{s:equivalence} that the two spaces are equivalent on every chart $(U, \psi)$ of~$M$. Section~\ref{s:curvilinear} contains a number of technical results that will be used in Section~\ref{s:structure} in the study of the structure of~$GBD_{F}(\Om)$ and of~$GBD({\rm M})$.

\section{Preliminaries and notation}

\subsection{Basic notation}
For $n , k \in \mathbb{N}$, we denote by~$\mathcal{L}^{n}$ and by~$\mathcal{H}^{k}$ the Lebesgue and the $k$-dimensional Hausdorff measure in~$\R^{n}$, respectively. The symbol $\mathbb{M}^{n \times n}$ stands for the space of square matrices of order~$n$ with real coefficients, while~$\mathbb{M}^{n \times n}_{sym}$ denotes its subspace of symmetric matrices. The set $\{e_{i}\}_{i=1}^{n}$ denotes the canonical basis of~$\R^{n}$ and $| \cdot|$ is the Euclidean norm on~$\R^{n}$. For every $\xi \in \R^{n}$, the map~$\pi_{\xi} \colon \R^{n} \to \R^{n}$ is the orthogonal projection over the hyperplane orthogonal to~$\xi$, which will be indicated by~$\xi^{\bot}$. For $x \in \R^{n}$ and~$\rho>0$, ${\rm B}_{\rho}(x)$ stands for the open ball of radius~$\rho$ and center~$x$ in~$\R^{n}$. For every $A \subseteq \R^{n} \times \mathbb{S}^{n-1}$, every $\xi \in \mathbb{S}^{n-1}$, and every $x \in \R^{n}$ we will denote 
$$ A_{\xi} := \{ (x \in \R^{n}: (x, \xi) \in A\} \qquad A_{x} := \{ \xi \in \mathbb{S}^{n-1}: (x, \xi) \in A\}\,.$$ \EEE

\RRR We report the definitions of countably rectifiable set in~$\R^{n}$.

\begin{definition}[Countably rectifiable set]
 We say that a set $R \subseteq \Omega$ is countably $(n-1)$-rectifiable if and only if $R$ equals a countable union of images of Lipschitz maps $(f_i)_i$ from some bounded sets $E_i \subset \mathbb{R}^{n-1}$ to $\Omega$. 
\end{definition} \EEE

Given~$U_{j}$ a sequence of open subsets of~$\R^{n}$, $\Omega$ open subset of~$\R^{n}$, and~$f_{j} \in C^{\infty}(U_{j}; \R^{k})$, we say that $f_{j} \to f$ in~$C^{\infty}_{loc} (\Om; \R^{k})$ if~$f \in C^{\infty}(\Om; \R^{k})$, $U_{j} \nearrow \Om$, and $f_{j} \to f$ in~$C^{\infty}(W; \R^{k})$ for every $W \Subset \Om$. \RRR We recall the definition of jump set of a measurable function.

\begin{definition}
Let $\Omega$ be an open subset of~$\R^{n}$ and let $u \colon \Omega \to \mathbb{R}^m$ be measurable. We say that $x \in \Omega$ belongs to $J_u$ if and only if there exists $(u^+(x),u^-(x),\nu(x)) \in \mathbb{R}^m \times \mathbb{R}^m \times \mathbb{S}^{n-1}$ such that
\[
\aplim_{\substack{z \to x \\ \pm(z-x) \cdot \nu (x) >0}} u(z)=u^\pm(x).
\]
\end{definition} 

We further recall the definition of approximate symmetric gradient of a measurable function.

\begin{definition}[Approximate symmetric gradient]
A measurable function $u \colon \Om \to \R^n$ admits an \emph{approximate symmetric gradient} at $x \in \Om$ if there exists $\tilde{e}(u)(x) \in \mathbb{M}^{n \times n}_{sym}$ such that
\begin{equation}
\label{e:apsym1.1}
    \aplim_{z \to x} \frac{|(u(z)-u(x)) \cdot (z  - x) -\tilde{e}(u)(x)(z - x)\cdot (z - x)|}{| z - x |^2} = 0.
\end{equation}
\end{definition}

Notice that the approximate symmetric gradient, if it exists, is unique by formula~\eqref{e:apsym1.1}.

 Given a metric space~$(X, d_{X})$,~$\mathcal{M}_{b}(X)$ (resp.~$\mathcal{M}_{b}^{+}(X)$) is the space of bounded Radon measures on~$X$ (resp. bounded and positive Radon measures on~$X$). Given $(Y, d_{Y})$ another metric space, a Borel map~$f\colon X \to Y$, and a measure~$\mu \in \mathcal{M}_{b}(X)$, the push-forward measure of~$\mu$ through~$f$ is denoted by~$f_{\sharp}(\mu) \in \mathcal{M}_{b}(Y)$. The set of all Borel subset of~$X$ is indicated by~$\mathcal{B}(X)$. For a Lipschitz function $f\colon X \to Y$, we denote by ${\rm Lip} (f;X)$ the least Lipschitz constant of~$f$ on~$X$, defined as
\begin{displaymath}
{\rm Lip} (f, X) := \inf_{x, y \in X, x \neq y} \, \frac{d_{Y} (f(x), f(y))}{d_{X} (x, y)}\,.
\end{displaymath} 
We will drop the dependence on the set whenever it is clear from the context.

We recall the definition of $(n-1)$-rectifiable measure in~$\R^{n}$.

 \begin{definition}[Rectifiable measure]
Let $\Om \subseteq \R^{n}$ and let $\mu$ be a measure on $\Omega$. We say that~$\mu$ is $(n-1)$-rectifiable if there exist an $(n-1)$-rectifiable set $R$ and a real-valued measurable function~$\theta$ such that
\begin{equation*}
%\label{e:intro8}
\mu = \theta \, \mathcal{H}^{n-1} \restr R\,.
\end{equation*}
 \end{definition}
 
\RRR For $U \in \mathcal{B}( \R^{n})$, \EEE for every $p \in [1, +\infty]$ the symbol $L^{p}(U; \R^{k})$ stands for the space of $p$-summable functions from~$U$ with values in~$\R^{k}$. The usual $L^{p}$-norm is denoted by~$\| \cdot\|_{L^{p}(U)}$. We will drop the set~$U$ in the notation of the norm when there is no chance of misunderstanding.

Given $({\rm M} ,g)$ an $n$-dimensional Riemannian manifold, we denote by $\text{exp}_p$ the exponential map centered at $p \in \RRR \rm M \EEE$.
  For $q \in {\rm M}$, we further denote by $d \,  {\rm exp}_{p}[q]$ the differential of the exponential map ${\rm exp}_{p}$ in the point~$q$, which maps ${\rm T}_{p}{\rm M}$ in ${\rm T}_{q}{\rm M}$. The symbol $\langle \cdot ,\cdot \rangle_{p}$ stands for the usual duality pairing between ${\rm T}^{*}_{p}{\rm M}$ and ${\rm T}_{p}{\rm M}$, while $(\cdot, \cdot)_{p}$ indicates the Riemannian scalar product in~${\rm T}_{p}{\rm M}$. The norm of a vector $v \in {\rm T}_{p}{\rm M}$ is denoted by $| v|_{p} = \sqrt{(v, v)_{p}}$. \RRR We will indicate by~${\rm inj}_{p}>0$ the injectivity radius of~$\text{exp}_p$, that is, the smallest~$r>0$ such that~$\text{exp}_p$ is injective on the set~$\{v \in {\rm T}_{p}{\rm M} : \ |v|_{p} <\rho \}$ for every $\rho \in (0, r)$. For every $p \in {\rm M}$ and every quadratic form $Q \in {\rm T}_{p}{\rm M} \otimes {\rm T}_{p}{\rm M}$ we define
  \begin{displaymath}
  \| Q\|_{{\rm T}_{p}{\rm M} \otimes {\rm T}_{p}{\rm M}} := \sup_{v \in {\rm T}_{p}{\rm M}} \frac{Q(v)}{ | v|_{p}^{2}}\,.
  \end{displaymath}
  \EEE

\subsection{Jump set of one-forms}
Let $({\rm M},g)$ be an $n$-dimensional Riemannian manifold. We define the jump set of a measurable one-form $\omega \in \mathcal{D}^1(\rm M)$ as follows.

\begin{definition}
\label{d:jump-omega}
Let $\omega \in \mathcal{D}^1(\rm M)$ measurable. We say that a point $p \in \rm M$ belongs to the jump set $J_\omega$ of $\omega$ if and only if there exists $\nu_\omega \in {\rm T}_p{\rm M}$ with $|\nu_\omega|_{p}=1$ and $\omega^\pm(p) \in {\rm T}^*_p{\rm M}$ with $\omega^+(p) \neq \omega^-(p)$ such that
\begin{equation}
\label{d:formjumpset}
    \aplim_{\substack{q \to p \\ q \in H^{\pm}(p)}} \, \langle\omega(q), d \, \text{exp}_p[q](v)\rangle_{q} = \langle\omega^{\pm}(p),v \rangle_{p} \ \ \text{ for every }v \in {\rm T}_p{\rm M}\, ,
\end{equation}
where $H^{\pm}(p) \cap \mathbb{B}_r(p) =\text{exp}_p(\{v \in {\rm T}_{p}{\rm M} : \ |v|_{p} <r \ \text{and} \  \pm (v ,  \nu_\omega(p))_{p} >0 \})$ for every $0 < r < \text{inj}_p$. 
\end{definition}

\subsection{Approximate symmetric gradient of one-forms}
Assume that $({\rm M} ,g)$ is an $n$-dimensional Riemannian manifold and consider a measurable one-form $\omega \in \mathcal{D}^1(\rm M)$. We want to define the \emph{approximate symmetric gradient} of $\omega$ at a point $p \in \rm M$ as a quadratic form acting on the tangent space ${\rm T}_p{\rm M}$, and satisfying a suitable first order expansion in a measure theoretical sense. Inspired by the euclidean notion of approximate symmetric gradient for measurable vector fields \cite{dal}, we give the following definition.
\begin{definition}[Approximate symmetric gradient of one-forms]
    Let $\omega \in \mathcal{D}^1(\rm M)$ be measurable. Then we say that $\omega$ admits an approximate symmetric gradient at $p \in \rm M$ if there exists a quadratic form $e(\omega)(p) \in {\rm T}_p{\rm M} \otimes {\rm T}_p{\rm M}$ such that
    \begin{equation}
        \label{e:apsymmanifold1}
        \aplim_{q \to p}\, \frac{|\langle\omega(q), d\,\text{exp}_p[q](v_q)\rangle_{q} - \langle \omega(p), v_q\rangle_{p} - e(\omega)(p)( v_q)|}{\text{d}_{\text{M}}(q,p)^2} =0\,,
    \end{equation}
    where $v_q \in {\rm T}_p{\rm M}$ is the unique vector satisfying $\text{exp}_p(v_q)=q$, namely, $v_q= \text{exp}^{-1}_p(q)$.
\end{definition}

\begin{remark}
 \RRR   The approximate limit in \eqref{e:apsymmanifold1} is well defined since $\exp_p^{-1}$ is a well defined map for every $q \in \mathbb{B}_r(p)$ for every $0 < r < \text{inj}_p$. \EEE
\end{remark}

\begin{remark}[Uniqueness of approximate symmetric gradient]
    The approximate symmetric gradient $e(\omega)(p)$ of a measurable one-form $\omega$ at $p$ is unique whenever it exists. This can be easily checked by a contradiction argument. Indeed, assuming the existence of two different approximate symmetric gradients $e^1(\omega)(p)$ and $e^2(\omega)(p)$ we obtain the validity of
    \begin{equation}
    \label{e:apsymmanifold2}
    \lim_{q \to p} \frac{|e^1(\omega)(p)(v_q) - e^2(\omega)(p)( v_q)|}{\text{d}_{\text{M}}(q,p)^2} =0.
    \end{equation}
    The approximate limit in \eqref{e:apsymmanifold2} can be rewritten as
    \[
    \lim_{v \to 0} \frac{|e^1(\omega)(p)(v) - e^2(\omega)(p)(v)|}{\text{d}_{\text{M}}(\exp_p(v),p)^2} =\lim_{v \to 0} \frac{|e^1(\omega)(p)(v) - e^2(\omega)(p)(v)|}{\RRR | v|_{p}^{2} \EEE }=0,
    \]
    which clearly implies $e^1(\omega)(p)= e^2(\omega)(p)$ as element of ${\rm T}_p{\rm M} \otimes {\rm T}_p{\rm M}$. 
\end{remark}

\subsection{A rectifiability criterion for a class of integralgeometric measures}

%This section is based on the techniques developed in \cite{Tas22}. 

%\begin{definition}[Pushforward measure]
%Given a Radon measure $\mu$ in $\mathbb{R}^n$ and a Borel map $P \colon \mathbb{R}^n \to \mathbb{R}^m$ we define the pushforward $P_{\sharp}\mu$ as the Borel regular (outer) measure on $\mathbb{R}^m$ defined as
%\begin{equation*}
%    P_\sharp\mu(E) := \inf_{\substack{E \subset B \\ B \text{ Borel}}} \mu(P^{-1}(B)).
%\end{equation*}
%
%\end{definition}

The notion of \emph{transversal family of maps} will play a fundamental role along this section. The following definition is an adaptation of \cite[Definition 2.4]{hov} \RRR (see also~\cite[Definition~2.3]{AT_22-preprint}). \EEE 

\begin{definition}[Transversality]
\label{d:transversal}
Let $\Omega \subseteq \mathbb{R}^n$ be open and let $S_i:=\{\xi \in \mathbb{S}^{n-1}  :  |\xi\cdot e_i| \geq 1/\sqrt{n}  \}$ for $i=1,\dotsc,n$. We say that a family of Lipschitz maps $P_\xi \colon \Omega \to \xi^\bot$ for $\xi \in \mathbb{S}^{n-1}$ is a transversal family of maps on~$\Omega$ if for every $i=1,\dotsc,n$ the maps
\begin{align*}
P^i_\xi(x) & := \pi_{e_i}\circ P_\xi(x) \qquad \text{for } \xi \in S_i, \  x \in  \Omega\,,
\\
 T^{i}_{xx'}(\xi) & := \frac{P^i_\xi(x) - P^i_\xi(x')}{|x-x'|} \qquad \text{for } \xi \in S_{i}, \ x, x' \in \Om \text{ with $x \neq x'$}
\end{align*}
satisfy the following properties:
\begin{enumerate}[label=(H.\arabic*),ref=H.\arabic*]
    \item For every $x \in \Omega$ the map $\xi \mapsto P^i_\xi(x)$ belongs to $C^2(S_i;\mathbb{R}^{n-1})$ and
    \begin{equation*}
    \label{e:h1}
    \sup_{(\xi,x) \in S_i \times \Omega} |D^j_\xi P^i_\xi(x)| < \infty, \ \ \text{for }j=1,2\, ;
    \end{equation*}
    \item \label{hp:H2} There exists a constant $C' >0$ such that for every $\xi \in S_i$ and $x,x' \in \Omega$ with $x \neq x'$ 
    \begin{equation*}
    \label{e:h2}
        |T^{i}_{xx'}(\xi)| \leq C' \ \ \ \text{ implies } \ \ \
        |\text{J}_\xi T^{i}_{xx'}(\xi)| \geq C';
    \end{equation*}
   \item \label{hp:H3} There exists a constant $C'' >0$ such that 
   \begin{equation*}
   \label{e:h3}
       | D^j_\xi T^{i}_{xx'}(\xi) | \leq C'',\ \ \text{for }j=1,2\,
   \end{equation*}
   for $\xi \in S_i$ and $x,x' \in \Omega$ with $x \neq x'$.
\end{enumerate}
\end{definition}

\section{The space $GBD$ on a Riemannian manifold}
\label{s:GBD-M}

In this section we define the space of generalized functions of bounded deformation on a Riemannian manifold~$({\rm M}, g)$ of dimension $n$. To this aim, we first introduce the notions of parametrized maps and curvilinear projections on~$\rm M$, following the ideas of~\cite{AT_22-preprint} in the Euclidean setting (see also Section~\ref{s:curvilinear}).

%\begin{definition}[Transversality on~$M$]
%Let $V \subseteq M$ open. We say that a family $P_{\xi} \colon V \to \xi^{\bot}$, $\xi \in \mathbb{S}^{n-1}$, is a transversal family of maps on~$V$ if for every chart $(U, \psi)$ on~$M$ with $U \subseteq V$ we have that the family $\{ P_{\xi} \circ \psi^{-1}\}_{\xi \in \mathbb{S}^{n-1}}$ is a transversal family of maps on~$\psi(U)$.
%\end{definition}

\begin{definition}[Parametrized maps on~$\rm M$]
Let $V \subseteq \rm M$ open and~$\xi \in \mathbb{S}^{n-1}$. We say that a map $P \colon V \to \xi^{\bot}$ is a parametrized map on~$V$ if there exist $\rho, \tau>0$ and a smooth Lipschitz map~$\varphi \colon \{y + t\xi: (y, t) \in  [\xi^{\bot} \cap {\rm B}_{\rho} (0)] \times (-\tau, \tau)\} \to \rm M$ such that the following conditions hold:
\begin{enumerate}[label=(\arabic*), ref=(\arabic*)]
    \item $V \subseteq \text{Im}(\varphi)$;
    \item $\varphi^{-1}  \restr V$ is a bi-Lipschitz diffeomorphism with its image;
    \item $ P(\varphi ( y +  t\xi )) = y$ for every $(y, t) \in [\xi^{\bot} \cap \mathrm{B}_{\rho}(0)] \times  (-\tau, \tau)$ such that $y + t\xi \in  \varphi^{-1} (V)$.
    % for every $t \in (-\tau,\tau)$, $y \in \xi^\bot \cap B_\rho(0)$, and $\xi \in \mathbb{S}^{n-1}$;
\end{enumerate}
\end{definition}

\begin{remark}
 Conditions (2) and (3) of parametrized map imply
\begin{enumerate}[label=(\arabic*), ref=(\arabic*), ]
\setcounter{enumi}{3}
    \item
    $ \varphi ( P (p) +  (\varphi^{-1}(p) \cdot \xi) \xi \big) = p$ for every $p \in V$.
\end{enumerate}
 We will more compactly denote by~$t^\xi_p$ the real number $ \varphi^{-1}(p) \cdot \xi$ for every $p \in V$ and every $\xi \in \mathbb{S}^{n-1}$. Whenever~$\xi$ is fixed and there is no chance of misunderstanding, we drop the index~$\xi$ and write~$t_p$ instead of~$t^\xi_p$. 
\end{remark}

\RRR
\begin{remark}
\label{r:M1}
Let $({\rm M}, g)$ be a Riemannian manifold of dimension~$n$, let $V\subseteq \rm M$ open, and let $P \colon V \to \xi^{\bot}$ be a parametrized map on~$V$ with parametrization~$\varphi \colon \{y + t\xi: (y, t) \in  [\xi^{\bot} \cap {\rm B}_{\rho} (0)] \times (-\tau, \tau)\} \to \rm M$. Then, for every chart~$(U, \psi)$ on~$\rm M$ with~$U \subseteq V$ we have that $\overline{P} := P \circ \psi^{-1}\colon \psi(U) \to \xi^{\bot}$ is a parametrized map on~$\psi(U) \subseteq \R^{n}$ with parametrization $$\overline{\varphi} := \psi\circ \varphi \colon  \{y + t\xi: (y, t) \in  [\xi^{\bot} \cap {\rm B}_{\rho} (0)] \times (-\tau, \tau)\} \to \R^{n}\,.$$
\end{remark}
\EEE

\begin{definition}[Velocity field on~$\rm M$]
Let $V \subseteq \rm M$ open,~$\xi \in \mathbb{S}^{n-1}$, let $P \colon V \to \xi^{\bot}$ be a parametrized map on~$V$ with parametrizationi~$\varphi \colon \{y + t\xi: (y, t) \in  [\xi^{\bot} \cap {\rm B}_{\rho} (0)] \times (-\tau, \tau)\} \to \rm M$. For $(y, t) \in [\xi^{\bot} \cap {\rm B}_{\rho} (0)] \times (-\tau, \tau)$, we denote by $\dot{\varphi} (y + t\xi) \in {\rm T}_{\varphi(y + t\xi)} \rm M $ the velocity field of the curve $t \mapsto \varphi(y + t\xi)$.
\end{definition}

\begin{definition}[Curvilinear projections on~$\rm M$]
\label{d:CP_M}
Let $V \subseteq \rm M$ open and~$\xi \in \mathbb{S}^{n-1}$. We say that a map $P \colon V \to \xi^{\bot}$ is a curvilinear projection on~$V$ if the following conditions hold:
\begin{enumerate}[label=(\arabic*), ref=(\arabic*)]
\item $P$ is parametrized on~$V$ with parametrization $\varphi \colon \{y + t\xi: (y, t) \in  [\xi^{\bot} \cap {\rm B}_{\rho} (0)] \times (-\tau, \tau)\} \to \rm M$;
\item the parametrization $\varphi$ is such that for every $y \in  [\xi^{\bot} \cap {\rm B}_{\rho} (0)]$ the curve $t \mapsto \varphi(y + t\xi)$ is a geodesic on~$\rm M$.
\end{enumerate}
\end{definition}

\RRR
\begin{remark}
For $P\colon V \to \xi^{\bot}$ and~$\varphi \colon \{y + t\xi: (y, t) \in  [\xi^{\bot} \cap {\rm B}_{\rho} (0)] \times (-\tau, \tau)\} \to \rm M$ as in Definition~\ref{d:CP_M}, we define
\begin{displaymath}
\| \dot\varphi\|_{L^{\infty} , \rm M} := \sup_{(y, t) \in  [\xi^{\bot} \cap {\rm B}_{\rho} (0)] \times (-\tau, \tau)} \, | \dot{\varphi} (y + t\xi)|_{\varphi(y + t\xi)}\,.
\end{displaymath}
Notice that in the notation of $\| \dot\varphi\|_{L^{\infty} , \rm M}$ we will never drop the index~$\rm M$.
\end{remark}
\EEE

Given $V \subseteq \rm M$ open, a parametrized map $P\colon V \to \xi^{\bot}$, and \RRR ~$\omega \in \mathcal{D}^{1} (V)$ measurable, \EEE we define the slices of~$u$ w.r.t.~$P$.

\begin{definition}[Slices on~$\rm M$]
Let $({\rm M} , g)$ be a Riemannian manifold of dimension~$n$, let~$V \subseteq \rm M$ open,~$\xi \in \mathbb{S}^{n-1}$, and let~$P \colon V \to \xi^{\bot}$ be a curvilinear projection on~$V$ with parametrization~$\varphi \colon \{ y + t\xi: (y, t) \in [\xi^{\bot} \cap {\rm B}_{\rho} (0)] \times (-\tau, \tau)\}$. For every $B \in \mathcal{B} (V)$ we define
\begin{displaymath}
B^{\xi}_{y} := \{ t \in (-\tau, \tau) : \, \varphi(y + t\xi) \in B\} \qquad \text{for $y \in [\xi^{\bot} \cap {\rm B}_{\rho} (0)]$}.
\end{displaymath}
For every $\omega \in \mathcal{D}^{1}(V)$ we define 
\begin{displaymath}
\hat{\omega}^{\xi}_{y} (t) := \left\langle \omega (\varphi (y + t\xi)) , \dot\varphi(y + t\xi)\right\rangle_{\varphi (y + t\xi)} \qquad \text{for $t \in V^{\xi}_{y}$.}
\end{displaymath}
\RRR In addition for every $u \in \Gamma(V)$ we define  
\begin{displaymath}
\hat{u}^{\xi}_{y} (t) :=  (u (\varphi (y + t\xi)) ,  \dot\varphi(y + t\xi)) _{\varphi(y + t\xi)}  \qquad \text{for $t \in V^{\xi}_{y}$.}
\end{displaymath}\EEE
\end{definition}

\begin{definition}
\label{d:GBD_M}
Let $({\rm M}, g)$ be a Riemannian manifold of dimension $n$. We say that $u \in \mathcal{D}^{1}( {\rm M})$ has {\em generalized bounded deformation} on~$\rm M$, and we write $\omega \in GBD({\rm M})$, if there exists $\lambda \in \mathcal{M}^{+}_{b} ({\rm M})$ such that for every $V \subseteq \rm M$ open, every $\xi \in \mathbb{S}^{n-1}$, and every curvilinear projection $P\colon V \to \xi^{\bot}$ on~$V$, the following facts hold:
\begin{enumerate}
\item for $\mathcal{H}^{n-1}$-a.e.~$y \in \xi^{\bot}$ the map $\hat{\omega}^{\xi}_{y}$ belongs to $BV_{loc}( V^{\xi}_{y})$;

\item for every $B \in \mathcal{B} ({\rm M})$ we have that
\begin{displaymath}
\int_{\xi^{\bot}}  \big(\big| | {\rm D} \hat{\omega}^{\xi}_{y} | (  B^{\xi}_{y} \setminus J^{1}_{\hat{\omega}^{\xi}_{y}})  + \mathcal{H}^{0} (  B^{\xi}_{y} \cap J^{1}_{\hat{\omega}^{\xi}_{y}} ) \big) \, \di \mathcal{H}^{n-1} (y) \leq  \RRR \|  \dot{\varphi} \|^{2}_{L^{\infty}, \rm M} \EEE\,  {\rm Lip}(P ; V)^{n-1} \lambda(B)\,.
\end{displaymath}
\end{enumerate}
\end{definition}

\begin{remark}
Since in Definition~\ref{d:GBD_M} the map $t\mapsto \varphi(y + t\xi)$ is a geodesic on~$\rm M$. Thus, for $y \in [\xi^{\bot} \cap {\rm B}_{\rho} (0)]$ the  speed modulus $| \dot{\varphi} (y + t\xi) |_{\varphi(y + t\xi)}$ is constant as a function of $t$. Hence, the $L^{\infty}$-norm in item~$(2)$ of the definition is computed as a supremum w.r.t.~$y$. 
\end{remark}

In order to study the structure of the space $GBD(\rm M)$, we show in the next section that, locally on charts of~$\rm M$, it is equivalent to the space $GBD_{F} (\Om)$ for a suitable open subset~$\Om$ of~$\R^{n}$ and field $F \colon \Om \times \R^{n} \to \R^{n}$.

\section{The space $GBD_{F}(\Om)$}
\label{s:definition}

We start by recalling the definition of $GBD_{F} (\Om)$ for an open set $\Om \subseteq \R^{n}$ and a  field~$F \in C^{\infty} (\R^{n} \times \R^{n}  ;\R^{n})$. First, we list the assumptions on~$F$ (see also~\cite[Sections~3.1 and~6.1]{AT_22-preprint}).

 \subsection{Assumptions on the field $F$}
\RRR We will always assume \EEE that $F \in C^{\infty} (\R^{n} \times \R^{n}; \R^{n})$ fulfills 
 \begin{enumerate}[label=(F.1),ref=F.1]
 \item \label{hp:F}  $F$ is a quadratic form in the second variable, that is, for every $x \in \R^{n}$ and every $v_1,v_2 \in \mathbb{R}^n$
\begin{equation}
    \label{e:quadratic}
    F(x,v_1 + v_2) + F(x,v_1 - v_2) = 2 F(x,v_1) +2 F(x,v_2)\,.
\end{equation}
\end{enumerate}

In Section~\ref{s:jump-u} we will require an additional property on~$F$, namely the so-called {\em Rigid Interpolation} property (cf.~\cite[Section~6.1]{AT_22-preprint})), which needs some further notation. Let $\{e_1,\dotsc,e_n\}$ be the canonical basis of $\mathbb{R}^n$. Thanks to~\eqref{e:quadratic} we associate to~$F$ a map $F^q \colon \mathbb{R}^n \to \text{Lin}(\mathbb{R}^n \otimes \mathbb{R}^n \otimes \mathbb{R}^n;\mathbb{R})$ as follows:
 \begin{equation*}
     \label{e:rip3}
     F^q(x)(v_1 \otimes v_2 \otimes v_3)  := \frac{v_3}{2} \cdot (F(x,v_1+v_2)  -F(x,v_1)  -F(x,v_2) ) \qquad  v_1,v_2,v_3 \in \mathbb{R}^n.
 \end{equation*}
 It is worth noting that, under our hypothesis~\eqref{hp:F}, for every $v_3 \in \mathbb{R}^n$ the map $(v_1,v_2) \mapsto F^q(x)(v_1 \otimes v_2 \otimes v_3)$ is symmetric and hence can be represented as an element of $\mathbb{M}^{n \times n}_{sym}$. For this reason we can write
 \begin{equation*}
     \label{e:rip6}
     F^q(x)(v_1 \otimes v_2 \otimes v_3)= (v_3 \cdot F^q(x))v_1 \cdot v_2  \qquad \text{for $v_1,v_2,v_3 \in \mathbb{R}^n$,}
     \end{equation*}
     for a suitable $(v_3 \cdot F^q(x))\in \mathbb{M}^{n \times n}_{sym}$ depending on~$v_3$. Given $r>0$ and a point $x \in \mathbb{R}^n$ we define $F_{r, x} \colon \mathbb{R}^n \times \mathbb{R}^n \to \mathbb{R}^n$ as $F_{r, x}(z,v):= r F ( x + r z , v )$ and analogously $F^q_{r, x} \colon \mathbb{R}^n \to \text{Lin}(\mathbb{R}^n \otimes \mathbb{R}^n \otimes \mathbb{R}^n; \mathbb{R})$ as $F^q_{r, x}(z):= rF^q(x+rz)$. 

%At this point it is convenient to introduce some notation. 
 For $z \in {\rm B}_1(0)$, we set $\mathcal{S}_{0,z}:= \{z+e_0, \dotsc,z+e_n  \}$, where~$e_0:=0$. For $r>0$ and $0 \leq i < j \leq n$  we define $t \mapsto \ell_{z, r, ij}(t)$ as the curve $\gamma(\cdot)$ (whenever it is well defined) satisfying 
\begin{equation*}
\label{e:poincare15000}
    \begin{cases}
    \ddot{\gamma}(t) = F_{r, x}(\gamma(t),\dot{\gamma}(t)), \ t \in [0,t_{ij}], \ \text{for some }t_{ij}>0 & \\
    \gamma(0)=z+e_i, \ \gamma(t_{ij})=z+e_j &\\
     |\dot{\gamma}(0)|= 1.  &
    \end{cases}
\end{equation*}

\begin{remark}
Notice that, as shown in~\cite[Lemma~3.13 and Remark~6.1]{AT_22-preprint}, for $r>0$ sufficiently small the curve $t \mapsto \ell_{z, r, ij}(t)$ is well-defined for every $z \in {\rm B}_{1}(0)$ and every $0 \leq i < j \leq n$.
\end{remark}

We denote by~$\mathcal{S}_{r,1,z}$ the 1-dimensional geodesic skeleton of~$\mathcal{S}_{0,z}$, i.e.,  
\[
\mathcal{S}_{r,1,z} := \{ h \in \mathbb{R}^n : \, h = \ell_{z, r, ij}(t) \ \text{for some }t \in [ 0,t_{ij}] \text{ and } \ i \neq j \}.
\]
For $0 \leq i <j \leq n$ we further set 
\[
\xi_{r,ij}(z):= \dot{\ell}_{z, r,ij}(0) \qquad \text{ and } \qquad \xi_{r,ji}(z):= \dot{\ell}_{z, r,ij}(t_{ij}) \,.
\]
 We consider the semi-norm $E_{r, z} \colon \mathbb{R}^{n+1} \times \mathbb{R}^n \to [0, + \infty) $ defined as
\[
E_{r, z} (w) := \sum_{0 \leq i < j \leq n} |w^j \cdot \xi_{r,ji}(z) - w^i \cdot \xi_{r,ij}(z)|  \qquad \text{for $w \in \mathbb{R}^{(n+1) \times n}$,}
\]
where $w^{i}$ denotes the $i$-th column of the matrix~$w$.  Eventually, we denote by $\mathcal{S}_{n,z}$ the convex hull of $\mathcal{S}_{0,z}$. Observing that every~$z$ belonging to $\{ z \in \mathrm{B}_1(0) : \, z \cdot e_i < 0, \ i=1,\dotsc,n  \}$  satisfies $\mathcal{S}_{n,z}\subset {\rm B}_1(0)$ and that
\[
\mathcal{L}^n(\{ z \in \mathrm{B}_1(0) : \, z \cdot e_i < 0, \ i=1,\dotsc,n  \})=\frac{\omega_n}{2^n} \,,
\]
we infer the existence of a dimensional constant $0<\rho(n) \leq 1$ such that $2^{n+1}\mathcal{L}^n(Q(n)) \geq \omega_n$ whenever
\begin{equation*}
\label{e:rip1}
Q(n):= \{z \in \mathrm{B}_1(0) :\, \mathrm{B}_{\rho(n)}(0) \subset \mathring{\mathcal{S}}_{n,z} \subset \mathcal{S}_{n,z} \subset \mathrm{B}_1(0) \} \,.
\end{equation*}

With the above notation at hand, the \RRR Rigid \EEE Interpolation property reads as follows:
 \begin{enumerate}[label=(RI),ref=RI]
   \item \label{hp:F2}  Given $x \in \mathbb{R}^n$ there exists a radius $r_x>0$ such that for every $z \in Q(n)$, every $w \in \mathbb{R}^{(n+1)  \times n}$, and every $0 < r \leq r_x$, we find a smooth map $a_r \colon \mathrm{B}_1(0) \to \mathbb{R}^n$ such that 
 \begin{align}
    \label{e:rip4}
    &  a_r(h)=w^i  \qquad \text{ for every $h \in \mathcal{S}_{0,z}$,} \\
    \label{e:rip5}
    &\|\tilde{e}(a_r) - a_r \cdot F^q_{r , x} \|_{L^{\infty}(\mathcal{S}_{n,z}; \mathbb{M}^{n \times n}_{sym})} \leq c(n) E_{r, z} (w)\,,
\end{align}
where $c(n)>0$ is a dimensional constant and where $\tilde{e}(a_r)$ denotes the symmetric gradient of $a_r$. 
\end{enumerate}

\subsection{Curvilinear projections on~$\Om$.}
\label{sub:curvilinear}

In this section we recall the definitions of (families of) curvilinear projections on~$\Om$ w.r.t.~a field $F \in C^{\infty} (\R^{n} \times \R^{n}; \R^{n})$ satisfying~\eqref{hp:F}. (cf.~Definitions~\ref{d:param} and~\ref{d:CP} and \cite[Section~3.1]{AT_22-preprint}). We refer to Section~\ref{s:curvilinear} for some technical properties of a specific family of curvilinear projections, which will be used in Section~\ref{s:approximate-sym-gradient}.

\begin{definition}[Velocity field]
\label{d:velocity-field}
Let~$\Om$ be a bounded open subset of~$\R^{n}$, $\xi \in \mathbb{S}^{n-1}$, and let $P \colon \Om \to \xi^{\bot}$ be a map parametrized by~$\varphi$ on~$\Om$. For every $x \in \Omega$ we define the {\rm velocity field} 
\begin{equation*}
\xi_{\varphi}(x):= \dot{\varphi}(P (x) +t_x \xi).
\end{equation*}
\end{definition}

%\begin{definition}[Geodesic maps]
%\label{d:geodesics-map}
%Let~$\Om$ be an open subset of~$\R^{n}$, $\xi \in \mathbb{S}^{n-1}$, and $P \colon \Omega \to \xi^{\bot}$ be a map parametrized by $\varphi \colon \{ y + t\xi: (y, t) \in [\xi^{\bot} \cap \mathrm{B}_{\rho}(0)] \times  (-\tau, \tau)\}  \to \R^{n}$. We say that $P$ is {\em geodesics} with respect to $F$ if
%\begin{equation*}
%    %\label{e:geodesics}
%    \ddot{\varphi} (y + t\xi) = F(\varphi(y + t\xi) ,\dot{\varphi}(y + t\xi) ) \qquad \text{for every  $(y, t) \in [\xi^{\bot} \cap {\rm B}_\rho(0)] \times (-\tau,\tau)$}\,.
%\end{equation*}
%\end{definition}

\begin{definition}[Curvilinear projections on~$\Om$ w.r.t.~$F$ ({\cite[Definition~3.5]{AT_22-preprint}})]
\label{d:CP-maps}
Let $\Omega$ be an open subset of~$\mathbb{R}^n$ and $\xi \in \mathbb{S}^{n-1}$. We say that a smooth Lipschitz map $P \colon \Omega \to \xi^{\bot}$  is a \emph{curvilinear projection} (with respect to $F$) on~$\Om$ if the following holds:
\begin{enumerate}
\item $P$ is parametrized on~$\Om$ by $\varphi \colon \{ y + t\xi: (y, t) \in [\xi^{\bot} \cap \mathrm{B}_{\rho}(0)] \times  (-\tau, \tau)\}  \to \R^{n}$;
\item\label{CP:item-2} For every $(y, t) \in [\xi^{\bot} \cap {\rm B}_\rho(0)] \times (-\tau,\tau)$ we have
\begin{equation*}
    %\label{e:geodesics}
    \ddot{\varphi} (y + t\xi) = F(\varphi(y + t\xi) ,\dot{\varphi}(y + t\xi) )\,.
    \end{equation*}
\end{enumerate}
\end{definition}

\begin{remark}
\label{r:M2}
Let $({\rm M}, g)$ be a Riemannian manifold of dimension~$n$, let~$V\subseteq \rm M$ open, and let $P \colon V \to \xi^{\bot}$ be a curvilinear projection on~$V$ with parametrization $\varphi \colon \{y + t\xi: (y, t) \in  [\xi^{\bot} \cap {\rm B}_{\rho} (0)] \times (-\tau, \tau)\} \to \rm M$. Similar to Remark~\ref{r:M1}, we notice for every chart~$(U, \psi)$ on~$\rm M$ with~$U \subseteq V$ we have that, setting $\overline{P} := P \circ \psi^{-1}$ and $\overline{\varphi}:= \psi \circ \varphi$, the map~$\overline{P}$ is a curvilinear projection on~$\psi(U)$ with respect to the field~$F$
\begin{equation}
\label{e:chris}
F(x, v)  := - \bigg( \sum_{i,j=1}^n\Gamma^1_{ij}(x)v_iv_j,\dotsc,\sum_{i,j=1}^n\Gamma^n_{ij}(x)v_iv_j \bigg)  \qquad (x,v) \in \psi(U) \times \mathbb{R}^n\,,
\end{equation}
where $\Gamma^{\ell}_{ij}$ denote the Christoffel symbols on~$\rm M$ induced by the chart~$(U,\psi)$. Indeed, for every $y \in [\xi^{\bot} \cap {\rm B}_{\rho}(0)]$ the curve $t \mapsto \overline{\varphi} (y + t\xi)$ solves the ODE $\ddot{\overline{\varphi}} = F(\overline{\varphi}, \dot{\overline{\varphi}})$, since $t\mapsto \varphi(y+t\xi)$ is a geodesic on~$\rm M$.

We notice that the viceversa is also true: for every open subset~$\Om$ of~$\psi(U)$ and every curvilinear projection~$\overline{P}\colon \Om \to \xi^{\bot}$ on~$\Om$ with respect to the field~$F$ and with parametrization~$\overline{\varphi}$, we have that $P:= \overline{P} \circ \psi \colon \psi^{-1} (\Om) \to \xi^{\bot}$ is a curvilinear projection on~$\psi^{-1} (\Om)\subseteq \rm M$ with parametrization~$\varphi:= \psi^{-1} \circ \overline{\varphi}$.
\end{remark}

%\RRR S: I do not understand whether we need any transversality or regularity of $\xi_{\varphi}$ at this point. Probably $\xi_{\varphi} \neq 0$? \EEE

%In our analysis, we will further need the following notions of parametrized family and of family of curvilinear projections.

\begin{definition}[Parametrized family on~$\Om$ ({\cite[Definition~3.6]{AT_22-preprint}})]
\label{d:param}
Let $\Omega$ be an open subset of~$\mathbb{R}^n$. We say that a family $P_\xi \colon \Omega \to \xi^\bot$ for $\xi \in \mathbb{S}^{n-1}$ is {\em parametrized} on~$\Omega$ if and only if there exist $\rho,\tau>0$, an open subset~$A$ of~$\R^{n} \times \mathbb{S}^{n-1}$, and a smooth Lipschitz map $\varphi \colon  A \to \R^{n}$  such that
\begin{enumerate}[label=(\arabic*), ref=(\arabic*)]
    \item for every $\xi \in \mathbb{S}^{n-1}$ we have $A_{\xi} = \{ y + t\xi: (y, t) \in   [\xi^\bot \cap {\rm B}_\rho(0)] \times (-\tau,\tau) \} $;
    \item for every $\xi \in \mathbb{S}^{n-1}$, $P_{\xi}$ is parametrized on~$\Om$ by the map $\varphi_{\xi} := \varphi (\cdot, \xi) \colon A_{\xi} \to \R^{n}$. 
    %we have $\Omega \subseteq \text{Im}(\varphi_\xi )$;
   % \item for every $\xi \in \mathbb{S}^{n-1}$,  $\varphi_{\xi}^{-1} \restr \Om$ is a bi-Lipschitz diffeomorphism with its image;
    %\item $P_\xi(\varphi_\xi (y+t\xi))=y$ for every $\xi \in \mathbb{S}^{n-1}$ and every $y + t\xi \in A_{\xi} \cap \varphi_{\xi}^{-1}(\Om)$;
    % for every $t \in (-\tau,\tau)$, $y \in \xi^\bot \cap B_\rho(0)$, and $\xi \in \mathbb{S}^{n-1}$;
    %\item $\varphi_\xi(P_\xi(x)+(\varphi^{-1}_\xi(x) \cdot \xi)\xi) =x$ for every $(x,\xi) \in \Omega \times \mathbb{S}^{n-1}$.
\end{enumerate}
\end{definition}

%\RRR S: Somewhere we need some measurability of~$\xi \mapsto \xi(x)$ (see Proposition~\ref{p:r=rxi}). \EEE

%\begin{definition}[Geodesic family]
%\label{d:geodesics}
%Let~$\Om$ be an open subset of~$\R^{n}$ and let $P_\xi \colon \Omega \to \xi^\bot$ for $\xi \in \mathbb{S}^{n-1}$ be parametrized by $\varphi \colon  A \to \R^{n}$. We say that $(P_\xi)_{\xi \in \mathbb{S}^{n-1}}$ is {\em geodesics} with respect to $F$ if~$P_{\xi}$ is geodesics with respect to~$F$ for every $\xi \in \mathbb{S}^{n-1}$.
%%\begin{equation*}
%%    %\label{e:geodesics}
%%    \ddot{\varphi}_\xi(y + t\xi)= F(\varphi_\xi(y+t\xi),\dot{\varphi}_\xi(y +t\xi)) \qquad \text{for every $(y,  t) \in [\xi^\bot \cap B_\rho(0)] \times  (-\tau,\tau)$}\,,
%%\end{equation*}
%%where $\tau, \rho>0$ are as in Definition~\ref{d:param}.
%\end{definition}

We also give the definition of \RRR family \EEE of curvilinear projections.

\begin{definition}[Family of curvilinear projections on~$\Om$ ({\cite[Definition~3.7]{AT_22-preprint}})]
\label{d:CP}
Let $\Omega$ be an open subset of~$\mathbb{R}^n$. We say that a family of maps $P_\xi \colon \Omega \to \xi^\bot$ for $\xi \in \mathbb{S}^{n-1}$ is a family of \emph{curvilinear projections} on~$\Om$ if the following conditions hold:
\begin{enumerate}[label=(\arabic*), ref=(\arabic*)]
\item  the family $(P_{\xi})_{\xi \in \mathbb{S}^{n-1}}$ is parametrized by $\varphi \colon  A \to \R^{n}$; 
\item  for every $\xi \in \mathbb{S}^{n-1}$, $P_\xi$ is a curvilinear projection on~$\Om$ with parametrization $\varphi_{\xi} = \varphi(\cdot, \xi)$; 
\item $(P_\xi)_{\xi \in \mathbb{S}^{n-1}}$ is a transversal family of maps on~$\Om$;
\item\label{d:CP-4} for every $x \in \Om$, the map $\xi \mapsto \xi_{\varphi}(x)/|\xi_{\varphi}(x)|$ is a diffeomorphism from $\mathbb{S}^{n-1}$ onto itself.
\end{enumerate}
\end{definition}

We conclude by recalling the definition of slices of a measurable function~$u \colon \Om \to \R^{n}$ w.r.t.~a curvilinear projection~$P \colon \Om \to \xi^{\bot}$ on~$\Om$.

\begin{definition}[Slices]
\label{d:slices}
Let $\Om$ be an open subset of~$\R^{n}$, $\xi \in \mathbb{S}^{n-1}$ and let $P \colon \Om \to \xi^{\bot}$ be a curvilinear projection on~$\Om$ parametrized by $\varphi \colon \{ y + t\xi: (y, t) \in [\xi^{\bot} \cap {\rm B}_{\rho} (0)] \times (-\tau, \tau)\} \to \R^{n}$. For every measurable function $u \colon \Om \to \R^{m}$, we define $u_\xi \colon \Omega \to \mathbb{R}$ by
\begin{equation*}
u_\xi(x) := \hat{u}^\xi_{P (x)}(t_x)=\hat{u}^\xi_{P (x)}(\varphi^{-1}(x) \cdot \xi),
\end{equation*}
and we notice the following identity
\begin{equation}
    \label{e:sliceide}
    u_\xi(\varphi ( y + t \xi ) ) = \hat{u}^\xi_y(t) \qquad  \text{for }\xi \in \mathbb{S}^{n-1} \text{ and } (y,t) \in [\xi^\bot \cap \mathrm{B}_\rho(0)] \times (-\tau,\tau).
\end{equation}
For a measurable function $v \colon \Om \to \R^{m}$ we also set $v^{\xi}_{y} (t) := v(\varphi (y + t\xi))$ for $t \in \Om^{\xi}_{y}$. 
%In addition, given a $C^1$-regular curve $\gamma \colon [a,b] \to \mathbb{R}^m$ we define $\hat{u}_\gamma \colon [a,b] \to \mathbb{R}$ as
%\begin{equation*}
%\hat{u}_\gamma(t) := u(\gamma(t)) \cdot \dot{\gamma}(t) \,.
%\end{equation*}
Eventually, in order to simplify the notation, if $\varphi$ is the identity we use the notation
\begin{align*}
& \widetilde{B}^{\xi}_{y} := \{ t \in \R : \, y + t\xi \in B\}\,,\\
& \widetilde{u}^{\xi}_{y}(t):= u(y + t\xi) \cdot \xi \qquad \text{for $t \in \widetilde{\Om}^{\xi}_{y}\,.$}
\end{align*}
\end{definition}

\subsection{Definition of~$GBD_{F} (\Om)$}

We recall here the definition of the space~$GBD_{F}(\Om)$ introduced in~\cite[Definition~6.4]{AT_22-preprint}.

\begin{definition}[The space $GBD_{F}(\Om)$]
\label{d:GBD}
Let $\Om$ be an open subset of~$\R^{n}$. We say that a measurable function~$u \colon \Om \to \R^{n}$ belongs to $ GBD_{F}(\Om)$ if there exists $\lambda \in \mathcal{M}_{b}^{+}(\Om)$ such that for every $U \subseteq \Om$, every $\xi \in \mathbb{S}^{n-1}$, and every curvilinear projection $P \colon U \to \xi^{\bot}$ on~$U$ the following facts hold: 
\begin{enumerate}
\item\label{e:slice-1} $\hat{u}^{\xi}_{y} \in BV_{loc} (U^{\xi}_{y})$  for $\HH^{n-1}$-a.e.~$y \in \xi^{\bot}$;
\vspace{1mm}
\item\label{e:slice-2} for every Borel subset $B \in \mathcal{B}(U)$
\begin{align*}
\int_{\xi^{\bot}} \Big( |{\rm D} \hat{u}^{\xi}_{y} | (B^{\xi}_{y} \setminus J^{1}_{\hat{u}^{\xi}_{y}}) & + \HH^{0} (B^{\xi}_{y} \cap J^{1}_{\hat{u}^{\xi}_{y} } ) \Big)\, \di \HH^{n-1}(y) 
\leq \|\dot{\varphi}_\xi\|^2_{L^\infty} \, {\rm Lip}(P_{\xi};U)^{n-1} \lambda(B)\,. 
\end{align*}
\end{enumerate}
\end{definition}

\begin{remark}
\label{r:nontrivial}
Notice that the integral in~\eqref{e:slice-2} is justified by~\cite[Proposition~6.3]{AT_22-preprint}.
%
%We notice that, thanks to the construction in Section~\ref{sub:curvpro} (see Theorem~\ref{p:curvpro}), for every open set $\Om \subseteq \R^{n}$ there exists an at most countable family 
%\[
%\big\{ \big(x_{i}, r_{i}, (P_{\xi, x_{i}})_{\xi \in \mathbb{S}^{n-1}}\big): \, i \in I \big\}
%\] 
%such that $\{ {\rm B}_{r_{i}} (x_{i})\}_{i \in I}$ is a cover of~$\Om$ and $(P_{x_{i}, \xi})_{\xi \in \mathbb{S}^{n-1}}$ is a family of curvilinear projections in~${\rm B}_{r_{i}} (x_{i})$. This justifies the requests~$(1)$ and~$(2)$ in Definition~\ref{d:GBD} and that the set~$GBD_{F}(\Om)$ is nontrivial.
\end{remark}

%\RRR S: Since the construction coming later are local, I was thinking of saying that there exists $\lambda \in \mathcal{M}_{b}(\Om)$ such that for every $U \Subset \Om$, every $P_{\xi}\colon U \to \xi^{\bot}$ CP properties (1) and (2) must hold on $U$. Would all the theory go through? \EEE

%\smallskip
%
%\RRR Before the definition we should to show that $F$ induces (locally) a family $P_{\xi}$ (Section 5). Then we have a countable family that covers all of~$\Om$. \EEE

%\medskip
We recall here a general fact about measurable functions that was proven in~\cite[Theorem~3.5]{dal}. To this end, we introduce the following notation.

\begin{definition}
We denote by~$\mathcal{T}$ the family of all functions $\tau \in C^{1}(\R)$ such that $-\frac{1}{2} \leq \tau \leq \frac{1}{2}$ and $0 \leq \tau' \leq 1$.
\end{definition}

As usual, for a function $v \in L^{1}_{loc} (U)$ and for $\xi \in \mathbb{S}^{n-1}$ we denote by~$D_{\xi} v$ the distributional derivative of~$v$ in direction~$\xi$.

\begin{theorem}
\label{t:DM-3.5}
Let $v \colon U \to \R$ be measurable, $\xi \in \mathbb{S}^{n-1}$, and $\theta \in \mathcal{M}^{+}_{b}(U)$. Then, the following facts are equivalent:
\begin{itemize}
\item[$(i)$] for every $\tau \in \mathcal{T}$, $D_{\xi}(\tau(v))$ belongs to~$\mathcal{M}_{b}(U)$ and
\begin{displaymath}
| D_{\xi} (\tau(v))| (B) \leq \theta(B) \qquad \text{for every $B \in \mathcal{B}(U)$;}
\end{displaymath}

\item[$(ii)$] for $\mathcal{H}^{n-1}$-a.e.~$y \in \xi^{\bot}$ the function $\widetilde{v}^{\xi}_{y}$ belongs to $BV_{loc}(\widetilde{U}^{\xi}_{y})$ and
\begin{displaymath}
\int_{\xi^{\bot}} \big(| D\widetilde{v}^{\xi}_{y}| (\widetilde{B}^{\xi}_{y} \setminus J^{1}_{\widetilde{v}^{\xi}_{y}}) + \mathcal{H}^{0} (\widetilde{B}^{\xi}_{y} \cap J^{1}_{\widetilde{v}^{\xi}_{y}}) \big) \di \mathcal{H}^{n-1}(y) \leq \theta(B)
\end{displaymath}
for every $B \in \mathcal{B}(U)$.
\end{itemize}
\end{theorem}

As a consequence of Theorem~\ref{t:DM-3.5} we can show the equivalent property in our parametrized setting.

\begin{corollary}
\label{c:DM-3.5}
Let $\Om$ be an open subset of~$\R^{n}$ and $u \colon \Om \to \R^{n}$ measurable. Then, $u \in  GBD_{F} (\Om)$ if and only if there exists $\lambda \in \mathcal{M}^{+}_{b}(\Om)$ such that for every $U \subseteq \Om$, every $\xi \in \mathbb{S}^{n-1}$, every curvilinear projection $P_\xi$ on~$U$, every $\tau \in \mathcal{T}$, and every $B \in \mathcal{B}(U)$ it holds
\begin{equation*}
%\label{e:variation-tau}
(\varphi_{\xi})_{\sharp}  \big| D_{\xi} \big( \tau((u \cdot \xi_{\varphi}) \circ \varphi_{\xi} ) \big) \big| (B) \leq  \|\dot{\varphi}_\xi\|^2_{L^\infty} \,  {\rm Lip} (P_{\xi} ; U) ^{n-1} \lambda(B)\,.
\end{equation*} 
\end{corollary}

\begin{proof}
We may apply Theorem~\ref{t:DM-3.5} to the function $v = (u \cdot \xi_\varphi) \circ \varphi_{\xi} \colon \varphi_{\xi}^{-1}(U) \to \R$ and $\theta = \|\dot{\varphi}_\xi\|^2_{L^\infty} {\rm Lip} (P_{\xi} ; U) ^{n-1} (\varphi_{\xi}^{-1})_{\sharp} \lambda$ and use the fact that $\varphi_{\xi}$ is a bi-Lipschitz diffeomorphism with its image. Indeed, it is enough to notice that for every $C \in \mathcal{B} (\varphi_{\xi}^{-1}(U))$, we have $B := \varphi_{\xi}(C) \in \mathcal{B}(U)$ and $B^{\xi}_{y} = \widetilde{C}^{\xi}_{y}$ for $y \in \xi^{\bot}$. Moreover, $\widetilde{v}^{\xi}_{y} (t) = \hat{u}^{\xi}_{y} (t)$ for a.e.~$t \in B^{\xi}_{y}$.
\end{proof}

Similar to~\cite[Definition 3.7]{dal}, we now define the measures~$\mu^{\xi}_{y}$ and~$\mu^{\xi}_{u}$.

\begin{definition}
\label{d:muximeasure}
Let $\Om$ be an open subset of~$\R^{n}$, $U\subseteq \Om$ open, $\xi \in \mathbb{S}^{n-1}$, let $P_{\xi} \colon U \to \xi^{\bot}$ be a curvilinear projection on~$U$, and $u \in GBD_{F}(\Om)$. For $\mathcal{H}^{n-1}$-a.e.~$y \in \xi^{\perp}$ we define the measure $\mu^{\xi}_{y} \in \mathcal{M}^{+}_{b}( U^{\xi}_{y} )$ as
\begin{displaymath}
\mu^{\xi}_{y} (B) := |{\rm D} \hat{u}^{\xi}_{y} | \big( B^{\xi}_{y} \setminus J^{1}_{\hat{u}^{\xi}_{y}} \big)  + \HH^{0} \big( B^{\xi}_{y} \cap J^{1}_{\hat{u}^{\xi}_{y} } \big)  
\end{displaymath} 
for $B \in \mathcal{B}(U^{\xi}_{y})$. We also define $\mu^{\xi}_{u} \in \mathcal{M}^{+}_{b}(U)$ as
\begin{equation*}
%\label{e:measure-mu-xi}
\mu^{\xi}_{u} (B) := \int_{\xi^{\bot}} \Big( |{\rm D} \hat{u}^{\xi}_{y} | (B^{\xi}_{y} \setminus J^{1}_{\hat{u}^{\xi}_{y}})  + \HH^{0} (B^{\xi}_{y}\cap J^{1}_{\hat{u}^{\xi}_{y} } ) \Big) \di \mathcal{H}^{n-1}(y)
\end{equation*} 
for every $B \in \mathcal{B}(U)$.
\end{definition}

\begin{remark}
We notice that the measures $\mu^{\xi}_{y}$ and $\mu^{\xi}_{u}$ depend on the choice of the curvilinear projection~$P_{\xi}$. For simplicity, we have decided to not explicitly indicate such dependence in our notation.
\end{remark}

In the following proposition we state the lower semicontinuity of~$\mu^{\xi}_{u}(U)$ w.r.t.~$\xi \in \mathbb{S}^{n-1}$. It is worth noting that such a lower semicontinuity will be only used to extend some structure properties holding for a.e. $\xi$ to the entire of $\mathbb{S}^{n-1}$. Even if this fact can be deduced from \cite{dal}, for convenience of the reader we present here its proof.

\begin{proposition}
\label{p:xi-lsc}
Let $\Om$ be an open subset of~$\R^{n}$, $U \subseteq \Om$ open, let $(P_{\xi})_{\xi \in \mathbb{S}^{n-1}}$ be a family of curvilinear projections on~$U$, let $u \in GBD_{F}(\Om)$, and let $(\mu^{\xi}_{u})_{\xi \in \mathbb{S}^{n-1}}$ be the family of measures introduced in Definition~\ref{d:muximeasure}. Then, for every $\xi_{j}, \xi \in \mathbb{S}^{n-1}$ such that $\xi_{j} \to \xi$ we have that
\begin{equation}
\label{e:xi-lsc}
\mu^{\xi}_{u} (U) \leq \liminf_{j \to \infty} \, \mu^{\xi_{j}}_{u} (U)\,.
\end{equation}
\end{proposition}

\begin{proof}
%We first prove that
%\begin{equation}
%\label{e:lsc-mu-xi-u}
%\mu^{\xi}_{u}(U) \leq \liminf_{j \to \infty} \mu^{\xi}_{u_{j}} (U)\,,
%\end{equation}
%whenever $U \subseteq \Omega$ is open and $u_j \to u$ in measure on~$U$. Indeed, For every $\psi \in C^{1}_{c} (\varphi^{-1}_{\xi} (U))$ the map
%\begin{displaymath}
%u \mapsto \int_{\varphi^{-1}_{\xi} (U)}  \tau ((u \cdot \xi_{\varphi}) (\varphi_{\xi}(z)) \nabla\psi(z) \cdot \xi \, \di z
%\end{displaymath}
%is continuous w.r.t.~the convergence in measure. Hence, $u \mapsto | D_{\xi} (\tau(u \cdot \xi_{\varphi}) \circ \varphi_{\xi})| (U)$ is lower-semicontinuous. 

We prove that given $U \subseteq \Omega$ open the value $\mu^\xi_u(U)$ can be obtained as
\begin{equation*}
\label{e:some-sup}
\mu^{\xi}_{u} (U) = \sup_{k \in \mathbb{N}} \, \sup \sum_{i=1}^{k}  \big| D_{\xi} (\tau_{i} ((u \cdot \xi_\varphi) \circ \varphi_{\xi})) \big| (\varphi_{\xi}^{-1}(U_{i}))
\end{equation*}
where the second supremum is taken over all the families $\tau_{1}, \ldots \tau_{k} \in \mathcal{T}$ and all the families of pairwise disjoint open subsets~$U_{1}, \ldots U_{k}$ of~$U$. To this purpose, let $\varphi_{\xi}$ be the parametrization of~$P_{\xi}$ over~$U$, as given in Definition~\ref{d:CP}. Denoting $v :=(u \cdot \xi_\varphi) \circ \varphi_{\xi}$, by \cite[Theorem~3.8]{dal} we have that for every open set $V \subseteq \varphi^{-1}_{\xi}(U)$ it holds
\begin{align*}
\label{e:some-sup-2}
\int_{\xi^{\bot}} &  \Big( |{\rm D} \widetilde{v}^{\xi}_{y} | (\widetilde{V}^{\xi}_{y} \setminus J^{1}_{\widetilde{v}^{\xi}_{y}})  + \HH^{0} (\widetilde{V}^{\xi}_{y}\cap J^{1}_{\widetilde{v}^{\xi}_{y} } ) \Big) \di \mathcal{H}^{n-1}(y)
= \sup_{k \in \mathbb{N}} \, \sup \sum_{i=1}^{k}  \big| D_{\xi} (\tau_{i} (v)) \big| (V_{i})\,, 
\end{align*}
where the second supremum is taken over all the families $\tau_{1}, \ldots \tau_{k} \in \mathcal{T}$ and all the families of pairwise disjoint open subsets~$V_{1}, \ldots V_{k}$ of~$V$. Using the fact that $\varphi_{\xi}$ is a bi-Lipschitz diffeomorphism with its image, we have that for $V= \varphi_{\xi}^{-1}(U)$ with $U \subseteq \Omega$ open it holds $\widetilde{V}^{\xi}_{y} = U^{\xi}_{y}$ for every $\xi \in \mathbb{S}^{n-1}$ and every $y \in \xi^{\bot}$. Moreover, $\widetilde{v}^{\xi}_{y} (t) = \hat{u}^{\xi}_{y}(t)$ for a.e.~$t \in U^{\xi}_{y}$.

Now we observe that for every $\psi \in C^{1}_{c} (\varphi^{-1}_{\xi} (U))$ and for $j \in \mathbb{N}$ large enough we have that $\psi \in C^{1}_{c} (\varphi^{-1}_{\xi_{j}} (U))$. Thus, for every $\tau \in \mathcal{T}$ we have
\begin{align*}
\int_{\varphi^{-1}_{\xi}(U)} \tau ((u \cdot\xi_{\varphi} ) (\varphi_{\xi}(z)) \nabla\psi(z) \cdot \xi \, \di z = \lim_{j\to\infty}\, \int_{\varphi^{-1}_{\xi_{j} }(U)} \tau (u \cdot\xi_{j, \varphi}) (\varphi_{\xi_{j}}(z)) \nabla\psi(z) \cdot \xi_{j} \, \di z\,.
\end{align*}
This, together with Definition~\ref{d:CP} of curvilinear projection, implies~\eqref{e:xi-lsc}.
\end{proof}

\subsection{Equivalence of $GBD({\rm M})$ and $GBD_{F}(\Om)$.}
\label{s:equivalence}

\RRR Let $({\rm M}, g)$ be a Riemannian manifold of dimension~$n$. \EEE  We now show that, when restricted to a chart $(U,\psi)$ on~$\rm M$ with $U \Subset \rm M$, the spaces~$GBD(U)$ and~$GBD_{F}(\psi(U))$ \RRR are equivalent, \EEE with the field~$F$  defined in~\eqref{e:chris}. We recall that properties~\eqref{hp:F} and~\eqref{hp:F2} for \RRR such \EEE $F$ have been proven in~\cite[Section~6.4]{AT_22-preprint}. The above equivalence implies that the structure of the space~$GBD(\rm M)$ can be deduced \RRR from \EEE the structure of~$GBD_{F} (\Om)$, that will be discussed in the next sections. \RRR Using \EEE the classical notation for manifolds we write $\omega \in \mathcal{D}^{1}(\rm M)$ in coordinates on a chart~$(U, \psi)$ as
\begin{equation*}
%\label{e:upsi}
\omega(p) = \sum_{i=1}^{n} \omega_{i} (p) g^{i} (p) \qquad \text{for $p \in U$,}
\end{equation*}
where $\omega_{i} \colon U \to \R$ is measurable and for every $i =1, \ldots, n$ and every $f \in C^{\infty} (\rm M)$ we have $g_{i} \in \Gamma({\rm TM})$
\begin{equation*}
g_{i} (f) (p) := \frac{\partial (f \circ \psi^{-1})}{\partial x_{i}} (\psi(p))\qquad \text{for $p \in U$,}
\end{equation*}
and $g^{i} \in \mathcal{D}^{1} (\rm M)$ is such that $\langle g^{i} (p) , g_{j} (p) \rangle = \delta_{ij}$. We also define the function~$u \colon \psi(U) \to \R^{n}$ as
\begin{equation}
\label{e:vpsi}
u(x) := \sum_{i=1}^{n} \omega_{i} (\psi^{-1} (x)) e_{i}\qquad \text{for $x \in \psi(U)$,}
\end{equation}
where $\{e_{i}\}_{i=1}^{n}$ denotes the canonical basis of~$\R^{n}$.

\begin{proposition}
\label{p:equivalent}
Let $({\rm M}, g)$ be a Riemannian manifold of dimension~$n$, let~$\omega \in GBD(\rm M)$, and let~$(U,\psi)$ be a chart on~$\rm M$ with~$U \Subset \rm M$. Then, the function~$u\colon \psi(U) \to \R^{n}$ defined in~\eqref{e:vpsi} belongs to~$GBD_{F} (\psi(U))$ for~$F$ as in~\eqref{e:chris}. In particular, if $\lambda \in \mathcal{M}_{b}^{+} (\rm M)$ can be used in~(2) of Definition~\ref{d:GBD_M} for $\omega$, then 
\begin{equation}
\label{e:lambda-bar}
\overline{\lambda}:=  \big[ {\rm Lip} (\psi^{-1}; \psi(U)) ^{2}  \RRR {\rm Lip} (\psi; U )^{n-1} \EEE \big] \psi_{\sharp}\lambda  \quad  \in \mathcal{M}_{b}^{+} (\psi(U))
\end{equation}
can be used in~(2) of Definition~\ref{d:GBD} for~$u$.
\end{proposition}

\begin{proof}
Let $\omega  \in \mathcal{D}^{1}(\rm M)$, let $V \subseteq \rm M$ open, let $P\colon V \to \xi^{\bot}$ be a curvilinear projection on~$V$ with parametrization~$\varphi \colon \{ y + t\xi: \, (t, \xi) \in [\xi^{\bot} \cap {\rm B}_{\rho} (0)] \times (-\tau, \tau) \} \to \rm M$, and let~$(U, \psi)$ be a chart on~$\rm M$ with $U \Subset V$. We set $\overline{P}:= P \circ \psi^{-1}$ and $\overline{\varphi}:= \psi\circ \varphi$. Thanks to Remark~\ref{r:M2},~$\overline{P}$ is a curvilinear projection on~$\psi(U)$ with respect to~$F$ defined in~\eqref{e:chris} with parametrization~$\overline{\varphi}$. We write $\dot{\varphi}$ in coordinates as
\begin{align}
\label{e:dotphi}
\dot{\varphi} (y + t\xi ) & = \sum_{i=1}^{n} \bigg[\frac{\partial(\psi\circ \varphi (y + t\xi) )}{\partial t}  \bigg]_{i} g_{i} ( \varphi (y + t\xi ) ) 
\\
&
=  \sum_{i=1}^{n} \bigg[\frac{\partial \overline{\varphi}(y + t\xi) }{\partial t}  \bigg]_{i} g_{i} ( \varphi (y + t\xi ) ) 
=  \sum_{i=1}^{n}  \dot{\overline{\varphi}}_{i} (y + t\xi)  g_{i} ( \varphi (y + t\xi ) ) \,. \nonumber
\end{align}
Hence, for $(y, t) \in [\xi^{\bot} \cap {\rm B}_{\rho} (0)] \times (-\tau, \tau)$ we have that
\begin{align*}
\hat{\omega}^{\xi}_{y} (t) & = \left\langle  \omega(\varphi(y + t\xi) , \dot{\varphi} (y + t\xi) \right\rangle _{\varphi(y + t\xi)} 
\\
&
 = \left\langle \sum_{i=1}^{n} \omega_{i} (\varphi(y + t\xi)) \,  g^{i} ( \varphi (y + t\xi ) ) ,  \sum_{i=1}^{n}  \dot{\overline{\varphi}}_{i} (y + t\xi)  g_{i} ( \varphi (y + t\xi ) ) \right\rangle_{\varphi(y + t\xi)} 
\\
&
 = \sum_{i=1}^{n}  \omega_{i} (\varphi(y + t\xi)) \dot{\overline{\varphi}}_{i} (y + t\xi) = \sum_{i=1}^{n} u_{i} (\overline{\varphi} (y + t\xi)) \dot{\overline{\varphi}}_{i} (y + t\xi) = \hat{u}^{\xi}_{y} (t) \,.
\end{align*} 
This implies that for every $B \in \mathcal{B} (\psi(U))$ it holds true
\begin{align}
\label{e:man-rn}
 \int_{\xi^{\bot}}&   \big(\big| | {\rm D} \hat{u}^{\xi}_{y} | ( B^{\xi}_{y} \setminus J^{1}_{\hat{u}^{\xi}_{y}}) + \mathcal{H}^{0} (  B^{\xi}_{y} \cap J^{1}_{\hat{u}^{\xi}_{y}} ) \big) \, \di \mathcal{H}^{n-1} (y)
\\
&
= \int_{\xi^{\bot}} \Big(\big|  {\rm D} \hat{\omega}^{\xi}_{y} | ( \psi^{-1} (B)^{\xi}_{y} \setminus J^{1}_{\hat{\omega}^{\xi}_{y}}) + \mathcal{H}^{0} (  \psi^{-1} (B)^{\xi}_{y} \cap J^{1}_{\hat{\omega}^{\xi}_{y}} ) \Big) \, \di \mathcal{H}^{n-1} (y)  \nonumber
\\
&
\leq \RRR \|  \dot\varphi  \|^2_{L^{\infty}, \rm M} \EEE \, {\rm Lip} (P; U)^{n-1} \lambda(\psi^{-1} (B))\,. \nonumber
\end{align}
By definition of~$\overline{P}$ we have that
\begin{equation}
\label{e:LIPP}
{\rm Lip} (P ;  U ) \leq {\rm Lip} (\psi; U ) \, {\rm Lip}(\overline{P}; \psi(U))\,.
\end{equation}
Since $\rm M$ is a Riemannian manifold and $U \Subset V$, by~\eqref{e:dotphi} we have that for $p \in U$
\begin{align*}
|\dot{\varphi} |_{p} & = \sum_{i, j=1}^{n}\bigg[\frac{\partial(\psi\circ \varphi (\RRR P(p) \EEE + t\xi) )}{\partial t} (t_{p})  \bigg]_{i}  \bigg[\frac{\partial(\psi\circ \varphi (\RRR P(p) \EEE + t\xi) )}{\partial t} (t_{p}) \bigg]_{j}   
\\
&
\qquad\qquad g_{i} (\varphi (P(p) + t_{p}\xi )) \cdot g_{j} (\varphi (P(p) + t_{p}\xi )) \nonumber
\\
&
\nonumber = \sum_{i, j=1}^{n}  \dot{\overline{\varphi}}_{i} (P(p) + t_{p}\xi)  \dot{\overline{\varphi}}_{j} (P(p) + t_{p}\xi) g_{i} ( \varphi (P(p) + t_{p}\xi ) ) \cdot g_{j} ( \varphi (P(p) + t_{p}\xi ) )\,,
\end{align*}
where we have denoted by $t_{p}$ the unique $t \in (-\tau, \tau)$ such that $p = \varphi (P(p) + t_{p} \xi)$. Hence, we deduce that
\begin{align}
\label{e:dotphi111}
\RRR \| \dot{\varphi}   \|_{L^{\infty}, \rm M} \EEE  \leq  {\rm Lip} (\psi^{-1}; \psi(U)) ^{2}  \| \dot{\overline{\varphi}} \|_{L^{\infty} (\psi(U))}\,.
\end{align}
%we have that there exists $C_{U} >0$ depending only on the chart~$(U, \psi)$ but not on the curvilinear projection~$P$ such that
%\begin{align}
%\label{e:dotphi}
% \| g(\dot{\varphi})  \|_{L^{\infty} (U)} \leq  C_{U} \| \dot{\overline{\varphi}} \|_{L^{\infty} (\psi(U))}\,.
%\end{align}
Setting~$\overline{\lambda}$ as in~\eqref{e:lambda-bar}, we deduce from~\eqref{e:man-rn}--\eqref{e:dotphi111} that for every $B \in \mathcal{B} (\psi(U))$ it holds
\begin{align}
\label{e:M3}
 \int_{\xi^{\bot}}  \Big(\big| | {\rm D} \hat{u}^{\xi}_{y} | ( B^{\xi}_{y} \setminus J^{1}_{\hat{u}^{\xi}_{y}}) & + \mathcal{H}^{0} (  B^{\xi}_{y} \cap J^{1}_{\hat{u}^{\xi}_{y}} ) \big) \, \di \mathcal{H}^{n-1} (y)
 \\
 &
  \leq \| \dot{\overline{\varphi}}\|^2_{L^{\infty} (\psi(U))} {\rm Lip} (\overline{P}; \psi(U))^{n-1} \overline{\lambda} (B)\,. \nonumber
\end{align}
By \RRR the first equality in~\eqref{e:man-rn} \EEE and by Remark~\ref{r:M2} we can show that~\eqref{e:M3} actually holds for every curvilinear projection~$\widetilde{P}\colon \Om \to \xi^{\bot}$  on some open subset~$\Om$ of~$\psi(U)$ with respect to the field~$F$ in~\eqref{e:chris}. This implies that $u \in GBD_{F} (U)$. 
\end{proof}

\begin{remark}
\label{r:lambda-bar}
For later use, we notice that for every $\epsilon>0$ we may further assume, up to rescaling~$\psi$ and taking~$U \Subset V$ small enough (depending on~$\epsilon)$), that 
\begin{displaymath}
 {\rm Lip} (\psi; U ) < 1 + \epsilon \,, \qquad  {\rm Lip} (\psi^{-1}; \psi(U)) < 1 + \epsilon\,,
 \end{displaymath}
 which in turn implies that $\overline{\lambda} \leq \RRR (1 + \epsilon)^{n+1}\EEE \psi_{\sharp} \lambda$, with the notation used in~\eqref{e:lambda-bar}.
\end{remark}

\section{A particular family of curvilinear projections and its properties}
\label{s:curvilinear}
 
 We recall here the construction of a local family of curvilinear projections discussed in~\cite[Section~3.3]{AT_22}.

\begin{definition}
\label{d:curvpro}
Let $x_0 \in \mathbb{R}^n$ and $\rho_{0}>0$. For every $\xi \in {\rm B}_{2}(0)$ and every $y \in \xi^{\bot}\cap {\rm B}_{\rho_{0}}(0)$ we consider the solution $t \mapsto u_{\xi, y}(t)$ \RRR to \EEE the ODE system
\begin{equation*}
%\label{e:curvpro1.1}
    \begin{cases}
    \ddot{u}(t) = F(u(t), \dot{u}(t)) &  t \in \mathbb{R},\\
    u(0)=y+x_0\,,\\
    \dot{u}(0)=\xi\,,
    \end{cases}
\end{equation*}
which is well-defined for $t \in (-\tau, \tau)$, for a suitable $\tau >0$ depending only on~$x_{0}$ and~$\rho_{0}$, but not on~$\xi$ and~$y$. Then, we define $\varphi_{\xi, x_{0}} \colon \mathbb{R}^n \to \mathbb{R}^n$ as follows: for every $x \in \R^{n}$, if $x = y + t\xi$ with $y \in  \xi^{\bot}\cap {\rm B}_{\rho_{0}}(0)$ and $t \in (-\tau, \tau)$, we set $\varphi_{\xi, x_{0}} (x) := u_{\xi, y}(t)$.

We further define $\varphi_{x_{0}} \colon \R^{n} \times \mathbb{S}^{n-1} \to \R^{n}$ as $\varphi_{x_{0}} (x, \xi) := \varphi_{\xi, x_{0}} (x)$ for $x \in \R^{n}$ and $\xi \in \mathbb{S}^{n-1}$.
\end{definition}

In \cite[Corollary~3.21]{AT_22-preprint} it has been shown that for every $x_{0} \in \Om$ there exists $R_{x_{0}}>0$ such that for every $\xi \in {\rm B}_{2}(0)$ the map $\varphi_{\xi, x_{0}}^{-1} \lfloor {\rm B}_{R_{x_{0}}} (x_{0})$ is a diffeomorphism with its image. This justifies the following definition.

\begin{definition}
\label{d:P-xi}
Let $x_{0} \in \Om$, $\xi \in {\rm B}_{2}(0)$ and $R_{x_{0}}>0$ be as above. We define the map $P_{\xi,x_0} \colon {\rm B}_{R_{x_{0}}} (x_{0})  \to \xi^\bot$ as
\begin{equation*}
   % \label{e:curvproj}
    P_{\xi,x_0}:=\pi_\xi \circ \varphi_{\xi, x_{0}} ^{-1},
\end{equation*}
where $\pi_\xi \colon \mathbb{R}^n \to \xi^\bot$ denotes the orthogonal projection onto the orthogonal to $\xi$. 
%Furthermore, for $r>0$ we define $P_{\xi, x_{0}, r}\colon {\rm B}_{\frac{R_{x_{0}}}{r}} (0) \to \xi^{\perp}$ as
%\begin{equation}
%    \label{e:curvproj-2}
%    P_{\xi,x_0, r}:=\pi_\xi \circ \varphi_{\xi, x_{0},r} ^{-1}.
%\end{equation}
%Since at this level the definition is only formal, we do not discuss the invertibility of $\varphi_{\xi, x_{0}}$. Moreover, when no misunderstandings may occur, we omit the dependence on~$x_0$ and write simply~$P_{\xi}$.
\end{definition}

In~\cite[Theorem~3.25]{AT_22-preprint} the following result has been proven.

\begin{theorem}
\label{p:curvpro}
Let $F \in C^{\infty}( \mathbb{R}^n \times \mathbb{R}^n ; \mathbb{R}^n)$ satisfy condition~\eqref{hp:F}. Then, for every $x_0 \in \Omega$ there exists $R_0>0$ such that the family of maps $\{ P_{\xi,x_0} \colon {\rm B}_{R_0}(x_0) \to \xi^\bot: \xi \in \mathbb{S}^{n-1}\}$ is a family of curvilinear projections on~${\rm B}_{R_{0}}(x_{0})$. 
\end{theorem}

We further recall the definitions of~$\varphi_{\xi, x_{0}, r}$, of~$P_{\xi, x_{0}, r}$, and of the exponential map \RRR in~$\R^{n}$ induced by~$F$ \EEE (see also~\cite[Definitions~3.11, 3.19, and 3.22]{AT_22-preprint}).

\begin{definition}
\label{d:varphi-xi-r}
Let $x_{0} \in \R^{n}$ and $R_{0}>0$ be as in Theorem~\ref{p:curvpro}, and let $r>0$. For every $\xi \in {\rm B}_{2}(0)$ we define
\begin{align*}
%\label{e:varphi-xi-r}
\varphi_{\xi,x_0,r}(x) &  :=r^{-1}(\varphi_{ \xi, x_{0}} (rx) - x_0) \qquad \text{\RRR for $x \in {\rm B}_{\frac{R_{0}}{ r}} (0)$,\EEE}\\
%\label{e:varphi-xi-r-2}
 P_{\xi,x_0, r} & :=\pi_\xi \circ \varphi_{\xi, x_{0},r} ^{-1}\,.
\end{align*}
\end{definition}

\RRR
In \cite[Lemmas~3.20 and 3.24]{AT_22-preprint} we have proven the following.

\begin{lemma}
\label{l:3.20-24}
Let $x_{0} \in \R^{n}$ and let~$\varphi_{\xi,x_0,r}$ and~$P_{\xi,x_0, r}$ be as in Definition~\ref{d:varphi-xi-r}. Then, 
\begin{align*}
& \varphi_{\xi,x_0,r} \to id \qquad \text{in $C^{\infty}_{loc} (\R^{n}; \R^{n})$,}\\
& P_{\xi,x_0, r} \to \pi_{\xi} \qquad \text{in $C^{\infty}_{loc} (\R^{n}; \R^{n})$,}
\end{align*}
and the convergences are uniform w.r.t.~$\xi \in \mathbb{S}^{n-1}$.
\end{lemma}
\EEE

%\begin{lemma}
%\label{l:curvpro1.3}
%For every $x_{0} \in \Om$ it holds
%\begin{align}
%    \label{e:curvpro1.3}
%  &  \varphi_{\xi, x_0,r} \to  id  
%\end{align}
%in $C^{\infty}_{loc} (\R^{n}; \R^{n})$ as $r\searrow0$, uniformly w.r.t.~$\xi \in {\rm B}_{2} (0)$.
%\end{lemma}

%\begin{remark}\label{r:rx0}
%Under the assumptions of Definition~\ref{d:curvpro}, the map $\varphi_{\xi, x_{0}}$ is well defined on the open ball ${\rm B}_{r_{x_{0}}} (0)$, for a suitable $r_{x_{0}}>0$ which only depends on~$x_{0}$, but not on~$\xi \in {\rm B}_{2}(0)$.
%\end{remark}

\begin{definition}
\label{d:exp}
 Let $F \in C^{\infty} (\R^{n} \times \R^{n}; \R^{n})$ satisfy~\eqref{hp:F}. For every $x_{0} \in \R^{n}$ we define, where it exists, the {\em exponential map} $\text{exp}_{x_0} \colon \mathbb{R}^n \to \mathbb{R}^n$ as $\text{exp}_{x_0}(\xi):= v_{\xi, x_{0}} (1)$, where $t\mapsto v_{\xi, x_{0}} (t)$ solves
\begin{equation}
\label{e:system-exponential}
\begin{cases}
\ddot{u}(t) = F(u(t),\dot{u}(t)), \ &t \in \mathbb{R}\,, \\
u(0)=x_0\,,&\\
\dot{u}(0)=\xi\,.&
\end{cases}
\end{equation}
\end{definition}

\RRR
\begin{remark}
Points on a manifold~${\rm M}$ are always denoted by~$p$ and~$q$, while points in~$\R^{n}$ will be denoted by~$x$ or~$x_{0}$. Hence, there is no confusion in the definition of the two exponential maps ${\rm exp}_{p} \colon {\rm T}_{p}{\rm M} \to \rm M$ and ${\rm exp}_{x_{0}} \colon \R^{n} \to \R^{n}$.
\end{remark}
\EEE

In the next definition we introduce the concept of injectivity radius.

\begin{definition}
For every $x_0 \in \mathbb{R}^n$ we define the injectivity radius $\text{inj}_{x_0}\in  [0, + \infty)$ as the supremum of all $r>0$ for which $\text{exp}_{x_0} \restr {\rm B}_{\overline{r}}(0)$ is well defined and $\text{exp}^{-1}_{x_0} \restr {\rm B}_{\overline{r}}(x_0)$ is a diffeomorphism with its image. 
\end{definition}

The well-posedness of~$\text{exp}_{x_0}$ in a small ball ${\rm B}_{r} (0)$ has been justified in~\cite[Lemma~3.13]{AT_22-preprint}. We recall here the statement, together with the asymptotic behavior of the exponential map.

\begin{lemma}
\label{l:exp}
Let $F \in C^{\infty}( \mathbb{R}^n \times \mathbb{R}^n ; \mathbb{R}^n)$ satisfy~\eqref{hp:F}.
%be 2-homogeneous in the second variable. 
For every~$x_0 \in \Omega$ we have $\emph{inj}_{x_0}>0$. Moreover, for $x \in {\rm B}_1(0)$ and $r >0$ we have
\begin{align}
\label{e:retr11}
& \frac{\text{\em exp}^{-1}_{x_0}(x_0+rx)}{r|x|} - \frac{x}{|x|} = o(r|x|)\,,
\\
&
\label{e:retr16}
   \frac{\text{\em exp}^{-1}_{x_0}(x_0+rx)}{|\text{\em exp}^{-1}_{x_0}(x_0+rx)|} = \frac{x}{|x|} + o(r|x|)\,.
\end{align}
\end{lemma}

\begin{proof}
 By the $2$-homogeneity of $F(x,\cdot)$ we get that 
\begin{equation}
\label{e:retr17}
v_{s\xi, x_{0}} (t) = v_{\xi, x_{0}} (st)\qquad  \text{for $s,t \in [0, + \infty)$,  $\xi \in \mathbb{R}^n$.}
\end{equation}
Hence, by the local well-posedness of ODEs we have that there exists~$r>0$ such that~$\text{exp}_{x_{0}}$ is well-defined on~${\rm B}_{r}(0)$. 

For every $i \in \{1,\dotsc,n \}$ we have that
\[
D\exp_{x_0}(0) e_i = \lim_{t \to 0^+} \frac{v_{te_{i}, x_{0}} (1) - v_{0, x_{0}} (1)}{t} = \lim_{t \to 0^+} \frac{v_{e_{i}, x_{0}} (t) - v_{e_{i}, x_{0}} (0) }{t} = \dot{v}_{e_{i}, x_{0}} ( 0) = e_i\,.
\]
Thus, the differential of $\text{exp}_{x_0}$ at $0$ is the identity. Applying the implicit function theorem, we find a sufficiently small $\tilde r>0$ such that $\text{exp}^{-1}_{x_0} \restr {\rm B}_{\tilde r}(x_0)$ is a diffeomorphism with its image. We conclude by setting $\overline{r} := \min\{ r, \tilde{r}\}$.

Since $\text{exp}_{x_{0}}$ is $C^\infty$-regular and its differential at $0$ is the identity, we get~\eqref{e:retr11}. As a consequence, for every $x \in {\rm B}_1(0)$ and $r>0$ 
\begin{equation*}
 \frac{\text{exp}^{-1}_{x_0}(x_0+rx)}{|\text{exp}^{-1}_{x_0}(x_0+rx)|} - \frac{x}{|x|} =\bigg(\frac{x}{|x|} +o(r|x|) \bigg)( 1+  o(r|x|)) -\frac{x}{|x|}=o(r|x|)\,,
\end{equation*}
which is exactly~\eqref{e:retr16}.
\end{proof}

We introduce an auxiliary map~$\phi_{x_{0}}$ for $x_{0} \in \Om$.

\begin{definition}
 Let $F \in C^{\infty}( \mathbb{R}^n \times \mathbb{R}^n ; \mathbb{R}^n)$ satisfy~\eqref{hp:F}, let $x_{0} \in \Om$, let $(P_{\xi, x_{0}})_{\xi \in \mathbb{S}^{n-1}}$ be the family introduced in Definition~\ref{d:P-xi}, and let~$R_{0}$ be as in Theorem~\ref{p:curvpro}. Thanks to Lemma~\ref{l:exp} we may define for $0 < \overline{r} < \text{inj}_{x_0}$ the map $\phi_{x_{0}}\colon \mathrm{B}_{\overline{r}}(x_0) \setminus\{x_{0}\} \to \mathbb{S}^{n-1}$ as
\begin{equation}
\label{e:retr12}
\phi_{x_{0}}(x):=  \frac{\text{exp}^{-1}_{x_0}(x)}{|\text{exp}^{-1}_{x_0}(x)|}  \qquad\text{for every }x \in \mathrm{B}_{\overline{r}}(x_0) \setminus \{x_0\}\,.
\end{equation}
%\UUU Is it not $\xi_{\varphi} (x_{0})$ in the definition of~$\psi_{x_{0}}$? See also after (3.19). \EEE 
\end{definition}

The following proposition holds.

\begin{proposition}
\label{p:retr}
%Let $F \in C^{\infty}( \mathbb{R}^n \times \mathbb{R}^n ; \mathbb{R}^n)$ be 2-homogeneous in the second variable and 
 Let $F \in C^{\infty}( \mathbb{R}^n \times \mathbb{R}^n ; \mathbb{R}^n)$ satisfy~\eqref{hp:F}, let $x_{0} \in \Om$, let $(P_{\xi, x_{0}})_{\xi \in \mathbb{S}^{n-1}}$ be the family be the family introduced in Definition~\ref{d:P-xi} \RRR parametrized by~$\varphi_{x_{0}}$, \EEE and let $R_{0}$ be as in Theorem~\ref{p:curvpro}. For $0 < \overline{r} < \emph{inj}_{x_0}$, the function $\phi_{x_{0}}  \in C^{1}(  \mathrm{B}_{\overline{r}}(x_0) \setminus \{x_0\} ; \mathbb{S}^{n-1})$ satisfies
\begin{align}
\label{e:chi1.1}
     &P_{\xi}(x)=P_{\xi}(x_0) \ \ \text{if and only if $\xi = \phi_{x_0}(x)$ for every $x \in \mathrm{B}_{\overline{r}}(x_0)\setminus \{x_0\}$.}\\
     \label{e:retr2}
     &  \frac{C'_{x_0}}{|x-x_0|^{n-1}}\leq |\emph{J}\phi_{x_{0}}(x)| \leq \frac{C_{x_0}}{|x-x_0|^{n-1}} \qquad \text{for every $x \in \mathrm{B}_{\overline{r}}(x_0) \setminus \{x_0\}$},
\end{align} 
for some constant $C_{x_0}, C'_{x_0} >0$. 
%Moreover, if we assume that
%\begin{equation}
%\label{e:retr9.1.3}
%\inf_{(x,\xi) \in {\rm B}_{\overline{r}}(x_0)\times \mathbb{S}^{n-1}}|\text{J}_x P_{\xi}(x)|>0 \ \,
%\end{equation}
%then we find a constant $>0$ such that
%\begin{equation}
%    \label{e:retr2.1}
%    |\emph{J}\phi_{x_{0}}(x)|  \qquad \text{for every $x \in \mathrm{B}_{\overline{r}}(x_0) \setminus \{x_0\}$.}
%\end{equation}
\end{proposition}

\begin{proof}
The result is a byproduct of~\cite[Proposition~3.15]{AT_22-preprint} \RRR and of Lemma~\ref{l:3.20-24}. \EEE
\end{proof}

We conclude by introducing the maps $\chi_{x_{0}}$, and~${\rm d} (\cdot, x_{0})$ for $x_{0} \in \Om$ and by collecting their convergence properties.

\begin{definition}
\label{d:chi1.1.1}
 Let $F \in C^{\infty}( \R^{n} \times \R^{n} ; \R^{n})$ satisfy~\eqref{hp:F},  let $x_{0} \in \Om$, let~$(P_{\xi, x_{0}})_{\xi \in \mathbb{S}^{n-1}}$ be the family introduced in Definition~\ref{d:P-xi}, and let~$R_{0}$ be as in Theorem~\ref{p:curvpro}. Given $0 < \overline{r} < \text{inj}_{x_0}$ we define the vectorfield $\chi_{x_{0}} \colon {\rm B}_{\overline{r}}(x_0) \setminus \{x_0\} \to \mathbb{S}^{n-1}$ by
\begin{equation}
\label{e:chi1.1.1}
    \chi_{x_{0}} (x)= \xi_\varphi(x) %\xi(x)
     \ \ \text{with} \ \ \xi = \phi_{x_{0}}(x).
\end{equation}
%\RRR Is it better to write $\chi_{x_{0}} (x) = \xi_{\varphi} (x) = $? I think the notation above is misleading, because $\chi$ acts on~$x$, but $\varphi$ acts on some parametrization of~$\R^{n}$. \EEE
Moreover, we define the function $\text{d}(\cdot,x_0) \colon {\rm B}_{\overline{r}}(x_0) \to \mathbb{R}$ by
\begin{equation}
    \label{e:chi1.2}
    \text{d}(x,x_0):= |\text{exp}_{x_0}^{-1}(x)|\,.
\end{equation}
%where $\text{exp}_{x_0}$ is the exponential map introduced at the beginning of the proof of Proposition \ref{p:retr}.
\end{definition}

\begin{proposition}
Let $x_0 \in \Omega$,  let~$(P_{\xi, x_{0}})_{\xi \in \mathbb{S}^{n-1}}$ be the family of curvilinear projections defined in Definition~\ref{d:P-xi}, and let $R_0>0$ be given by Theorem \ref{p:curvpro}. Then, the following convergences hold:
\begin{align}
    \label{e:chi1}
   & \phi_{x_{0}} (x_0 +r\xi) \to \xi \qquad \text{in } C^\infty(\mathbb{S}^{n-1};\mathbb{S}^{n-1}), \  \text{as } r \searrow 0\,,
   \\
    \label{e:chi2}
    &\frac{x}{r|x|^2}-\frac{\chi_{x_{0}} (x_0+rx)}{r|x|} \to 0 \qquad \text{uniformly in }  {\rm B}_1(0) \setminus \{0\}, \ \text{as } r \searrow 0\,, 
    \\
    \label{e:chi3}
    &\frac{{\rm d}(x_0+rx,x_0)}{r|x|}\to 1  \qquad  \text{uniformly in }  {\rm B}_1(0) \setminus \{0\}, \ \text{as } r \searrow 0\,.
\end{align}
\end{proposition}

\begin{proof}
Since~$x_{0}$ is fixed, we drop the index~$x_{0}$ in the functions~$\phi_{x_{0}}$ and~$\chi_{x_{0}}$. Definition~\ref{d:exp} of the map $\text{exp}_{x_{0}}$,  
%From the definition of exponential map we have $\text{exp}_{x_0}(x)=v(x,1)$ where $u(t):= v(x,t)$ satisfies 
% \[
%\begin{cases}
%\ddot{u}(t) = F(u(t),\dot{u}(t)), \ &t \in \mathbb{R}, \\
%u(0)=x_0,&\\
%\dot{u}(0)=x.&
%\end{cases}
%\]
% Therefore, 
 the uniqueness property of ODEs, and the fact that $F(x,\cdot)$ is 2-homogeneous imply that the solution~$u$ of \eqref{e:system-exponential} with initial datum $\xi := x/|x|$ satisfies $u(|x|)= \text{exp}_{x_{0}} (x)$. From Definition~\ref{d:curvpro} we deduce that 
 $\varphi_{x/|x|,x_0}(x)=u(|x|)= \text{exp}_{x_{0}} (x)$. Thus, we deduce from \cite[Lemma 3.20]{AT_22-preprint} that
%  we have in particular
%\begin{equation}
%\label{e:retr8.1}
%\frac{\varphi_{\xi,x_0}(r \, \cdot) -x_0}{r} \to id \qquad  \text{in } C^\infty_{loc} (\R^{n}; \R^{n}), \ \text{as } r \searrow 0\,.
%\end{equation}
%Hence, we get that 
\begin{equation}
\label{e:retr4.1}
\frac{\text{exp}_{x_0}(r\, \cdot) -x_0}{r} 
%= \frac{\varphi_{x/|x|,x_0}(rx) -x_0}{r} 
\to id \qquad \text{in }  C^\infty_{loc} (\R^{n} \setminus \{0\}; \R^{n}),\ \text{as } r \searrow 0\,.
\end{equation}
The convergence in \eqref{e:retr4.1} allows us to pass to the inverse maps, and thus write
\begin{equation}
\label{e:retr5.1}
\frac{\text{exp}^{-1}_{x_0}(x_0+r \, \cdot ) }{r} \to id\qquad \text{in }  C^\infty_{loc} (\R^{n} \setminus \{0\}; \R^{n}),\ \text{as } r \searrow 0\,.
\end{equation}
%In particular, we notice that the convergences in~\eqref{e:retr8.1}--\eqref{e:retr5.1} are uniform w.r.t.~$\xi \in \mathbb{S}^{n-1}$.
By definition of~$\phi$ in~\eqref{e:retr12}, \eqref{e:retr5.1} gives exactly~\eqref{e:chi1}.
%\begin{equation}
%\label{e:retr6.1}
%\phi(x_0+r\xi) = \frac{\text{exp}^{-1}_{x_0}(x_0+r\xi) }{|\text{exp}^{-1}_{x_0}(x_0+r\xi) |} \to \xi\qquad  C^\infty(\mathbb{S}^{n-1};\mathbb{S}^{n-1}),\ \text{as } r \to 0^+.
%\end{equation}
%Convergence \eqref{e:chi1} is thus proved.

In order to show \eqref{e:chi2}, let us recall the notation $\xi_{\varphi} (x) = \dot{\varphi}_{\xi,x_0}(P_{\xi,x_0}(x) + (\varphi^{-1}_{\xi,x_0}(x)\cdot \xi)\xi)$ and let us suppose for a moment that we already know that
\begin{equation}
\label{e:retr7.1}
\xi_{\varphi} (x)= \xi + o(|x-x_0|) \qquad \text{for every } \xi \in \mathbb{S}^{n-1}.
\end{equation}
Then, we can write for every $x \in {\rm B}_1(0) \setminus \{0\}$ 
\[
\begin{split}
\frac{x}{r|x|^2}-\frac{\chi(x_0+rx)}{r|x|} &= \frac{x}{r|x|^2}-\frac{\phi(x_0+rx) + o(r|x|)}{r|x|}\\
&= \frac{x}{r|x|^2}-\frac{\exp^{-1}_{x_0}(x_0+rx) }{r|x||\exp^{-1}_{x_0}(x_0+rx)|} +\frac{o(r|x|)}{r|x|},
\end{split}
\]
and from \eqref{e:retr16} we immediately deduce \eqref{e:chi2}. 

It remains to prove \eqref{e:retr7.1}. \RRR In view of Lemma~\ref{l:3.20-24} \EEE we have that for $\xi \in \mathbb{S}^{n-1}$
\begin{equation}
    \label{e:retr9.1}
    \frac{\varphi_{\xi,x_0}(ry+ r t\xi)-x_0}{r} \to y+t\xi \qquad \text{in $C^\infty_{loc} ( \xi^\bot \times \mathbb{R}; \R^{n})$ as $r \searrow 0$.}
\end{equation}
In particular, the convergence in~\eqref{e:retr9.1} is uniform w.r.t.~$\xi \in \mathbb{S}^{n-1}$. The convergence in~\eqref{e:retr9.1} implies that
\begin{align}
    \label{e:retr10.1}
   &  \dot{\varphi}_{\xi,x_0}(ry+ rt \xi) 
    %    \partial_t\frac{\varphi_{\xi,x_0}(ry+ r t\xi)-x_0}{r}
    \to \xi \qquad \text{in } L^\infty_{loc} ( \mathbb{S}^{n-1} \times  \xi^\bot \times \mathbb{R}; \R^{n})\,,
\\
&
    \label{e:retr11.1}
    \partial^\beta_{y_i}\partial^\alpha_t\dot{\varphi}_{\xi,x_0}(ry+ r t \xi) \to 0 \qquad  \text{in } L^\infty_{loc}( \mathbb{S}^{n-1} \times \xi^\bot \times \mathbb{R}; \R^{n})\,,
\end{align}
where $y_i:= y \cdot \eta_i$ and $\{\eta_1,\dotsc,\eta_{n-1}\}$ is an orthonormal basis of $\xi^\bot$ and $\alpha,\beta$ are any positive integers. Again, the convergences in~\eqref{e:retr10.1}--\eqref{e:retr11.1} are uniform w.r.t.~$\xi \in \mathbb{S}^{n-1}$. We can thus fix $\xi \in \mathbb{S}^{n-1}$ and write a Taylor expansion of the form
\begin{align}
\label{e:retr13.1}
    \dot{\varphi}_{\xi,x_0}(ry + r t\xi) = & \ \xi +  \RRR \partial_t\dot{\varphi}_{\xi,x_0}(ry+rt \xi) rt
    \\
    &
    + \sum_{i=1}^{n-1} \partial_{y_i}\dot{\varphi}_{\xi,x_0}(ry+ rt \xi) ry_i + o_\xi(r |(y,t)|)\,, \nonumber
\end{align}
where we used also $ \dot{\varphi}_{\xi,x_0}( 0 )=\xi $. Notice that \RRR $o_\xi( r |(y,t)|)$ \EEE can be taken independent of $\xi \in \mathbb{S}^{n-1}$ because of the uniform convergences \eqref{e:retr10.1}--\eqref{e:retr11.1}. 

We recall that for every $x \in {\rm B}_1(0)$ and $\xi \in \mathbb{S}^{n-1}$
\[
P_{\xi,x_0}(x_0 + rx) =rP_{\xi,x_0,r}(x) \ \ \text{ and } \ \ \varphi^{-1}_{\xi,x_0}(x_0 + rx) =r\varphi^{-1}_{\xi,x_0,r}(x)\,.
\]
 Replacing in \eqref{e:retr13.1} $y, t$ with $P_{\xi,x_0,r}(x),\varphi^{-1}_{\xi,x_0,r}(x) \cdot \xi$, respectively, and using convergence~\eqref{e:retr11.1} we obtain
\begin{equation}
\label{e:retr14.1}
    \xi_{\varphi} (x_0+rx) - \xi
    = o( \RRR r \EEE |(P_{\xi,x_0,r}(x),\varphi^{-1}_{\xi,x_0,r}(x) \cdot \xi)|) \quad \text{for every $ (x, \xi) \in {\rm B}_1(0) \times \mathbb{S}^{n-1}$.} 
\end{equation}
By Lemma~\ref{l:3.20-24} we have that
\begin{align*}
& P_{\xi,x_0,r} \to \pi_\xi, \qquad \varphi^{-1}_{\xi,x_0,r} \to id  \qquad \text{in } C^\infty_{loc} (\R^{n}; \R^{n})\,,
\end{align*}
as $r \searrow 0$ uniformly w.r.t.~$\xi \in \mathbb{S}^{n-1}$. \RRR Hence, \EEE we can rewrite~\eqref{e:retr14.1} as
\begin{equation*}
    \xi_{\varphi} (x_0+rx) - \xi
    = o(\RRR r \EEE |x|) \qquad \text{for } \xi \in \mathbb{S}^{n-1}\,,
\end{equation*}
which implies~\eqref{e:retr7.1} and completes the proof of~\eqref{e:chi2}.

Finally, \eqref{e:chi3} follows directly from \eqref{e:retr11}.
\end{proof}

\section{Structure properties of $GBD_F(\Omega)$ and $GBD(\rm{M})$}
\label{s:structure}

\RRR
In this section we study the structure of functions in~$GBD({\rm M})$ and in~$GBD_{F} (\Om)$. In particular, in Sections~\ref{s:jump-u}--\ref{s:approximate-sym-gradient} we  show that under the assumptions~\eqref{hp:F} and~\eqref{hp:F2} on~$F$ the jump set of a function $u \in GBD_F(\Omega)$ can be sliced into the $0$-dimensional jump set of suitable one dimensional slices of~$u$ (see Theorem~\ref{p:euju}). In Theorem~\ref{t:apsym} we show that if~$F$ satisfies~\eqref{hp:F}, a function $u \in GBD_{F} (\Om)$ admits an approximate symmetric gradient. Relying on such results, we show in Sections~\ref{s:jump-M}--\ref{s:appr-sym-M} that $\omega \in GBD({\rm M})$ satisfies the same properties on the Riemannian manifold~$({\rm M}, g)$.
\EEE

\subsection{Rectifiability of the jump and its one dimensional slices in $GBD_{F}(\Om)$} 
\label{s:jump-u}

Throughout this section we assume that $\Omega$ is an open subset of $\mathbb{R}^n$ and we fix $F \in C^{\infty} ( \R^{n} \times \R^{n}; \R^{n})$ fulfilling~\eqref{hp:F} and \eqref{hp:F2}. Furthermore, we rely on the notation introduced in Section~\ref{s:curvilinear}.
Here we present a fundamental property of the jump set (cf.~\cite{del}).

\begin{theorem}
\label{t:delnin}
Let $u \colon \Omega \to \mathbb{R}^n$ be measurable. Then,~$J_u$ is countably $(n-1)$-rectifiable.
\end{theorem}

\RRR We define the directional jump set of a measurable function. \EEE

\begin{definition}
\label{d:Au}
Let $(P_{\xi})_{\xi \in \mathbb{S}^{n-1}}$ be a family of curvilinear projections on~$\Om$ and let $u \colon \Omega \to \mathbb{R}^m$. Given $\xi \in \mathbb{S}^{n-1}$ we define the {\it directional jump set}~as $J_{\hat u_\xi}:=\{x \in \Omega :\,  t^{\xi}_{x} \in J_{\hat{u}^\xi_{P_\xi(x)}}\}$. We further define $A_{\hat{u}} := \{ (x , \xi ) \in \Om \times \mathbb{S}^{n-1}: \, x \in J_{\hat{u}_{\xi}}\}$. \EEE
%\RRR Was it not $t_{x}$ here? Also later on there is $x$ now instead of~$t_{x}$. Am I missing something? \EEE
%and its truncated version
%\begin{equation*}
%S_{\hat u_\xi}:=\{x \in \Omega \ | \ x \in J_{\hat{u}^\xi_{P_\xi(x)}} \text{ and } \, |[\hat{u}^\xi_{P_\xi(x)}]|>1 \}
%\end{equation*}
%Moreover we define the directional jump set $A_u$ of $u$ as
%\begin{equation*}
 %   A_u := \{(x,\xi) \in \mathbb{R}^n \times \mathbb{S}^{n-1} \ | \ x \in S_{\hat u_\xi}  \},
%\end{equation*}
%and its truncated version
%\begin{equation*}
 %   A^1_u := \{(x,\xi) \in \mathbb{R}^n \times \mathbb{S}^{n-1} \ | \ x \in S^1_{\hat u_\xi}  \}.
%\end{equation*}
\end{definition}

\RRR We now introduce a family of Borel regular measures $\{\eta_{\xi}\}_{\xi \in \mathbb{S}^{n-1}}$ depending on a direction~$\xi$ and on the jump set of~$\hat{u}^{\xi}_{y}$ defined w.r.t.~a family of curvilinear projection~$\{P_{\xi}\}_{\xi \in \mathbb{S}^{n-1}}$. \EEE

\begin{definition}
\label{d:defiu}
Let $u \in GBD_{F} (\Om)$, let $(P_{\xi})_{\xi \in \mathbb{S}^{n-1}}$ be a family of curvilinear projections on an open subset~$U$ of~$\Omega$, and let $p \in [1, +\infty]$. For every $\xi \in \mathbb{S}^{n-1}$ consider the Borel regular measure $\eta_\xi$ of $\mathbb{R}^n$ given by
\begin{align}
    \label{e:defeta1}
    \eta_\xi(B) & := \int_{\xi^\bot} \sum_{t \in (B \cap J_{\hat u_\xi})^\xi_y} \big( |[\hat{u}^\xi_y(t)]| \wedge 1 \big) \, \di \mathcal{H}^{n-1}(y) \qquad  B \in \mathcal{B} ( \mathbb{R}^n) \,,
    \\
    \label{e:defeta2}
    \eta_\xi(E) & := \inf \, \{\eta_\xi(B) : \, E \subseteq B, \ B \in \mathcal{B} (\Om)\}\,.
\end{align}
We recall that the integral in~\eqref{e:defeta1} is well defined thanks to~\cite[Lemma~4.5]{AT_22-preprint}. For $B \in \mathcal{B} (\Om)$ and~$\xi \in \mathbb{S}^{n-1}$, we further set $f_{B} (\xi) := \eta_{\xi} (B)$ and $\zeta_{p} (B) := \| f_{B}\|_{L^{p} (\mathbb{S}^{n-1})}$. Via the classical Carath\'eodory's construction we define the measure
\begin{equation*}
\mathscr{I}_{u,p} (E) := \sup_{\delta>0} \, \inf_{G_{\delta}} \sum_{B \in G_{\delta}} \zeta_{p} (B) \qquad \text{for $E \subseteq \Om$,}
\end{equation*}
where~$G_{\delta}$ is the family of all countable Borel coverings of $E$ made of sets having diameter less than or equal to~$\delta$.
\end{definition}

\begin{definition}
\label{d:defI2}
In the same setting of Definiton~\ref{d:defiu}, we define
\begin{align*}
\hat{\zeta}(A)& := \int_{\mathbb{S}^{n-1}} \eta_{\xi}(A_{\xi}) \, \di \mathcal{H}^{n-1}(\xi) \qquad \text{for every $A \in \mathcal{B} (\Om\times \mathbb{S}^{n-1})$}\,,\\
\hat{\mathscr{I}}_{u} (\RRR D \EEE )&:= \sup_{\delta>0} \, \inf_{G_{\delta}} \sum_{B \in G_{\delta}} \hat\zeta (B) \qquad \text{for $\RRR D \EEE \subseteq \Om \times \mathbb{S}^{n-1}$,}
\end{align*}
where $G'_{\delta}$  is the family of all countable Borel coverings of $F$ made of sets having diameter less than or equal to~$\delta$.
\end{definition}

\RRR
In particular, we have the following representation for~$\mathscr{I}_{u,p}$ and $\hat{\mathscr{I}}_{u}$ (see also \cite[Proposition~2.9]{AT_22-preprint}).

\begin{proposition}
\label{p:coincidence}
In the setting of Definitions~\ref{d:defiu} and~\ref{d:defI2}, the measures $\mathscr{I}_{u,1}$ and $\hat{\mathscr{I}}_{u}$ satisfy
\begin{align*}
    \mathscr{I}_{u, 1} (E) &= \inf_{\substack{E \subseteq B \\ B \in \mathcal{B}(\Om)}} \int_{\mathbb{S}^{n-1}} \eta_\xi(B) \, \di \mathcal{H}^{n-1}(\xi) \qquad \text{for every  $E \subseteq \Omega$\,,} \\
    \hat{\mathscr{I}}_{u}( D ) &= \inf_{\substack{F \subseteq A \\ A \in \mathcal{B} ( \Om \times \mathbb{S}^{n-1}) }} \int_{\mathbb{S}^{n-1}} \eta_\xi(A_\xi) \, \di \mathcal{H}^{n-1}(\xi) \qquad  \text{for every  $ D  \subseteq \Omega \times \mathbb{S}^{n-1}$}.
\end{align*}
\end{proposition}
\EEE

%\RRR Assuming \EEE the finiteness of~$\mathscr{I}_{u,p}$, we want to derive the $(n-1)$-rectifiability of $\mathscr{I}_{u,1}$. In order to apply the rectifiability criterion for integralgeometric measures, we have to guarantee that $\eta_\xi$ is concentrated on some set $E$ which is $\sigma$-finite with respect to condition \eqref{e:condigm6}. 
%For this purpose we introduce the \emph{one dimensional radial oscillation of~$u$ around~$x$}. \RRR (measurability of the following integral?) \EEE 

We now show that the jump set of a vector field $u \in GBD_F(\Omega)$ can be sliced into the $0$-dimensional jump set of suitable one dimensional slices of $u$. More precisely, we prove the following refined version of~\cite[Theorem~1.1 and Corollary~1.2]{AT_22-preprint} for a $GBD_F$-vector field.

\begin{theorem}[Slicing of the jump set]
\label{p:euju}
 Let $F \in C^{\infty}(\R^{n} \times \R^{n}; \R^{n})$ satisfy~\eqref{hp:F}--\eqref{hp:F2}, let $\Om$ be an open subset of~$\R^{n}$, let $u \in GBD_{F}(\Omega)$, and let $(P_{\xi})_{\xi \in \mathbb{S}^{n-1}}$ be a family of curvilinear projections on an open subset~$U$ of~$\Om$. Then, it holds true
 \begin{align}
 \label{e:neuju1-u}
 J_{\hat u^\xi_y} &= (J_{u_\xi})^{\xi}_{y} \qquad  \text{for every $\xi \in \mathbb{S}^{n-1}$, for $\mathcal{H}^{n-1}$-a.e. $y \in \xi^\bot$},\\
 \label{e:neuju2-u}
 J_{\hat u^\xi_y} &= (J_{u})^{\xi}_{y} \qquad  \text{for $\mathcal{H}^{n-1}$-a.e. $\xi \in \mathbb{S}^{n-1}$, for $\mathcal{H}^{n-1}$-a.e. $y \in \xi^\bot$},\\
 \label{e:neuju3-u}
 J_{\hat u^\xi_y} &\subseteq (J_{u})^{\xi}_{y} \qquad  \text{for every $\xi \in \mathbb{S}^{n-1}$, for $\mathcal{H}^{n-1}$-a.e. $y \in \xi^\bot$}.
 \end{align}
 Moreover, the following relation between traces holds true for every $\xi \in \mathbb{S}^{n-1}$, for $\mathcal{H}^{n-1}$-a.e. $y \in \xi^\bot$, and for every $t \in (J_{u_\xi})^\xi_y$
 \begin{equation}
 \label{e:euju5-u}
 \aplim_{\substack{z \to x \\ \pm( z - x) \cdot \nu_{u_\xi}(x) >0}} u_{\xi}(z)  = \aplim_{s \to t^{\pm\sigma(x)}}\hat{u}^{\xi}_{y}(s)\,, 
\end{equation}
whenever $x=\varphi_\xi(y+t\xi) \in U $ and $\nu_{u_\xi} \colon J_{u_\xi} \to \mathbb{S}^{n-1}$ is a Borel measurable orientation ($J_{u_\xi}$ is countably $(n-1)$-rectifiable thanks to Theorem \ref{t:delnin}). 
\end{theorem}

Before proving Theorem~\ref{p:euju}, we state the equivalent of \cite[Corollary~1.2]{AT_22-preprint} for $u \in GBD_{F}(\Om)$. 
%We recall the notation of $\mathscr{I}_{u, p}$ introduced in Definition~\ref{d:defiu} and of ${\rm Osc}_{r} (u, x, \rho)$ given in Definition~\ref{d:oscillation}. \EEE

\begin{theorem}
\label{t:rectiu1}
Let $F \in C^{\infty}(\R^{n} \times \R^{n}; \R^{n})$ satisfy~\eqref{hp:F}--\eqref{hp:F2}, let $\Om$ be an open subset of~$\R^{n}$, let $u \in GBD_F(\Omega)$, let $U\subseteq \Om$ open, and let $(P_{\xi})_{\xi \in \mathbb{S}^{n-1}}$ be a family of curvilinear projections on~$U$. Then, it holds that
\begin{equation}
    \label{e:keycontgbd}
    (J_{\hat{u}_\xi})_y^\xi = (J_{u_{\xi}})^\xi_y \qquad \text{for $\mathcal{H}^{n-1}$-a.e.~$\xi \in \mathbb{S}^{n-1}$, for $\mathcal{H}^{n-1}$-a.e.~$y \in \xi^\bot$}.
\end{equation}
\end{theorem}

\begin{proof}
The thesis follows from \cite[Corollary~1.2]{AT_22-preprint} with the choice $g(x, z) = z$ for $(x, z) \in \Om \times \R^{n}$. Indeed, we notice that~\cite[Condition (G.2)]{AT_22-preprint} is automatically satisfied, while~\cite[Corollary~1.2, item (3)]{AT_22-preprint} is automatically satisfied in view of Definition~\ref{d:GBD} and Corollary~\ref{c:DM-3.5}.
\end{proof}

 We now prove an intermediate lemma which allows us to pass from an a.e.~condition on $\xi$ to the entire~$\mathbb{S}^{n-1}$ in~\eqref{e:keycontgbd}. 

% \begin{definition}[Jump set in the direction~$\xi$]
% \label{d:jump-direction-xi}
% Let~$\Om$ be an open subset of~$\R^{n}$, let $(P_{\xi})_{\xi \in \mathbb{S}^{n-1}}$ be a parametrized family of transversal maps on~$\Om$. Given $u \colon \Omega \to \mathbb{R}^n$ and $\xi \in \mathbb{S}^{n-1}$, we define the {\em jump set in the direction}~$\xi$ as
% \begin{equation*}
%     J_{\hat u_\xi} := \Big\{x \in \Omega : \, t_x \in J_{\hat{u}^\xi_{P_\xi(x)}}  \Big\}.
% \end{equation*}
% \end{definition}

\begin{lemma}
\label{l:everyxi}
 Let $F \in C^{\infty}(\R^{n} \times \R^{n}; \R^{n})$ satisfy~\eqref{hp:F}--\eqref{hp:F2}, let $\Om$ be an open subset of~$\R^{n}$, let $u \in GBD_F(\Omega)$, let $U\subseteq \Om$ open, and let $(P_{\xi})_{\xi \in \mathbb{S}^{n-1}}$ be a family of curvilinear projections on~$U$. If $B \in \mathcal{B}(U)$ satisfies $\mathscr{I}_{u,1}(B)=0$, then we have 
 \begin{equation}
     \label{e:everyxi}
     \mathcal{H}^{n-1} \big( P_\xi(J_{\hat u_\xi} \cap B) \big) = 0 \qquad  \text{ for every } \xi \in \mathbb{S}^{n-1}.
 \end{equation}
\end{lemma}

\begin{proof}
In order to simplify the notation we assume that $U = \Om$. From the definition of \RRR $\mathscr{I}_{u,1}$ \EEE we have that condition $\mathscr{I}_{u,1}(B)=0$ implies that there exists $N \subseteq \mathbb{S}^{n-1} $ with~$\mathcal{H}^{n-1}(N)=0$ such that for every $\xi \in \mathbb{S}^{n-1} \setminus N$ it holds
\[
\mathcal{H}^{n-1} \big( P_\xi(J_{\hat u_\xi} \cap B)\big) = 0\,.
\]
 If we define the measure $\mu_u$ as
 \begin{equation}\label{e:mi-u-def}
\mu_{u} (\RRR B' \EEE ) := \sup_{k \in \mathbb{N}} \, \sup \sum_{i=1}^{k} \mu^{\xi_{i}}_{u} (B_{i})
\end{equation}
for every $\RRR B' \EEE \in \mathcal{B}(\Om)$, where the second supremum is taken over all the families $\xi_{1}, \ldots \xi_{k} \in \mathbb{S}^{n-1} \setminus N$ and all the families of pairwise disjoint Borel subset~$B_{1}, \ldots, B_{k}$ of~$\RRR B' \EEE$, we have $\mu_u(B)=0$. Since for every $\xi \notin N$ we have by construction $\mu^{\xi}_{u} \leq \tilde{\mu}_u$, Proposition~\ref{p:xi-lsc} implies that for every open set $U \subseteq \Omega$ and every $\xi \in \mathbb{S}^{n-1}$
\begin{equation}
\label{e:muumuxi}
\mu^\xi_{u} (U) \leq \mu_u(U).
\end{equation}
Both measures appearing in \eqref{e:muumuxi} are Radon, hence the outer regularity of Radon measures implies that for every $B' \in \mathcal{B}( \Omega) $ and for every $\xi \in \mathbb{S}^{n-1}$
\begin{equation*}
\mu^\xi_{u} (B') \leq \mu_u(B')\,.
\end{equation*}
The previous inequality computed in $B$ implies that $\mu^\xi_{u}(B)=0$ for every $\xi \in \mathbb{S}^{n-1}$. Denoting by~$v$ a Borel representative of~$u$, we write
\[
\begin{split}
0=\mu^\xi_{u} (B \cap J_{\hat v_\xi})& = \int_{\xi^\bot} |D\hat{u}^\xi_y| \big( B^{\xi}_{y} \cap (J_{\hat v_\xi})_y^\xi \setminus J^1_{\hat{u}^\xi_y} \big) + \mathcal{H}^0 \big(B^{\xi}_{y} \cap (J_{\hat v_\xi})_y^\xi \cap J^1_{\hat{u}^\xi_y})\, \di \mathcal{H}^{n-1}(y)\\
&= \int_{P_\xi(J_{\hat v_\xi} \cap B)} m(y) \, \di \mathcal{H}^{n-1}(y) \,,
\end{split}
\]
where $m(y) >0$ for $\mathcal{H}^{n-1}$-a.e.~$y \in P_\xi(J_{\hat v_\xi} \cap B)$. We deduce $\mathcal{H}^{n-1}(P_\xi(J_{\hat v_\xi} \cap B))=0$ from which \eqref{e:everyxi} follows. 
\end{proof}

We are now in a position to prove Theorem~\ref{p:euju}.
%that the slices of  the jump set~$J_{\hat{u}_{\xi}}$ in the direction~$\xi$ (see Definition~\ref{d:Au}) and of the slices of the jump set~$J_{u_{\xi}}$ of~$u_{\xi}$ coincide.

%\RRR S: I would have said that $(J_{u_{\xi}})^{\xi}_{y} = J_{\hat{u}^{\xi}_{y}}$. Formula~\eqref{e:euju1.1.3} should say that $R^{\xi}_{y} \subseteq J_{\hat{u}^{\xi}_{y}} \subseteq (J_{u})^{\xi}_{y}$. Then, the same formula for $J_{u}$ should say that $(J_{u_{\xi}})^{\xi}_{y} \supseteq J_{\hat{u}^{\xi}_{y}}$. Is it right? \EEE

\begin{proof}[Proof of Theorem~\ref{p:euju}]
In order to simplify the notation we assume that $U = \Om$. \RRR Formula~\eqref{e:euju5-u} \EEE holds by \cite[Proposition~4.10]{AT_22-preprint}, since $J_{u_{\xi}}$ is countably $(n-1)$-rectifiable \RRR and $D_{\xi} (u_{\xi} \circ \varphi_{\xi}) \in \mathcal{M}_{b} (\varphi^{-1}_{\xi} (\Om))$ (see Corollary~\ref{c:DM-3.5}). \EEE Theorem~\ref{t:rectiu1} \RRR yields \eqref{e:neuju1-u}--\eqref{e:neuju3-u}, \EEE respectively, but for $\mathcal{H}^{n-1}$-a.e.~$\xi \in \mathbb{S}^{n-1}$. In order to pass from $\HH^{n-1}$-a.e.~$\xi$ to the entire~$\mathbb{S}^{n-1}$, we notice that \RRR the inclusion
\begin{displaymath}
 J_{\hat u^\xi_y} \subseteq (J_{u})^{\xi}_{y} \qquad  \text{for $\HH^{n-1}$-a.e.~$\xi \in \mathbb{S}^{n-1}$, for $\mathcal{H}^{n-1}$-a.e. $y \in \xi^\bot$}
\end{displaymath}
leads to $\mathscr{I}_{u,1}(\Omega \setminus J_u )=0$. \EEE This last information allows us to make use of Lemma~\ref{l:everyxi} and infer $\mathcal{H}^{n-1}(P_\xi(J_{\hat{u}_\xi} \cap (\Omega \setminus J_u)))=0$ for every $\xi \in \mathbb{S}^{n-1}$. Using the identity $(J_{\hat{u}_\xi})^\xi_y= J_{\hat{u}^\xi_y}$ we immediately infer the validity \RRR of \eqref{e:neuju3-u}. In order to prove~\eqref{e:neuju1-u} \EEE we argue as above with $J_u$ replaced by $J_{u_\xi}$ and infer that
\begin{equation*}
    %\label{e:neuju1.1}
    J_{\hat u^\xi_y} \subseteq (J_{u_\xi})^{\xi}_{y} \qquad   \text{for every $\xi \in \mathbb{S}^{n-1}$, for $\mathcal{H}^{n-1}$-a.e.~$y \in \xi^\bot$}.
\end{equation*}
The opposite inclusion is a direct consequence of \eqref{e:euju5-u}.
\end{proof}

\subsection{The approximate symmetric gradient in $GBD_{F}(\Om)$}
\label{s:approximate-sym-gradient}

%In this section we suppose that $F \colon \mathbb{R}^n \times \mathbb{R}^n \to \mathbb{R}^n$ is of the form
%\[
%F_p(x,v):= \sum_{i,j=1}^n \Gamma^p_{ij}(x)v_i \cdot v_j, \ \ p \in \{1, \dotsc, n \}
%\]
%where $\Gamma \colon \Omega \to \R^{n \times n \times n}$ is $C^\infty$-regular and its $(p, i, j) $-entry is symmetric with respect to $i, j$, namely 
%\begin{equation}
%\label{e:symm}
%\Gamma^p_{ij}(x)=\Gamma^p_{ji}(x), \ \ p,i,j \in \{1,\dotsc,n\}, \ x\in \Omega.
%\end{equation}

\RRR In the next theorem \EEE we show that every function~$u \in GBD_{F}(\Om)$ admits an approximate symmetric gradient at a.e.~$x \in \Om$.
We recall the definition of approximate symmetric gradient.

\begin{theorem}[Existence of the approximate symmetric gradient]
\label{t:apsym}
Let $F \in C^{\infty} (\R^{n} \times \R^{n}; \R^{n})$ satisfy~\eqref{hp:F}, let $\Om$ be an open subset of~$\R^{n}$, and let $u \in GBD_{F}(\Omega)$. Then, there exists $e(u) \in L^1(\Om;\mathbb{M}^{n \times n}_{sym})$ such that, setting
 \begin{equation}
\label{e:apsym1.3}
\tilde{e}(u)(x) \zeta \cdot \zeta := e(u)(x)\zeta \cdot \zeta+ u(x) \cdot F(x, \zeta) \qquad \text{for $x \in \Om$ and $\zeta \in \mathbb{R}^n$},
\end{equation}
$\tilde{e}(u)(x)$ is the approximate symmetric gradient of $u$ at $x$ for a.e.~$x \in \Om$ and
\begin{equation}
\label{e:apsym1.3.99}
\int_B |e(u)|\, \RRR \di x \EEE \leq \lambda(B), \qquad \text{ for every } B \subset \Omega \text{ Borel}.
\end{equation}
Moreover, if $(P_{\xi})_{\xi \in \mathbb{S}^{n-1}}$ is a family of curvilinear projections on an open set~$U \subseteq \Om$, then for $\mathcal{H}^{n-1}$-a.e.~$\xi \in \mathbb{S}^{n-1}$ it holds true
\begin{equation}
    \label{e:apsym1.4}
    \nabla \hat{u}_y^\xi (t) = (e(u))^{\xi}_{y} \, \dot{\varphi}_{\xi} (y + t\xi) \cdot \dot{\varphi}_{\xi} (y + t\xi) \quad \text{for $\mathcal{H}^{n-1}$-a.e. }y \in \xi^\bot, \text{ for a.e. } t \in U_y^\xi\,.
\end{equation}
% There exists $q \colon \Omega \times \mathbb{R}^n \to \mathbb{R}$ with $q(x,tz)=t^2q(x,z)$ for every $(x,z,t) \in \Omega \times \mathbb{R}^n \times \mathbb{R}$ such that if we define $q_\xi(x):=q(x,\xi(x))$ then for every $\xi \in \mathbb{S}^{n-1}$ we have $q_\xi \in L^1(\Omega)$ and
%\begin{equation}
%\label{e:apsym1.2}
%    \nabla \hat{u}_y^\xi (t) = (q_\xi)_y^\xi (t) \qquad \text{for $\mathcal{H}^{n-1}$-a.e. }y \in \xi^\bot, \text{ for a.e. } t \in \Omega_y^\xi.
%\end{equation}
\end{theorem}

\begin{remark}
\label{r:appsymi}
We notice that formula \eqref{e:apsym1.3} together with the uniqueness of the symmetric approximate gradient tell us that $e(u)$ \emph{does not depend on the chosen family of curvilinear projections}.
\end{remark}

Before proving Theorem~\ref{t:apsym} we need \RRR two intermediate results. \EEE

\begin{lemma}
\label{l:quadratic}
Let $u \colon \Omega \to \mathbb{R}^n$ be measurable and \RRR with compact support in~$\Om$. \EEE Suppose that there exists $\tilde{q} \colon \Omega \times \mathbb{R}^n \to \mathbb{R}$ which \RRR is \EEE $2$-homogeneous in the second variable and satisfies for a.e. $x \in \Omega$
\begin{equation}
    \label{e:apsym15.99}
    \aplim_{z \to x} \frac{|(u(z) - u(x) ) \cdot (z - x) - \tilde{q}(x, z - x ) |}{| z - x |^2}=0\,.
\end{equation}
Then there exists a measurable map $\RRR S(u) \EEE \colon \Omega \to \mathbb{M}^{n \times n}_{sym}$ such that for a.e.~$x \in \Omega$ 
\begin{equation}
\label{e:apsym23}
\RRR S(u) \EEE (x)\zeta \cdot \zeta = \tilde{q}(x, \zeta) \qquad \text{for a.e.~$\zeta \in \mathbb{R}^n$.}
\end{equation}
\end{lemma}

\begin{proof}
%  , up to consider a larger negligible set in~$\Om$, still denoted by~$N$, we have the following parallelogram law \RRR I think this has to be: 
%for some $S \subseteq \R^{n} \times \R^{n}$ with~$\mathcal{L}^{2n} (\R^{n} \times \R^{n} \setminus S)=0$.
% First we notice that without loss of generality we may suppose that $u$ has compact support in $\Omega$. 
We claim that there exists $S \subseteq \R^{n} \times \R^{n}$ with $\mathcal{L}^{2n} (\R^{n} \times \R^{n} \setminus S)=0$ and such that for every $(\zeta_{1}, \zeta_{2}) \in S$ the parallelogram law
\begin{equation}
\label{e:apsym16}
\tilde{q}(x , \zeta_{1} + \zeta_{2}) + \tilde{q}(x, \zeta_{1} - \zeta_{2}) = 2 \tilde{q}(x, \zeta_{1}) + 2\tilde{q}(x, \zeta_{2}), \ \ (\zeta_{1},\zeta_{2}) \in S, %\ x \in \Omega \setminus N\,,
\end{equation}
 holds for a.e.~$x \in \Om$ (depending on~$(\zeta_{1} , \zeta_{2})$). Indeed an application of the dominated convergence theorem together with \eqref{e:apsym15.99} and the $2$-homogeneity of $\tilde{q}(x,\cdot)$ \RRR allows \EEE us to write for every $\lambda>0$
\begin{align}
\label{e:apsym1004}
\lim_{r \searrow 0} & \,\lambda^n \int_{{\rm B}_{\frac{1}{\lambda}}(0)} \bigg( \int_{\Omega} \bigg|\frac{(u(x + r\lambda z)-u(x)) \cdot \frac{z}{|z|}}{r\lambda|z|} -\tilde{q}\bigg(x,\frac{z}{|z|}\bigg)\bigg| \wedge 1 \, \di x\bigg) \di z
\\
&
\leq \omega_{n} \lim_{r \searrow 0} \int_\Omega \bigg( \mint_{{\rm B}_r(x)} \frac{|(u(z) - u(x)) \cdot ( z - x ) -\tilde{q}(x, z-x)|}{|z - x|^2} \wedge 1 \, \di z\bigg) \di x = 0\,, \nonumber
\end{align}
where $\omega_{n}:= \mathcal{L}^{n} ({\rm B}_{1}(0))$.
%or equivalently
%\[
%\lim_{r \to 0^+} \lambda^n \int_{{\rm B}_{\lambda^{-1}}(0)} \bigg( \int_{\Omega} \bigg|\frac{(u(x_0 +r\lambda x)-u(x_0)) \cdot \frac{x}{|x|}}{r\lambda|x|} -\tilde{q}\bigg(x_0,\frac{x}{|x|}\bigg)\bigg| \wedge 1 \, \di x_0\bigg) \di x=0,
%\]
%where $\lambda$ is any strictly positive real number.
Let us fix $\lambda_{m} \searrow 0$. It follows from~\eqref{e:apsym1004} and a diagonal argument that there exists a sequence $r_{\ell} \searrow 0$ such that for every $m \in \mathbb{N}$ it holds
%This implies that, up to pass through a subsequence $r_j \to 0$ as $j \to \infty$, we have for $\mathcal{L}^n$-a.e.~$z \in {\rm B}_{\frac{1}{\lambda}}(0)$
%\begin{equation}
%\label{e:apsym7}
%\lim_{j \to \infty} \int_{\Omega} \bigg|\frac{(u(x + r^\lambda_j z) - u(x)) \cdot \frac{z}{|z|}}{r_j^\lambda|z|} -\tilde{q}\bigg(x, \frac{z}{|z|}\bigg)\bigg| \wedge 1 \, \di x = 0\,,
%\end{equation}
%where we have set $r^\lambda_j:= \lambda r_j$. Consider a sequence $\lambda_m \searrow 0$. From \eqref{e:apsym7}, using a diagonal argument we can find a subsequence $r_l \searrow 0$ such that given $m=1,2,\dotsc$ we have for $\mathcal{L}^n$-a.e. $x \in {\rm B}_{\lambda_m^{-1}}(0)$ 
\begin{equation}
\label{e:apsym7.1}
\lim_{\ell \to \infty} \int_{\Omega} \bigg|\frac{(u(x +r_\ell z) - u(x)) \cdot \frac{z}{|z|}}{r_\ell |z|} - \tilde{q}\bigg(x,\frac{z}{|z|}\bigg) \bigg| \wedge 1 \, \di x =0 \ \  \text{for a.e.~$z \in {\rm B}_{\frac{1}{\lambda_{m}}} (0)$.}
\end{equation}
The arbitrariness of $m \in \mathbb{N}$ in~\eqref{e:apsym7.1} yields
\begin{equation}
\label{e:apsym7.2}
\lim_{\ell \to \infty} \int_{\Omega} \bigg|\frac{(u(x + r_\ell z) - u(x) ) \cdot \frac{z}{|z|}}{r_\ell |z|} - \tilde{q}\bigg(x,\frac{z}{|z|}\bigg)\bigg| \wedge 1 \, \di x = 0\quad \text{for a.e.~$z \in \mathbb{R}^n$.}
\end{equation}

Let us set 
\begin{align*}
\RRR K \EEE & := \{ z \in \R^{n}: \, \text{\eqref{e:apsym7.2} does not hold in~$z$}\} \,,
\\
 S  & := \{ (\zeta_{1}, \zeta_{2}) \in \R^{n} \times \R^{n}: \, \zeta_{1}\pm\zeta_{2} , 2\zeta_{1}, 2\zeta_{2} \notin \RRR K \EEE \}\,.
\end{align*}
As $\mathcal{L}^{n}(\RRR K \EEE )=0$, by Fubini's theorem we have that $\mathcal{L}^{2n} (\R^{n} \times \R^{n} \setminus S) = 0$. Let us fix $(\zeta_{1}, \zeta_{2} ) \in S$. We make use of the parallelogram identity to write for every $x \in \Om$ and for $r_{\ell}$ such that $x \pm r_{\ell} \zeta_{i} \in \Om$ for $i =1, 2$ 
\begin{align}
\label{e:apsym1005}
   &|2\tilde{q}(x, \zeta_{1} + \zeta_{2} ) + 2 \tilde{q}(x ,\zeta_{1} - \zeta_{2}) - 4\tilde{q}(x, \zeta_{1}) - 4\tilde{q}(x,\zeta_{2} ) |\wedge 1 \\
   &\leq \bigg(| \zeta_{1} + \zeta_{2} |^2\bigg|\frac{(u(x + r_\ell \zeta_{1} ) - u (x - r_\ell \zeta_{2} )) \cdot \frac{\zeta_{1} + \zeta_{2}}{| \zeta_{1} + \zeta_{2} | }}{r_\ell | \zeta_{1} + \zeta_{2}|} - \tilde{q}\bigg(x ,\frac{\zeta_{1} + \zeta_{2} }{| \zeta_{1} + \zeta_{2} |}\bigg)\bigg|\bigg)\wedge 1 \nonumber \\
   &\qquad + \bigg(|\zeta_{1} + \zeta_{2} |^2\bigg|\frac{(u(x + r_\ell \zeta_{2} ) - u(x - r_\ell \zeta_{1})) \cdot \frac{\zeta_{1} + \zeta_{2}}{|\zeta_{1} + \zeta_{2} |}}{r_\ell |\zeta_{1}+ \zeta_{2} |} - \tilde{q}\bigg(x,\frac{\zeta_{1}  +\zeta_{2} }{|\zeta_{1} + \zeta_{2} |}\bigg)\bigg|\bigg)\wedge 1 \nonumber \\
   &\qquad + \bigg(|\zeta_{1} - \zeta_{2} |^2\bigg|\frac{(u(x + r_\ell \zeta_{1}) - u(x + r_\ell \zeta_{2} ))  \cdot \frac{\zeta_{1} - \zeta_{2} }{|\zeta_{1} - \zeta_{2} |}}{r_\ell | \zeta_{1} - \zeta_{2} |} - \tilde{q}\bigg(x,\frac{\zeta_{1} - \zeta_{2}}{| \zeta_{1} - \zeta_{2} |}\bigg)\bigg|\bigg)\wedge 1 \nonumber \\
   &\qquad + \bigg(|\zeta_{1} - \zeta_{2} |^2\bigg|\frac{(u(x - r_\ell \zeta_{1} ) - u(x - r_\ell \zeta_{2} ) )   \cdot \frac{\zeta_{1} - \zeta_{2} }{|\zeta_{1} - \zeta_{2} |}}{r_\ell | \zeta_{1} - \zeta_{2} | } - \tilde{q}\bigg(x , \frac{\zeta_{1}  - \zeta_{2} }{| \zeta_{1} - \zeta_{2} |}\bigg)\bigg|\bigg)\wedge 1 \nonumber \\
   &\qquad + \bigg(4|\zeta_{1} |^2\bigg| \frac{(u(x + r_\ell \zeta_{1}) - u(x - r_\ell \zeta_{1})) \cdot \frac{\zeta_{1}}{|\zeta_{1}|}}{2 r_\ell |\zeta_{1}|} - \tilde{q}\bigg(x , \frac{\zeta_{1}}{|\zeta_{1}|}\bigg)\bigg|\bigg)\wedge 1 \nonumber \\
   &\qquad + \bigg( 4 | \zeta_{2} |^2\bigg| \frac{ ( u (x + r_\ell \zeta_{2}) - u(x - r_\ell \zeta_{2}) )  \cdot \frac{\zeta_{2}}{|\zeta_{2}|}}{2 r_{\ell} |\zeta_{2} | } - \tilde{q}\bigg(x , \frac{\zeta_{2}}{|\zeta_{2} | } \bigg)\bigg|\bigg)\wedge 1\,. \nonumber
\end{align}
In particular, we notice that since~$u$ has compact support, the restriction on~$r_{\ell}$ can be made independent of~$x \in \Om$. As~$(\zeta_{1}, \zeta_{2}) \in S$, by integrating~\eqref{e:apsym1005} w.r.t.~$x \in \Om$ and using~\eqref{e:apsym7.2} on each term on the right-hand side of~\eqref{e:apsym1005} we deduce that
\begin{equation*}
   % \label{e:apsym17}
\int_{\Omega} |2\tilde{q}(x , \zeta_{1} + \zeta_{2} ) + 2\tilde{q}(x, \zeta_{1} - \zeta_{2}) - 4 \tilde{q}(x , \zeta_{1}) - 4\tilde{q}(x , \zeta_{2} )|\wedge 1 \, \di x =0\,,
\end{equation*}
which in turn implies~\eqref{e:apsym16}. 
%that for $\mathcal{L}^n$-a.e. $x \in \Omega$
%\begin{equation}
%    \label{e:apsym18}
%\tilde{q}(x , \zeta_{1} + \zeta_{2}) + \tilde{q}(x, \zeta_{1} - \zeta_{2}) =  2\tilde{q}(x , \zeta_{1}) + 2\tilde{q}(x,\zeta_{2}) \qquad \text{for every $(\zeta_{1}, \zeta_{2}) \in S$.} 
%\end{equation}
We notice that the set of admissible pairs~$(\zeta_{1}, \zeta_{2})$ in~\eqref{e:apsym16} is independent of $x \in \Om$. 

We now claim that there exists a vector subspace $X$ over~$\mathbb{Q}$ which is countable and dense in~$\R^{n}$, fulfills $X \setminus\{0\} \subset \mathbb{R}^n \setminus \RRR K \EEE$, and such that the following hold: for every $\zeta_{1}, \zeta_{2} \in X$ and for \RRR a.e.~$x \in \Om$ \EEE
\begin{equation}
\label{e:apsym16.99}
\tilde{q}(x , \zeta_{1} + \zeta_{2}) + \tilde{q}(x, \zeta_{1} - \zeta_{2}) = 2 \tilde{q}(x, \zeta_{1}) + 2\tilde{q}(x, \zeta_{2})\,. %\ x \in \Omega \setminus N\,,
\end{equation}
To this regard we construct recursively a basis of~$X$. Let us define $U_1:= \{v \in \mathbb{R}^n \setminus \RRR K \EEE  : \, q v \in \mathbb{R}^n \setminus \RRR K \EEE, \ q \in \mathbb{Q} \setminus \{0\} \}$. Then, it holds
\begin{equation}
\label{e:apsym19}
U_1 = \bigcap_{q \in \mathbb{Q} \setminus \{0\}} q \cdot (\mathbb{R}^n \setminus \RRR K \EEE) \,.
\end{equation}
Indeed, if $v \in U_1$, then for every $q \in \mathbb{Q} \setminus \{0\}$ we have $v/q \in \mathbb{R}^n \setminus \RRR K \EEE$, which implies that~$v$ belongs to the intersection on the right-hand side of \eqref{e:apsym19}. Conversely, if~$v$ belongs to the intersection in the right-hand side of \eqref{e:apsym19}, then for every $q \in \mathbb{Q} \setminus \{0\}$ there exists $w \in \mathbb{R}^n \setminus \RRR K \EEE$ such that $v = qw$. Hence, $v/q \in \mathbb{R}^n \setminus \RRR K \EEE$ for every $q \in \mathbb{Q}\setminus\{0\}$ and~$v \in U_1$. Since $\mathcal{L}^{n} (\RRR K \EEE) =0$,~\eqref{e:apsym19} yields ~$\mathcal{L}^{n} (\R^{n} \setminus U_1 )=0$ and we fix $v_1 \in U_1$. Let $j \leq n$ and suppose we have already defined $U_1, \dotsc, U_{j-1}$ and $v_1, \dotsc,v_{j-1}$. Then, we set 
\begin{align*}
U_j := \{v \in \mathbb{R}^n \setminus \RRR K \EEE : & \, q v + q_1v_1 + \dotsc +q_{j-1}v_{j-1} \in \mathbb{R}^n \setminus \RRR K \EEE\,, 
\\
&
\,  (q,q_1,\dotsc,q_{j-1}) \in \mathbb{Q}^{j} \text{ with } q \neq \{0\} \}\,.
\end{align*} 
We can write
\begin{equation}
\label{e:apsym20}
U_j = \bigcap_{ \substack{ (q,q_1,\dotsc,q_{j-1}) \in \mathbb{Q}^{j}, \\ q \neq 0}} q \cdot (\mathbb{R}^n \setminus \RRR K \EEE) -q_1v_1 -\dotsc - q_{j-1}v_{j-1}\,.
\end{equation}
Indeed, if $v \in U_j$, then given any $j$th-uplet $(q,q_1,\dotsc,q_{j-1}) \in \mathbb{Q}^j$ with $q \neq 0$ we have $v/q + q_1v_1/q \dotsc + q_{j-1}v_{j-1}/q \in \mathbb{R}^n \setminus \RRR K \EEE$. Thus,~$v$ belongs to the intersection on the right-hand side of~\eqref{e:apsym20}. Conversely, if~$v$ belongs to the intersection on the right-hand side of~\eqref{e:apsym20}, then for every $j$th-uplet $(q,q_1,\dotsc,q_{j-1}) \in \mathbb{Q}^j$ and $q \neq 0$ there exists $w \in \mathbb{R}^n \setminus \RRR K \EEE$ such that $v = w/q-q_1v_1/q -\dotsc - q_{j-1}v_{j-1}/q$. Hence, $qv +q_1v_1+ \dotsc+q_{j-1}v_{j-1} \in \mathbb{R}^n \setminus \RRR K \EEE$ and~$v \in U_j$. Since $\mathcal{L}^{n}(\RRR K \EEE) = 0$, we have $\mathcal{L}^{n} (\R^{n} \setminus U_{j}) = 0$ and we can find $v_j \in U_j$. 
%Furthermore, we may assume $v_{1}, \ldots, v_{n}$ to be \textcolor{blue}{orthogonal and with the same norm}.  \RRR Is it needed now? I can fix an orthogonal basis with same norm afterwards, I think. \EEE

Let us set $X := \{q_1v_1 + \dotsc + q_nv_n :\, (q_1,\dotsc,q_n) \in \mathbb{Q}^n  \}$. We check that $X\setminus \{0\} \subseteq \R^{n} \setminus \RRR K \EEE$. For every $n$-tuple $(q_1,\dotsc,q_n) \in \mathbb{Q}^n \setminus\{0\}$ let~$j\leq n$ be  the largest positive integer less than or equal to~$n$ for which $q_j \neq 0$. Then, we have that $q_1v_1 + \dotsc + q_jv_j \in \mathbb{R}^n \setminus \RRR K \EEE$ by definition of~$U_j$. Hence, $X\setminus \{0\} \subseteq \R^{n} \setminus \RRR K \EEE$ and $X$ is a vector space over~$\mathbb{Q}$ which is at most countable and dense in~$\R^{n}$.
%This means that if we define $X := \{q_1v_1 + \dotsc + q_nv_n \ | \ (q_1,\dotsc,q_n) \in \mathbb{Q}^n  \}$ then $X \subset \mathbb{R}^n \setminus N$ is a vector space over $\mathbb{Q}$ which is countable and dense in $\mathbb{R}^n$. 

Since $\zeta_{1},\zeta_{2} \in X$ implies $(\zeta_{1}, \zeta_{2}) \in S$ and~$X$ is at most countable, we deduce~\eqref{e:apsym16.99}. Let us denote by~$N_{0} \subseteq \Om$ such an exceptional set.
% the existence of a negligible set $N_0 \subset \Omega$ such that
%\begin{equation}
%    \label{e:apsym21}
%    x,z \in X \  \text{implies} \  \tilde{q}(x_0,x +z) +\tilde{q}(x_0,x -z)= 2\tilde{q}(x_0,x)+2\tilde{q}(x_0,z), \ x_0 \in \Omega \setminus  N_0.
%\end{equation}
Arguing as in the proof of \cite[Theorem 9.1]{dal}, for every $x \in \Omega \setminus N_0$ we deduce the existence of a symmetric $\mathbb{Q}$-bilinear form ${\rm B}_{x} \colon X \times X \to \mathbb{R}$ such that
\[
{\rm B}_{x}(\zeta,\zeta) = \tilde{q}(x,\zeta) \qquad \text{for every $\zeta \in X$.}
\]
This implies that for every $x \in \Omega \setminus N_0$ there exists $\RRR S(u)(x) \EEE \in \mathbb{M}^{n \times n}_{sym}$ such that 
\begin{equation}
\label{e:apsym22}
\RRR S(u)(x) \EEE \zeta \cdot \zeta = \tilde{q}(x,\zeta) \qquad \text{for every $\zeta \in X$.}
\end{equation}

In order to pass from \eqref{e:apsym22} to \eqref{e:apsym23} let us fix an arbitrary $\zeta \in \mathbb{R}^n$ and let us denote by $X_{\zeta}$ the vector space over $\mathbb{Q}$ generated by $X \cup \{ \zeta\}$. Notice that the set of $\zeta \in \mathbb{R}^n \setminus \RRR K \EEE$ for which $X_{\zeta} \setminus\{0\} \subseteq \mathbb{R}^n \setminus \RRR K \EEE$ has full measure. We choose $\zeta \in \mathbb{R}^n \setminus \RRR K \EEE$ such that $X_{\zeta} \setminus\{0\} \subseteq \mathbb{R}^n \setminus \RRR K \EEE$. Using the same argument above with~$X$ replaced by~$X_{\zeta}$, we deduce the existence of a negligible set $N'_0 \supset N_0$ such that \eqref{e:apsym16.99} holds true with $x \in \Omega \setminus N'_0$ and $\zeta_{1},\zeta_{2} \in X_{\zeta}$. Therefore, we find for every $x \in \Omega \setminus N'_0$ a matrix $\RRR S'(u)(x) \EEE \in \mathbb{M}^{n \times n}_{sym}$ such that~\eqref{e:apsym22} holds true for $\tilde{\zeta} \in X_{\zeta}$. In addition, being $X \subset X_{\zeta}$, for every $x \in \Omega \setminus N'_0$ it holds true $\RRR S(u)(x) \EEE \tilde{\zeta} \cdot \tilde\zeta =\RRR S' (u)(x) \EEE\tilde\zeta
 \cdot \tilde\zeta$ for every $\tilde{\zeta} \in X$. As~$X$ dense in~$\mathbb{R}^n$,~$\RRR S(u)(x)  =  S' (u)(x) \EEE$. 
 
 In order to concludes the proof it remains to prove the measurability of the map $\RRR S(u)(x) \EEE \colon \Omega \to \mathbb{M}^{n \times n}_{sym} $. For this purpose we fix an orthonormal basis $\{w_{1}, \ldots, w_{n}\}$ of~$\R^{n}$ such that $w _{i} \in \R^{n} \setminus \RRR K \EEE$ and $| w_{1} |= \ldots = | w_{n}| =: \alpha >0$. We notice that if we denote by~$\tilde{e}(u)(x)_i^j$ the $(i, j)$ entry of the symmetric matrix $\tilde{e}(u)(x)$ represented with respect to the orthonormal basis $\{w_1/\alpha,\dotsc, w_n/\alpha \}$, we have
\begin{equation}
\label{e:apsym26}
2\alpha^2\RRR S(u)(x)_i^j = S(u)(x) (w_i +w_j) \cdot (w_i +w_j) - S(u)(x) w_i  \cdot w_i -  S(u)(x)  w_j  \cdot w_j \,. \EEE
\end{equation}
By formula~\eqref{e:apsym15.99} we have that $x \mapsto \tilde{q}(x,\xi)$ is $\mathcal{L}^n$-measurable for every $\xi \in \mathbb{S}^{n-1}$. Equalities~\eqref{e:apsym23} and~\eqref{e:apsym26} imply that $x \mapsto \RRR S(u)(x) \EEE$ is a $\mathcal{L}^n$-measurable map with values in~$\mathbb{M}^{n \times n}_{sym}$. 
\end{proof}

\RRR
\begin{lemma}
\label{l:bah}
Let $F \in C^{\infty} (\R^{n} \times \R^{n}; \R^{n})$ satisfy~\eqref{hp:F}, let $\Om$ be an open subset of~$\R^{n}$, let $u \in GBD_{F}(\Omega)$ be Borel measurable, and let~$\{P_{\xi}\}_{\xi \in \mathbb{S}^{n-1}}$ be a family of curvilinear projections on~$\Om$ w.r.t.~$F$ parametrized by~$\varphi$. Then, the set
\begin{align}
\label{e:bah-A}
A &:= \big \{(x, \xi)\in  \Omega \times \mathbb{S}^{n-1} :  \ t^\xi_x \text{ is a Lebesgue point of }\hat{u}^\xi_{P_\xi(x)},
\\
&
\qquad\qquad\qquad \qquad \qquad \,\,\, \, \hat{u}^\xi_{P_\xi(x)} \text{ is  approximatively differentiable at } t^\xi_x \big\}\,, \nonumber
\end{align}
is Borel measurable. Moreover, there exist two Borel measurable maps~$v, \theta \colon \Om \times \mathbb{S}^{n-1} \to \R^{n}$ such that for every $\xi \in \mathbb{S}^{n-1}$ $\theta(\cdot, \xi)$ is Borel measurable and
\begin{align}
%\label{e:apsym1'}
%\aplim_{t \to t^\xi_x} \, |\hat{u}^\xi_{P_\xi(x)}(t) - v(x,\xi )|&=0 \qquad  \text{for $\mathcal{H}^{n-1}$-a.e. }\xi \in \mathbb{S}^{n-1},\\
    \label{e:apsym111}
   & \aplim_{t \to t^\xi_x} \bigg|\frac{\hat{u}^\xi_{P_\xi(x)}(t) -v ( x , \xi )}{t-t^\xi_x}- \theta(x,\xi ) \bigg|  =0 \qquad  \text{for a.e. } x \in \Omega\,,\\
    \label{e:apsym1.1.1}
   & \nabla \hat{u}^\xi_y (t) = (\theta (\cdot, \xi))^\xi_y (t)  \qquad \text{for $\mathcal{H}^{n-1}$-a.e. }y \in \xi^\bot, \text{ for a.e. } t \in \Omega_y^\xi\,.
\end{align} 
\end{lemma}

\begin{proof}
We set 
\begin{align*}
u^+(x,\xi )&:= \aplims_{s \to t^\xi_x} \, \hat{u}^\xi_{P_\xi(x)}(s)\,, \\
u^-(x,\xi )&:= \aplimi_{s \to t^\xi_x} \, \hat{u}^\xi_{P_\xi(x)}(s) \,,\\
\theta^+(x,\xi ) &:= \aplims_{s \to t^\xi_x} \, \frac{\hat{u}^\xi_{P_\xi(x)}(s)-u^+(x,\xi )}{s-t^\xi_x}\,, \\
\theta^-(x,\xi )&:= \aplimi_{s \to t^\xi_x} \, \frac{\hat{u}^\xi_{P_\xi(x)}(s)-u^-(x,\xi )}{s-t^\xi_x}\,.
\end{align*}
Arguing as in~\cite[Proposition~4.15]{AT_22-preprint} we can prove that $(x,\xi) \mapsto u^\pm(x,\xi )$ and $(x,\xi) \mapsto \theta^\pm(x,\xi  )$ are Borel measurable functions. Since 
\[
A = \{ ( x, \xi ) \in  \Om \times \mathbb{S}^{n-1}  :  \, u^+(x,\xi )=u^-(x,\xi ) \text{ and } \theta^+(x,\xi ) = \theta^-(x,\xi ) \}\,,
\]
we deduce that~$A$ is Borel. For $( x, \xi ) \in A$, let us denote $v(x,\xi ):=  u^+(x,\xi )=u^-(x,\xi )$ and $\theta (x,\xi ) : = \theta^+(x,\xi ) = \theta^-(x,\xi )$.

By definition of $GBD_{F}(\Om)$ we know that given $\xi \in \mathbb{S}^{n-1}$, for $\mathcal{H}^{n-1}$-a.e.~$y \in \xi^\bot$ it holds $\hat{u}^\xi_y \in BV_{\text{loc}}(\Omega_y^\xi)$. Thanks to a well known property of BV functions in one variable, this implies that for fixed $\xi \in \mathbb{S}^{n-1}$, for $\mathcal{H}^{n-1}$-a.e. $y \in \xi^\bot$, for a.e.~$t \in \Omega^{\xi}_{y}$ we have that~$t$ is a Lebesgue point and a point of approximate differentiability of $s \mapsto \hat{u}^\xi_y(s)$. In particular, this implies~\eqref{e:apsym111} and~\eqref{e:apsym1.1.1}.
%that for every $\xi \in \mathbb{S}^{n-1}$ we have
%\begin{equation}
%    \label{e:apsym111}
%    \aplim_{t \to t^\xi_x} \bigg|\frac{\hat{u}^\xi_{P_\xi(x)}(t) - v(x,\xi )}{t-t^\xi_x} - \theta (x,\xi )\bigg|=0 \qquad  \text{for a.e. } x \in \Omega\,.
%\end{equation}
\end{proof}
\EEE

We now prove Theorem~\ref{t:apsym}.

\begin{proof}[Proof of Theorem \ref{t:apsym}.]
In view of Remark \ref{r:nontrivial}, we can consider a \RRR cover \EEE of~$\Omega$ made of at most countably many open sets $U_i \subset \Omega$ and associated family of curvilinear projections $(P^i_{\xi})_{\xi \in \mathbb{S}^{n-1}}$ on $U_i$. In addition, thanks to Remark \ref{r:appsymi}, we can limit ourselves to prove that for every $i=1,2,\dotsc$ the first part of the theorem is satisfied on $U_i$. Without loss of generality we may thus ease the notation by assuming that $U_i=U=\Omega$ and $(P_{\xi})_{\xi \in \mathbb{S}^{n-1}}$ are curvilinear projections on $\Omega$. Since statements \eqref{e:apsym1.1} and \eqref{e:apsym1.4} do not depend on the representative in the Lebesgue class, we may as well assume~$u$ to be Borel and to coincide with its Lebesgue representative out of a Borel negligible set. Moreover, since the problem is local, we may assume without loss of generality that $u$ has compact support in~$\Om$. \RRR Let us define \EEE
\begin{align*}
A &:= \big \{(x, \xi)\in  \Omega \times \mathbb{S}^{n-1} :  \ t^\xi_x \text{ is a Lebesgue point of }\hat{u}^\xi_{P_\xi(x)},
\\
&
\qquad\qquad\qquad \qquad \qquad \,\,\, \, \hat{u}^\xi_{P_\xi(x)} \text{ is  approximatively differentiable at } t^\xi_x \big\}\,,
 \\
 A_{\xi} &:= \{ x \in U: \, ( x, \xi) \in A\} \qquad A_{x} := \{ \xi \in \mathbb{S}^{n-1}: \, ( x, \xi ) \in A\}\,.
\end{align*}
Then, by Lemma~\ref{l:bah} 
%and set 
%\begin{align*}
%u^+(x,\xi_{\varphi}(x))&:= \aplims_{s \to t^\xi_x} \, \hat{u}^\xi_{P_\xi(x)}(s)\,, \\
%u^-(x,\xi_{\varphi}(x))&:= \aplimi_{s \to t^\xi_x} \, \hat{u}^\xi_{P_\xi(x)}(s) \,,\\
%q^+(x,\xi_{\varphi}(x))&:= \aplims_{s \to t^\xi_x} \, \frac{\hat{u}^\xi_{P_\xi(x)}(s)-u^+(x,\xi_{\varphi}(x))}{s-t^\xi_x}\,, \\
%q^-(x,\xi_{\varphi}(x))&:= \aplimi_{s \to t^\xi_x} \, \frac{\hat{u}^\xi_{P_\xi(x)}(s)-u^-(x,\xi_{\varphi}(x))}{s-t^\xi_x}\,.
%\end{align*}
%Arguing as in~\cite[Proposition~4.15]{AT_22-preprint} we can prove that $(x,\xi) \mapsto u^\pm(x,\xi_\varphi(x))$ and $(x,\xi) \mapsto q^\pm(x,\xi_\varphi(x))$ are Borel measurable functions. Since 
%\[
%A = \{ ( x, \xi ) \in  \Om \times \mathbb{S}^{n-1}  :  \, u^+(x,\xi_{\varphi}(x))=u^-(x,\xi_{\varphi}(x)) \text{ and } q^+(x,\xi_{\varphi}(x))=q^-(x,\xi_{\varphi}(x))\}\,,
%\]
we deduce that~$A$ is Borel. \RRR Moreover, there exists $v, \theta \colon \Om \times \mathbb{S}^{n-1} \to \R^{n}$ Borel measurable such that~\eqref{e:apsym111} and~\eqref{e:apsym1.1.1} hold. In particular, from~\eqref{e:apsym1.1.1} and condition~\eqref{e:slice-2} in Definition~\ref{d:GBD} we deduce that
\begin{align}
\label{e;apsym34}
    \int_{\xi^\bot}\bigg( \int_{ \Om^\xi_y } | (  (\theta (\cdot, \xi))^\xi_y(t)  | \, \di t \bigg) \di \mathcal{H}^{n-1} (y) \leq \|\dot{\varphi}_\xi\|^2_{L^\infty}\text{Lip}(P_\xi)^{n-1}\, \lambda(\Omega) \, ,
\end{align}
which implies (after a change of variables) $\theta (\cdot, \xi )  \in L^1(\Omega)$.   \EEE

Since for every $\xi \in \mathbb{S}^{n-1}$ we have $\mathcal{L}^n(A_\xi)=\mathcal{L}^n(\Omega)$, applying Fubini's theorem to the Borel set $A \subseteq   \Omega \times \mathbb{S}^{n-1} $ we obtain
\[
\mathcal{H}^{n-1}(\mathbb{S}^{n-1})\cdot\mathcal{L}^n(\Omega) = \int_{\mathbb{S}^{n-1}} \mathcal{L}^n(A_\xi) \, \di \mathcal{H}^{n-1}(\xi)= \int_\Omega \mathcal{H}^{n-1}(A_x) \, \di x\,.
\]
This implies that for $\mathcal{L}^n$-a.e.~$x \in \Omega$ we have $\mathcal{H}^{n-1}(A_x)= \mathcal{H}^{n-1}(\mathbb{S}^{n-1})$. Hence, by definition of~$A$ we infer that there exists $N \subseteq \Omega$ with $\mathcal{L}^n (N) =0$ and such that for every $x \in \Omega \setminus N$ %\textcolor{blue}{($t_x^\xi$ should be a Lebesgue point of $t \mapsto \hat{u}^\xi_{P_\xi(x)}(t)$ to write what follows)}
\begin{align}
\label{e:apsym1'}
\aplim_{t \to t^\xi_x} \, |\hat{u}^\xi_{P_\xi(x)}(t) - \RRR v ( x , \xi ) \EEE | & = 0 \qquad  \text{for $\mathcal{H}^{n-1}$-a.e. }\xi \in \mathbb{S}^{n-1},\\
\label{e:apsym1}
    \aplim_{t \to t^\xi_x} \, \bigg|\frac{\hat{u}^\xi_{P_\xi(x)}(t) - \RRR v(x,\xi )  \EEE }{ t - t^\xi_x} - \RRR \theta ( x ,\xi ) \EEE \bigg|& = 0 \qquad \text{for $\mathcal{H}^{n-1}$-a.e. }\xi \in \mathbb{S}^{n-1}.
\end{align}
Up to consider a larger negligible set, still denoted by $N$, we may suppose that $u$ is approximately continuous at every $x \in \Omega \setminus N$. Thus, by denoting, with abuse of notation, by~$u(x)$ the approximate continuous representative of $u$ at $x\in \Omega \setminus N$,~\eqref{e:apsym1'} and~\eqref{e:apsym1} may be rewritten as
\begin{align}
\label{e:apsym01'}
\aplim_{t \to t^\xi_x} |\hat{u}^\xi_{P_\xi(x)}(t) -u(x)\cdot \xi_{\varphi}(x)|&=0 \ \  \text{for $\mathcal{H}^{n-1}$-a.e. }\xi \in \mathbb{S}^{n-1},\\
\label{e:apsym01}
    \aplim_{t \to t^\xi_x} \bigg|\frac{\hat{u}^\xi_{P_\xi(x)}(t) - u(x)\cdot \xi_{\varphi}(x)}{ t - t^\xi_x} - \RRR \theta (x,\xi ) \EEE \bigg|&=0 \ \ \text{for $\mathcal{H}^{n-1}$-a.e. }\xi \in \mathbb{S}^{n-1}.
\end{align}
Moreover, since~$P_{\xi}$ is \RRR a curvilinear projection and condition~\eqref{CP:item-2} of Definition~\ref{d:CP-maps} holds, \EEE for every $x \in \Omega\setminus N$, for every~$\xi \in \mathbb{S}^{n-1}$ the curve $t\mapsto \exp_{x}(t \xi_\varphi(x))$ (see Definition \ref{d:exp}) coincides for~$t$ small enough with the curve~$t \mapsto \varphi_{\xi} (P_{\xi} (x) + (t + t_{x}^{\xi}) \xi)$ (remember that $\xi_{\varphi}(x) = \dot{\varphi}_{\xi} (P_{\xi} (x) + t_{x}^{\xi} \xi)$). Thanks to property~(3) of Definition~\ref{d:CP} of curvilinear projections, the map $\xi \mapsto \xi_{\varphi}(x)/| \xi_{\varphi}(x)|$ is a diffeomorphism between $\mathbb{S}^{n-1}$ and itself. Therefore,~\eqref{e:apsym01'} and~\eqref{e:apsym01} can be reformulated for every $x \in \Omega \setminus N$ and for $\mathcal{H}^{n-1}$-a.e.~$\xi \in \mathbb{S}^{n-1}$ as
\begin{align}
\label{e:apsym2'}
&\aplim_{t \to 0} |u(\exp_{x}(t \xi_\varphi(x))) \cdot \dot{\exp}_{x}(t \xi_\varphi(x)) -u(x) \cdot \xi_{\varphi}(x) | = 0 \,,
\\
    \label{e:apsym2}
    &\aplim_{t \to 0} \bigg|\frac{u(\exp_{x}(t \xi_\varphi(x))) \cdot \dot{\exp}_{x}(t \xi_\varphi(x)) - u(x) \cdot \xi_{\varphi}(x)}{t} - \RRR \theta (x , \xi ) \EEE \bigg|=0 \,.
\end{align}

Now for every $( x, \xi ) \in \Omega \times \mathbb{S}^{n-1} $ for which~\eqref{e:apsym2'} and~\eqref{e:apsym2} hold, \RRR we set $q(x, \xi_\varphi(x)) := \theta(x, \xi)$, and we define $q(x, \xi_{\varphi}(x) ):=0$ otherwise. Then, we consider \EEE $q \colon \Omega \times \mathbb{R}^{n} \to \mathbb{R}$ the positively $2$-homogeneous extension of $q(x,\xi_\varphi(x))$ in the second variable. This means that, exploiting the fact that $\xi \mapsto \xi_\varphi(x)/|\xi_\varphi(x)|$ is a diffeomorphism of $\mathbb{S}^{n-1}$ and itself, we have
\begin{equation}
\label{e:apsym100}
q(x,\zeta) = \frac{|\zeta|^2}{|\xi_\varphi(x)|^2} \, q(x,\xi_\varphi(x)) \qquad \text{for $(x, \zeta) \in \Omega \times \R^{n}$,}
\end{equation}
whenever $\zeta$ satisfies $\xi_\varphi(x)/|\xi_\varphi(x)|=\zeta/|\zeta|$.

For $x \in \Omega\setminus N$, let  $\phi \colon {\rm B}_{r}(x) \setminus \{x\} \to \mathbb{S}^{n-1}$ be given by Proposition \ref{p:retr}, and let $\chi \colon {\rm B}_{r}(x) \setminus \{x\} \to \mathbb{R}^{n}$ and $d(\cdot,x) \colon {\rm B}_{r}(x)  \to \mathbb{R}$ be given by Definition \ref{d:chi1.1.1} (notice that, for simplicity of notation, we have dropped the index~$x$). We now show that 
\begin{equation}
\label{e:apsym3}
\aplim_{z \to x}\bigg|\frac{u(z)\cdot \chi(z) - u(x)\cdot \phi(z)}{\RRR {\rm d} \EEE (z,x)} -q\big(x,\phi(z)\big) \bigg|=0\,.
\end{equation}
By~\eqref{e:chi3}, there exists $c>0$ such that $\mathcal{H}^1(\{\phi^{-1}(\xi) \cap {\rm B}_{\rho}(x)\})\leq c\rho$ for every $\xi \in \mathbb{S}^{n-1}$ and for every $\rho \in (0,r)$. Thanks to~\eqref{e:retr2} we can make use of Coarea formula with map~$\phi$ to write for $\rho \in (0, r)$ 
\begin{align*}
    &\mint_{{\rm B}_\rho (x)}\bigg|\frac{u(z)\cdot \chi(z) - u(x)\cdot \phi(z)}{d(z, x)} -  q \big(x, \phi(z) \big) \bigg|\wedge 1 \, \di z \\ 
    &\leq c \int_{\mathbb{S}^{n-1}} \bigg(\mint_0^{c\rho}\bigg|\frac{ u (\exp_{x}(t \eta) ) \cdot \dot{\exp}_{x}(t \eta) - u(x)\cdot \eta}{t}- q(x, \eta) \bigg|\wedge 1 \, \frac{|\dot{\exp}_x(t\eta)|}{|\text{J}\phi(\exp_{x}(t \eta))|} \, \di t\bigg) \di \mathcal{H}^{n-1} (\eta) \\
    &\leq c' \int_{\mathbb{S}^{n-1}} \bigg(\mint_0^{c\rho}\bigg|\frac{ u (\exp_{x}(t \eta) ) \cdot \dot{\exp}_{x}(t \eta) - u(x)\cdot \eta}{t}- q(x, \eta) \bigg|\wedge 1 \,  |\exp_{x}(t \eta) - x|^{n-1 } \, \di t\bigg) \di \mathcal{H}^{n-1} (\eta)\\
    &\leq c'' \int_{\mathbb{S}^{n-1}} \bigg(\mint_0^{c\rho} \bigg|\frac{ u (\exp_{x}(t \eta) ) \cdot \dot{\exp}_{x}(t \eta) - u(x)\cdot \eta}{t}- q(x, \eta) \bigg|\wedge 1 \,  \RRR t^{n-1 } \EEE \, \di t\bigg) \di \mathcal{H}^{n-1}(\eta).
\end{align*}
By~\eqref{e:apsym2}--\eqref{e:apsym100} and dominated convergence we infer~\eqref{e:apsym3}. Setting 
\begin{align*}
q_{w}( x, \zeta)  := w \cdot F(x, \zeta) \qquad \text{for every $w, \zeta \in \R^{n}$,}
\end{align*}
by using the definition of exponential map, it can be directly shown that if we replace~$u$ with a constant map $w \in \R^{n}$ in~\eqref{e:apsym3} it holds
\begin{equation}
\label{e:apsym900}
\begin{split}
   \aplim_{z \to x} \bigg|\frac{w \cdot (\chi(z) - \phi(z))}{\text{d}(z,x)} - q_{w}(x,\phi(z)) \bigg|=0\,.
\end{split}
\end{equation}
%where we have set
%\RRR I think from now on you need the special form of~$F$, or that it is 2-homogeneous.\EEE
%\eqref{e:apsym3} is satisfied with $q(x_0,x)=  q_w(x_0,x)$ where 
%\begin{equation}
%   q_w(x_0,x)= w\cdot F(x_0,x). \end{equation}

We introduce
\begin{equation}
\label{e:apsymq}
\tilde{q}(x,\zeta ) :=  q ( x, \zeta ) \RRR - \EEE q_{u(x)}(x,\zeta) \qquad   x \in \Omega \setminus N\text{ and }   \zeta \in \mathbb{R}^n.
\end{equation}
We now show that for every $x \in \Omega \setminus N$ it holds
\begin{equation}
    \label{e:apsym15}
    \aplim_{z \to x} \frac{|(u(z) - u(x) ) \cdot (z - x) - \tilde{q}(x, z - x ) |}{| z - x |^2}=0\,.
\end{equation}
To simplify the notation we set $e(z):= (z - x)/|z - x|$. Making use of the maps~$\phi$ and~$\chi$ defined for $x \in \Omega \setminus N$, we first estimate by triangle inequality 
\begin{align}
\label{e:apsym1000}
   \bigg|\frac{ ( u (z) - u(x)) \cdot e(z)}{|z - x |} - \tilde{q}(x , e(z))\bigg|
   &\leq \bigg|\frac{( u ( z) - u (x) ) \cdot (e(z) - \chi(z))}{ | z - x |} \bigg| \\ \nonumber
   &+ \bigg|\frac{u(x) \cdot (\chi(z) - \phi(z))}{\text{d}(z,x)} -  q_{u(x)}(x,\phi(z)) \bigg|\\\nonumber
   &+ \bigg|\frac{u(z)\cdot \chi(z) - u(x) \cdot \phi(z)}{\text{d}(z,x)} -  q(x,\phi(z)) \bigg| \nonumber\\\nonumber
   &+ \bigg| \frac{u (x) \cdot (\phi(z) - \chi(z) )}{\text{d} ( z,x)} \bigg| \bigg| 1 - \frac{\text{d} (z,x)}{|z - x|} \bigg|\\\nonumber
   &+ \bigg|\frac{u(z)\cdot \chi(z) - u(x) \cdot \phi(z)}{\text{d} (z , x)} \bigg|\bigg| 1 - \frac{\text{d} (z, x)}{| z - x |}\bigg| \nonumber\\\nonumber
    &+ | q_{u(x)}(x,\phi(z)) - q_{u(x)} (x,e(z))|\\\nonumber
    &+ | q (x,\phi(z)) - q(x,e(z)) |\,.\nonumber
\end{align}
Now we examine the limit as $r \to 0^+$ of each term appearing in the right-hand side of~\eqref{e:apsym1000}. By~\eqref{e:chi2} and by the approximate continuity of~$u$ at~$x$ we have that
\begin{align}
\label{e:apsym8}
&\limsup_{r \searrow 0}  \mint_{ {\rm B}_r(x)}\frac{|( u (z) - u(x))\cdot (e (z)-\chi(z))|}{| z - x|}\wedge 1 \, \di z \\
&\leq  \limsup_{r \searrow  0} \bigg[\bigg(\sup_{z \in {\rm B}_{r}(x) \setminus \{x\}}\frac{| e (z)-\chi(z)|}{|z - x |}\bigg) \vee 1 \bigg]\mint_{{\rm B}_r(x)}| u (z) - u(x) |\wedge 1 \, \di z =0\,.\nonumber
\end{align}
%By~\eqref{e:apsym900} with $w=u(x)$ we have that
%
%\RRR 
%\begin{equation}
%\label{e:apsym9}
%\begin{split}
%   \aplim_{z \to x} \bigg|\frac{u(x) \cdot (\chi(z) - \phi(z))}{\text{d}(z,x)} - q_{u(x)}(x,\phi(z))\bigg|=0\,.
%\end{split}
%\end{equation}
%\EEE
% We have
%\begin{equation}
%\label{e:apsym10}
%\aplim_{x \to x_0}\bigg|\frac{u(x)\cdot \chi(x)-u(x_0) \cdot \phi(x)}{\text{d}(x,x_0)} -q(x_0,\phi(x))\bigg|=0,
%\end{equation}
%from \eqref{e:apsym3}. 

By definition of~$\RRR q_{u(x)} \EEE$ we have
\begin{align}
\label{e:apsym1001}
     &\limsup_{r \searrow 0} \mint_{{\rm B}_r(x)}| \RRR q_{u(x)} \EEE (x,\phi(z)) - \RRR q_{u(x)} \EEE (x,e(z)) |\wedge 1 \, \di z\\
     &\leq\limsup_{r \searrow 0} \mint_{{\rm B}_r(x)}|u(x)\cdot [F(x,\phi(z)) -F(x,e(z))]|\wedge 1 \, \di z \nonumber\,.
     \end{align}
     We show that the limsup on the right-hand side of inequality~\eqref{e:apsym1001} goes to zero. To this purpose we write
     \begin{align}
     & \limsup_{r \searrow 0} \mint_{{\rm B}_r(x)}|u(x)\cdot [F(x,\phi(z)) -F(x,e(z))]|\wedge 1 \, \di z \nonumber\\
    &\leq(|u(x)|\vee 1)\limsup_{r \searrow 0} \int_0^1\bigg( \int_{\partial {\rm B}_s(0)} |F(x,\phi(x + r\eta)) -F(x_0,e(x+r\eta))|\, \di \mathcal{H}^{n-1} (\eta) \bigg) \di s \nonumber\\
    &=(|u(x)|\vee 1)\limsup_{r \searrow 0} \int_0^1\bigg(s^{n-1} \int_{\mathbb{S}^{n-1}} |F(x ,\phi(x +  rs\xi)) -F(x , \xi)|\, \di \mathcal{H}^{n-1} (\xi) \bigg) \di s\,. \nonumber
\end{align}
Convergence \eqref{e:chi1} implies that 
\begin{equation}
\label{e:apsym1002}
\phi(x + rs\xi) \to \xi \text{ in $C^\infty(\mathbb{S}^{n-1};\mathbb{S}^{n-1})$ as $r \searrow 0$, uniformly in $s \in (0,1)$.}  
\end{equation}
As $F(x,\cdot) \colon \mathbb{S}^{n-1} \to \R^{n}$ is an $\mathcal{H}^{n-1}$-measurable function, for $\epsilon>0$ there exists $\psi \in C^0(\mathbb{S}^{n-1};\R^{n})$ with $\int_{\mathbb{S}^{n-1}} |F(x,\xi) -\psi(\xi)|\, \di \mathcal{H}^{n-1}(\xi) \leq \epsilon$. Denoting by $\phi^{-1}_{x,s,r}(\xi)$ the inverse of the map of $\xi \mapsto \phi(x + rs\xi)$, we can continue with
\begin{align}
\label{e:apsym1003}
     &\limsup_{r \searrow 0} \mint_{{\rm B}_r(x)}|u(x)\cdot [F(x,\phi(z)) -F(x,e(z))]|\wedge 1 \, \di z \\
     &\leq(|u(x)| \vee 1) \bigg(  \limsup_{r \searrow 0} \int_0^1\bigg(s^{n-1} \int_{\mathbb{S}^{n-1}} |F(x , \phi(x + rs\xi) ) - \psi( \phi (x + rs \xi) ) | \, \di \mathcal{H}^{n-1} (\xi) \bigg) \di s \nonumber\\
     &\qquad + \limsup_{r \searrow 0} \int_0^1 \bigg(s^{n-1} \int_{\mathbb{S}^{n-1}} |\psi( \phi (x + rs\xi) ) - \psi (\xi) | \, \di \mathcal{H}^{n-1} (\xi) \bigg) \di s \nonumber \\
     &\qquad +  \limsup_{r \searrow 0} \int_0^1\bigg(s^{n-1} \int_{\mathbb{S}^{n-1}} | F (x, \xi) - \psi(\xi) | \, \di \mathcal{H}^{n-1} (\xi) \bigg) \di s \bigg) \nonumber \\
     &  \leq( | u(x) | \vee 1) \bigg( \limsup_{r \searrow 0} \int_0^1\bigg(s^{n-1} \int_{\mathbb{S}^{n-1}} | F(x, \xi ) -\psi (\xi) | \, | J_\xi\phi^{-1}_{x,s,r}(\xi) | \, \di \mathcal{H}^{n-1} (\xi) \bigg) \di s \nonumber \\
     &\qquad + \limsup_{r \searrow 0} \int_0^1\bigg(s^{n-1} \int_{\mathbb{S}^{n-1}} | F(x,\xi) - \psi(\xi) |\, \di \mathcal{H}^{n-1} (\xi) \bigg) \di s \bigg) \nonumber \\
     &\leq \epsilon (|u(x)| \vee 1) \bigg(  \limsup_{r \searrow 0} \int_0^1 \|J_\xi\phi^{-1}_{x,s,r}\|_{L^\infty(\mathbb{S}^{n-1})} s^{n-1}\,\di s + \frac{1}{n} \bigg). \nonumber
     \end{align}
 By~\eqref{e:apsym1002} we have that $\|J_\xi \phi^{-1}_{x,s,r}\|_{L^\infty(\mathbb{S}^{n-1})} \to 1$ uniformly in $s \in (0,1)$ as $r \searrow 0$. This also means that for every but sufficiently small $r>0$, $\|J_\xi \phi^{-1}_{x,s,r}\|_{L^\infty(\mathbb{S}^{n-1})}$ is bounded uniformly with respect to $s \in (0,1)$. 
 %\RRR Is it? I would say it follows from~\eqref{e:apsym1002}. \EEE 
 Applying the dominated convergence theorem to~\eqref{e:apsym1003} we infer that
\[
\limsup_{r\searrow 0} \mint_{{\rm B}_r(x)}|u(x)\cdot [F(x,\phi(z)) -F(x,e(z))]|\wedge 1 \, \di z \leq  \frac{2\epsilon \, (|u(x)|\vee 1)}{n}\,. 
\]
Thanks to the arbitrariness of $\epsilon >0$ we conclude that
\begin{equation}
\label{e:apsym11}
\lim_{r \searrow 0} \mint_{{\rm B}_r(x)}|u(x)\cdot [F(x,\phi(z)) -F(x,e(z))]|\wedge 1 \, \di z  =0\,,
\end{equation}
which in turn implies that
\begin{equation}
\label{e:apsym2000}
\limsup_{r \searrow 0} \mint_{{\rm B}_r(x)}| q_{u(x)}(x,\phi(z)) - q_{u(x)}(x,e(z)) |\wedge 1 \, \di z = 0\,.
\end{equation}

From~\eqref{e:apsym2} it follows that $ q(x,\cdot) \colon \mathbb{S}^{n-1} \to \mathbb{R}$ is an $\mathcal{H}^{n-1}$-measurable function for every $x \in \Omega \setminus N$. Therefore, we may argue as in~\eqref{e:apsym1001}--\eqref{e:apsym11} to deduce that
\begin{equation}
\label{e:apsym12}
\lim_{r \searrow 0} \mint_{{\rm B}_r(x)}| q(x,\phi(z)) - {q}(x,e(z)) | \wedge 1\, \di z =0\,.
\end{equation}

By triangle inequality, by~\eqref{e:chi3}, by~\eqref{e:apsym900} with $w = u(x)$, and by dominated convergence we get that
\begin{align}
    \label{e:apsym13}
        &\limsup_{r \searrow 0}\mint_{{\rm B}_r(x)} \bigg( \bigg| \frac{u(x) \cdot (\phi(z) - \chi(z))}{\text{d}(z,x)} \bigg| \bigg| 1 - \frac{\text{d}(z,x) }{ | z - x |} \bigg| \bigg) \wedge 1 \, \di z \\
        &\leq \limsup_{r \searrow 0}\mint_{{\rm B}_r(x)} \bigg( \bigg| \frac{u(x) \cdot (\chi(z) - \phi(z))}{\text{d}(z,x)} - {q}_{u(x)}(x,\phi(z)) \bigg| \bigg|1 - \frac{\text{d}(z,x)}{| z - x |} \bigg| \bigg) \wedge 1 \, \di z \nonumber \\
        &\qquad + \limsup_{r \searrow 0}\mint_{{\rm B}_r(x)} \bigg(| {q}_{u(x)}(x,\phi(z)) | \bigg| 1 - \frac{\text{d}(z, x)}{| z - x |} \bigg|\bigg) \wedge 1 \, \di z \nonumber\\
        &=  \limsup_{r \searrow 0}\mint_{{\rm B}_1(0)} \bigg(| {q}_{u(x)}(x,\phi(x + rz)) | \bigg| 1 - \frac{\text{d}(x + rz,x)}{|rz|} \bigg|\bigg) \wedge 1 \, \di z =0\,. \nonumber
\end{align}
%where in the last inequality we used \eqref{e:apsym9} while in the last equality we used convergence \eqref{e:chi2} to deduce that $\text{d}(x_0+rx,x_0)/|rx| \to 1$ pointwise on ${\rm B}_1(0)$ as $r \to 0^+$ and Lebesgue's dominated convergence theorem. 
In the very same way we also deduce that
\begin{equation}
    \label{e:apsym14}
    \lim_{r \searrow 0}\mint_{{\rm B}_r(x)} \bigg(\bigg|\frac{u(z) \cdot \chi(z) - u(x)\cdot \chi(z)}{\text{d}(z,x)} \bigg| \bigg| 1 - \frac{\text{d}(z, x)}{ | z - x |} \bigg|\bigg) \wedge 1 \, \di z =0\,.
\end{equation}
Combining~\eqref{e:apsym3},~\eqref{e:apsym900}, and~\eqref{e:apsym1000}--\eqref{e:apsym14} we infer~\eqref{e:apsym15}  for every $x \in \Omega \setminus N$.

%\begin{equation}
%    \label{e:apsym15}
%    \aplim_{x \to x_0} \frac{|(u(x)-u(x_0)) \cdot (x-x_0) -\tilde{q}(x_0,x-x_0)|}{|x-x_0|^2}=0,
%\end{equation}
%where we have set
%\begin{equation}
%\label{e:apsymq}
%\tilde{q}(x_0,x):=q(x_0,x) \RRR - \EEE q_{u(x_0)}(x_0,x), \ \  \text{for } x_0 \in \Omega \setminus N,\text{ for }   x \in \mathbb{R}^n.
%\end{equation}

 We are in position to apply Lemma \ref{l:quadratic} and infer the existence of a  \RRR symmetric \EEE bi-linear form $\tilde{e}(u)(x) \colon \mathbb{R}^n \times \mathbb{R}^n \to \mathbb{R}$ such that \RRR
 \begin{equation}
 \label{e:maybe-needed}
 \tilde{e}(u)(x) \zeta \cdot \zeta = \tilde{q} (x, \zeta) \qquad \text{for every $\zeta \in \R^{n}$}
 \end{equation}
 holds true. \EEE This, together with~\eqref{e:apsym15}, implies~\eqref{e:apsym1.1}. 
 
 By Lemma \ref{l:quadratic} $x \mapsto \tilde{e}(u)(x)$ is an $\mathcal{L}^n$-measurable map with values in~$\mathbb{M}^{n \times n}_{sym}$. Thus, the map $(x, \zeta) \mapsto \tilde{e}(u)(x)\zeta \cdot \zeta$ is $(\mathcal{L}^n \times \mathcal{L}^n)$-measurable. Therefore, thanks to Fubini's theorem and the $2$-homogeneity of $\tilde{q}(x,\cdot)$, relation~\eqref{e:maybe-needed} can be turned into the following one: for $\mathcal{H}^{n-1}$-a.e. $\xi \in \mathbb{S}^{n-1}$
\begin{equation}
    \label{e:apsym30}
    \tilde{e}(u)(x)\xi\cdot \xi = \tilde{q}(x,\xi) \qquad  \text{for a.e.~$x \in \Omega$.}
\end{equation}
%that for a.e.~$x \in \Om$
%\begin{equation}
%    \label{e:apsym28}
%    \aplim_{z \to x} \frac{|( u (z) - u(x)) \cdot ( z - x ) - \tilde{e}(u) (x) ( z - x ) \cdot ( z - x )|}{| z - x |^2} = 0\,.
%\end{equation}
\RRR Setting \EEE for a.e.~$x \in \Omega$ 
\begin{equation}
    \label{e:apsym32}
    e(u)(x)\zeta \cdot \zeta := \tilde{e}(u)(x)\zeta \cdot \zeta \RRR + \EEE u(x) \cdot F(x, \zeta ) \qquad \text{for $\zeta \in \mathbb{R}^n$,}
\end{equation}
we infer from~\eqref{e:apsym1.1.1}, \RRR from~\eqref{e:apsymq}, and from the equality~$\theta(x, \xi) = q(x, \xi_{\varphi} (x))$ \EEE that for $\mathcal{H}^{n-1}$-a.e.~$\xi \in \mathbb{S}^{n-1}$
\begin{equation}
    \label{e:apsym31}
    \nabla \hat{u}^\xi_y (t) = (e(u) )^{\xi}_{y} \, \dot{\varphi}_{\xi} (y + t\xi) \cdot \dot{\varphi}_{\xi} (y + t\xi)\quad \text{for $\mathcal{H}^{n-1}$-a.e.~$y \in \xi^\bot$,   for a.e.~$t \in \Omega_y^\xi$.}
\end{equation}

Finally, by~\eqref{e;apsym34} we have that
\begin{align}
    \label{e:apsym35}
    \int_\Omega & |e(u)(x)| \, \di x \leq \int_{\mathbb{S}^{n-1}} \bigg( \frac{1}{c(\Omega)}\int_U |e(u)(x)\xi_{\varphi}(x) \cdot \xi_{\varphi}(x)| \, \di x  \bigg)\di \mathcal{H}^{n-1}(\xi) \\
    &\leq \int_{\mathbb{S}^{n-1}} \bigg( \frac{c'(\Omega)}{c(\Omega)}\int_{\xi^\bot}\bigg(\int_{U_y^\xi} |(e(u) )^{\xi}_{y} \, \dot{\varphi}_{\xi} (y + t\xi) \cdot \dot{\varphi}_{\xi} (y + t\xi) | \, \di t \bigg) \di \mathcal{H}^{n-1}(y)  \bigg) \di \mathcal{H}^{n-1}(\xi)\nonumber\\
    &\leq \frac{c'(\Omega)}{c(\Omega)} \sup_{\xi \in \mathbb{S}^{n-1}} \|\dot{\varphi}_\xi\|^2_{L^\infty (\Omega)} \text{Lip}(P_\xi) ^{n-1} \, \lambda(\Omega)\,, \nonumber
\end{align}
where 
\[
c(\Omega):= \inf_{x \in \Omega}\min_{\substack{\RRR A \in \mathbb{M}^{n \times n}_{sym} \\ |A|=1\EEE}} \int_{\mathbb{S}^{n-1}} |\RRR A \EEE \xi_\varphi(x) \cdot \xi_\varphi(x)|\, \di \mathcal{H}^{n-1}(\xi) \ \ \text{and} \ \ c'(\Omega):= \sup_{\substack{\xi \in \mathbb{S}^{n-1}\\ x \in \Omega}}|\text{J}P_\xi(x)|^{-1}. 
\]

It remains to show that $e(u) \in L^{1}(\Om; \mathbb{M}^{n}_{sym})$ and inequality \eqref{e:apsym1.3.99}. To this end, given $\epsilon >0$, for every $x \in \Om$ we fix $r(x) \in (0, +\infty)$ such that the family of maps~$(P_{\xi, x})_{\xi \in \mathbb{S}^{n-1}}$ given in~Definition~\ref{d:P-xi} is a family of curvilinear projections on~${\rm B}_{r(x)}(x)$ and for every $0<r \leq r(x)$
\begin{align}
 \sup_{\RRR \xi \in \mathbb{S}^{n-1} \EEE}\| \dot\varphi_{\xi, x}\|_{L^{\infty} ({\rm B}_{r}(x))} \leq 1+\epsilon\,, &\qquad  \sup_{\RRR \xi \in \mathbb{S}^{n-1} \EEE}\text{Lip}(P_{\xi, x}; {\rm B}_{r} (x) ) \leq 1+\epsilon\,, \label{e:phi-P1}\\
 \frac{1}{1+\epsilon} \leq c({\rm B}_{r}(x))\,,& \qquad c'({\rm B}_{r}(x)) \leq 1+\epsilon\,. \label{e:phi-P2}
\end{align}
Notice that such~$r(x)$ exists in view of Theorem~\ref{p:curvpro} \RRR and of Lemma~\ref{l:3.20-24}. \EEE Hence, in view of Remark \ref{r:appsymi} the same estimate~\eqref{e:apsym35}--\eqref{e:phi-P2} holds true if we replace $(P_\xi)_{\xi \in \mathbb{S}^{n-1}}$ and~$\Omega$ with $(P_{\xi,x})_{\xi \in \mathbb{S}^{n-1}}$ and $\mathrm{B}_{r}(x)$ ($0<r\leq r(x)$), respectively. This implies that for every $x \in \Om$ it holds true
\begin{equation}
\label{e:phi-P3}
\int_{{\rm B}_{r} (x)}  |e(u)(z)| \, \di z \leq (1+\epsilon)^{n+3} \, \lambda ( {\rm B}_{r} (x)) \qquad \text{for } 0 < r \leq r(x).
\end{equation} 
We can apply Vitali covering theorem \RRR (see, e.g.,~\cite[Theorem~2.19]{afp}), \EEE  to find sequences $r_i>0$ and~$x_{i} \in \Om$ such that the family~$\{{\rm B}_{r_i} (x_{i}): \, i \in \mathbb{N}\}$ is pairwise disjoint and $\mathcal{L}^n(\Omega \setminus \bigcup_i {\rm B}_{r_i}(x_i))=0$. Therefore, we infer from~\eqref{e:phi-P3} and the arbitrariness of $\epsilon>0$ that
\begin{equation}
\label{e:phi-99}
\int_{\Om} | e(u)|\, \di x \leq  \lambda(\Om) <+\infty\,.
\end{equation} 
In particular $e(u) \in L^{1}(\Om; \mathbb{M}^{n \times n}_{sym})$. To conclude we notice that the same argument yields \eqref{e:phi-99} with $\Omega$ replaced by any open subsets $U \subseteq \Omega$. Thus, relation~\eqref{e:apsym1.3.99} follows from the approximation property by means of open sets of Radon measure and the proof is concluded.\EEE
\end{proof}

\begin{remark}
We notice that, differently from the classical Euclidean case~\cite[Theorem~9.1]{dal}, equality~\eqref{e:apsym1.4} does not hold for every $\xi \in \mathbb{S}^{n-1}$. For this reason, we show in the next proposition that $\nabla \hat{u}^{\xi}_{y}$ can always be controlled in terms of~$ (e(u) )^{\xi}_{y} \, \dot{\varphi}_{\xi} (y + t\xi) \cdot \dot{\varphi}_{\xi} (y + t\xi)$.
%\RRR Probably \eqref{e:apsym1.4} is valid for every $\xi$ but we were not able to find a proof of this fact. \EEE
\end{remark}

\begin{proposition}
%Let $F \in C^{\infty}( \mathbb{R}^n \times \mathbb{R}^n ; \mathbb{R}^n)$  be a quadratic form in the second variable, 
Let $\Om$ be an open subset of~$\R^{n}$, $u \in GBD_{F}(\Omega)$, and let $e(u) \in L^{1}(\Om; \mathbb{M}^{n\times n}_{sym})$ be the map determined in Theorem~\ref{t:apsym}. Then, for every family $(P_\xi)_{\xi \in \mathbb{S}^{n-1}}$ of curvilinear projections on some open set~$U \subset \Omega$, and every $\xi \in \mathbb{S}^{n-1}$ we have 
\begin{align}
%\label{e:inapsym1.1}
%    &|(q_\xi)^\xi_y (t) | \leq |[ e(u)\xi \cdot \xi]^\xi_y (t) |\qquad  \text{for $\mathcal{H}^{n-1}$-a.e.~$y \in \xi^\bot$, for  a.e.~$t \in \Omega_y^\xi$,}\\
    \label{e:inapsym1.2}
    &|\nabla\hat{u}_y^\xi (t) | \leq |(e(u) )^{\xi}_{y} \, \dot{\varphi}_{\xi} (y + t\xi) \cdot \dot{\varphi}_{\xi} (y + t\xi) | \quad  \text{for $\mathcal{H}^{n-1}$-a.e.~$y \in \xi^\bot$,  for a.e.~$t \in U_y^\xi$.}
\end{align}
\end{proposition}

%\RRR S: I think we have to reformulate in terms of the local projections~$P_{\xi, x_{0}}$ induced by~$F$. \EEE

\begin{proof}
In order to simplify the notation we assume $U=\Omega$. We denote by \RRR $v, \theta \colon \Om \times \mathbb{S}^{n-1} \to \R^{n}$ \EEE the maps constructed in Lemma~\ref{l:bah} by means of the family $(P_\xi)_{\xi \in \mathbb{S}^{n-1}}$. \RRR In particular, we have that for every $\xi \in \mathbb{S}^{n-1}$
\begin{equation}
    \label{e:apsym1.1.1.1}
    \nabla \hat{u}^\xi_y (t) = (\theta(\cdot, \xi))^\xi_y (t)  \qquad \text{for $\mathcal{H}^{n-1}$-a.e. }y \in \xi^\bot, \text{ for a.e. } t \in \Omega_y^\xi\,.
\end{equation}\EEE

We claim that for every $\xi \in \mathbb{S}^{n-1}$ and for every $B \in \mathcal{B}(\Omega)$ Borel we have
\begin{align}
    \label{e:inapsym1}
    \int_{\xi^\bot} & \bigg(\int_{B^\xi_y} |\RRR \theta(\cdot, \xi))^\xi_y (t) \EEE | \, \di t\bigg) \di \mathcal{H}^{n-1}(y)
    \\
    &
     \leq \int_{\xi^\bot} \bigg(\int_{B^\xi_y} |(e(u) )^{\xi}_{y} \, \dot{\varphi}_{\xi} (y + t\xi) \cdot \dot{\varphi}_{\xi} (y + t\xi) | \, \di t\bigg) \di \mathcal{H}^{n-1}(y). \nonumber
\end{align}

Let us set~$N:= \{\xi \in \mathbb{S}^{n-1}: \, \text{\eqref{e:apsym1.4} is satisfied in~$\xi$}\}$ and let us fix~$\xi \in \mathbb{S}^{n-1}$. Since~$\mathcal{H}^{n-1}(\mathbb{S}^{n-1} \setminus N) = 0$, there exists a sequence $\xi_j \in N$ such that $\xi_j \to \xi$ as $j \to \infty$. We define the measure $\tilde{\mu}_{u,k}$ as in \eqref{e:mi-u-def} with~$\xi$ is restricted to the family~$(\xi_j)_{j \geq k}$. By construction it holds that $\mu^{\xi_j}_{u} \leq \tilde{\mu}_{u,k}$  for every $j\geq k$. Therefore, Proposition~\ref{p:xi-lsc} and Theorem~\ref{t:apsym} imply that for every open set $V \subseteq \Omega$ 
\begin{equation}
\label{e:100}
     \int_{\xi^\bot} \bigg(\int_{V^\xi_y} |\RRR \theta(\cdot, \xi))^\xi_y (t) \EEE| \, \di t\bigg) \di \mathcal{H}^{n-1}(y) \leq \mu^\xi_{u} (V) \leq \tilde{\mu}_{u,k}(V) \qquad \text{for every }k.
\end{equation}
The measures appearing in inequality~\eqref{e:100} are Radon. Hence, we deduce that for every $B \in \mathcal{B}(\Omega)$ we have 
\begin{equation}
\label{e:inapsym2}
   \int_{\xi^\bot} \bigg(\int_{B^\xi_y} | \RRR \theta(\cdot, \xi))^\xi_y (t) \EEE | \, \di t\bigg) \di \mathcal{H}^{n-1}(y) \leq \mu^\xi_{u} (B) \leq \tilde{\mu}_{u,k}(B) \qquad  \text{for every }k.
\end{equation}
 We infer from~\eqref{e:apsym1.4} that for every~$j$ the absolutely continuous part~$(\mu_{u}^{\xi_j})^{a}$ of~$\mu_{u}^{\xi_j}$ w.r.t.~the Lebesgue measure is given by 
\begin{displaymath}
(\mu_{u}^{\xi_j})^{a} (B)= \int_{\xi_j^\bot} \bigg( \int_{B^{\xi}_{y}} (e(u))^{\xi_{j}}_{y} \, \dot{\varphi}_{\xi_{j}} (y + t\xi_{j}) \cdot \dot{\varphi}_{\xi_{j}} (y + t\xi_{j}) \, \di t\bigg)\di \mathcal{H}^{n-1}(y) \qquad \text{for $B \in \mathcal{B}(\Omega)$.}
\end{displaymath}
Therefore, for every~$k$ the absolutely continuous part~$\tilde{\mu}^a_{u,k}$ of~$\tilde{\mu}_{u,k}$ w.r.t.~the Lebesgue measure is given by
\begin{align}
\label{e:inapsym3}
\tilde{\mu}^a_{u,k}(B)& = \sup \sum_{j} (\mu_{u}^{\xi_j})^{a} (B_j) 
\\
& =\sup \sum_{j} \int_{\xi_j^\bot} \bigg( \int_{(B_j)^{\xi_j}_y} (e(u))^{\xi_{j}}_{y} \, \dot{\varphi}_{\xi_{j}} (y + t\xi_{j}) \cdot \dot{\varphi}_{\xi_{j}} (y + t\xi_{j})  \, \di t\bigg) \di \mathcal{H}^{n-1}(y) \nonumber\\
&= \sup \sum_{j} \int_{B_j} |e(u)(x) \xi_{j, \varphi} (x) \cdot \xi_{j, \varphi} (x)| |JP_{\xi_j}(x)|\, \di x \nonumber
\end{align}
%\RRR S: In the second line, $\xi_{j}$ means $\xi_{j}(x)$? \EEE
for every $B \subset \Omega$ Borel, where the supremum is taken among all subsets of indices in $\{j: j\geq k \}$ and among all finite families of pairwise disjoint Borel sets $B_j$ contained in $B$. Notice that \RRR in view of condition~\ref{d:CP-4} in Definition~\ref{d:CP} \EEE we can estimate
\[
\begin{split}
    &\int_{\Omega} (|e(u) (x) \xi_{j, \varphi} (x) \cdot \xi_{j, \varphi} (x)| |JP_{\xi_j}(x)| - | e(u)(x) \xi_{\varphi}(x) \cdot \xi_{\varphi}(x)| |JP_{\xi}(x)|)\, \di x \\
    &\leq \RRR C \EEE \int_{\Omega} |e(u)(x)| |\xi_{j, \varphi} (x) - \xi_{\varphi}(x)| |JP_{\xi_j}(x)| \, \di x + \int_{\Omega} |e(u)(x)| ( | JP_{\xi_j}(x)|  - | JP_{\xi}(x)|) \, \di x\,,
\end{split}
\]
\RRR for a positive constant~$C$ independent of~$j$. \EEE Therefore, exploiting the fact that $\xi_{j, \varphi} (x) \to \xi_{\varphi}(x)$ and $|JP_{\xi_j}(x)| \to |JP_{\xi}(x)|$ uniformly in~$\Omega$ as $j \to \infty$ and that $e(u) \in L^1(\Omega;\mathbb{M}^{n \times n}_{sym})$ we get from~\eqref{e:inapsym3} that
\begin{equation}
\label{e:inapsym4}
\begin{split}
\tilde{\mu}^a_{u,k}(B) = \int_{\xi^\bot} \bigg( \int_{B^\xi_y} (e(u))^{\xi}_{y} \, \dot{\varphi}_{\xi} (y + t\xi) \cdot \dot{\varphi}_{\xi} (y + t\xi) \, \di t\bigg) \di \mathcal{H}^{n-1}(y) + O(k^{-1}).
\end{split}
\end{equation}
Combining \eqref{e:inapsym2} and \eqref{e:inapsym4} we conclude~\eqref{e:inapsym1}.

Inequality~\eqref{e:inapsym1} can be extended to all measurable sets, and in particular holds for $B = \{x \in \Omega : \, \RRR \theta ( x ,\xi) \EEE > e(u) (x)\xi_\varphi(x)\cdot \xi_\varphi(x) \}$. This implies that for every $\xi \in \mathbb{S}^{n-1}$
\begin{align}
\label{e:inapsym1.1}
    &| (\theta(\cdot, \xi))^{\xi}_{y} (t)  | \leq  | (e(u))^{\xi}_{y} \, \dot{\varphi}_{\xi} (y + t\xi) \cdot \dot{\varphi}_{\xi} (y + t\xi)|
\end{align}
for $\mathcal{H}^{n-1}$-a.e.~$y \in \xi^\bot$, for  a.e.~$t \in \Omega_y^\xi$. Finally,~\eqref{e:inapsym1.2} follows from~\eqref{e:apsym1.1.1.1} and~\eqref{e:inapsym1.1}.
\end{proof}

\subsection{Slicing the jump set in $GBD(\rm M)$}
\label{s:jump-M}

 Let $({\rm M}, g)$ be a Riemannian manifold of dimension~$n$. In this subsection we recover the slicing properties of the jump set of $\omega \in GBD({\rm M})$.  In order to state the result, we need the notion of family of curvilinear projections on the manifold~$\rm M$, which follows from Definition~\ref{d:CP} of family of curvilinear projections on~$\Omega\subseteq \R^{n}$.

\begin{definition}[Family of curvilinear projections on~$\rm M$]
\label{d:CP_M111}
Let $V \subseteq \rm M$ open. We say that a family of maps $P_\xi \colon V \to \xi^{\bot}$ for $\xi \in \mathbb{S}^{n-1}$ is a family of \emph{curvilinear projections} on~$V$ if for every chart~$(U, \psi)$ we have that the $\{\overline{P}_{\xi} = P_{\xi}\circ \psi^{-1} \}_{\xi \in\mathbb{S}^{n-1}}$ is a family of curvilinear projections on~$\psi(U\cap V)$.
\end{definition}

Given $\omega \in \mathcal{D}^{1}(\rm M)$, $V \subseteq {\rm M}$ open, $\{P_{\xi}\}_{\xi \in \mathbb{S}^{n-1}}$ a family of curvilinear projections on~$V$, for every chart $(U, \psi)$ and every $\xi \in \mathbb{S}^{n-1}$ we set
\begin{displaymath}
\omega_{\xi} (p) : = u(\psi(p)) \cdot \xi_{\overline{\varphi}} (\psi(p))\,,
\end{displaymath}
where $\{\overline{P}_{\xi} = P_{\xi} \circ \RRR \psi^{-1} \EEE \}_{\xi \in \mathbb{S}^{n-1}}$ is a family of curvilinear projections on~$\psi(U\cap V)$, $\overline{\varphi}$ is a parametrization of the family $\{\overline{P}_{\xi} \}_{\xi \in \mathbb{S}^{n-1}}$ according to Definition~\ref{d:param}, $\xi_{\overline{\varphi}}$ is the velocity field defined in Definition~\ref{d:velocity-field}, and~$u$ is as in~\eqref{e:vpsi}.

\begin{theorem}
    \label{t:jumpslicingform}
    Let $({\rm M}, g)$ be an $n$-dimensional Riemannian manifold and let $\omega \in GBD(\rm M)$. Then $J_\omega$ is countably $(n-1)$-rectifiable. Moreover, if $(U,\psi)$ is a chart and $(P_\xi)_{\xi \in \mathbb{S}^{n-1}}$ is a family of curvilinear projections on $U$, then it holds true
 \begin{align}
 \label{e:neuju1}
 J_{\hat \omega^\xi_y} &= (J_{\omega_\xi})^{\xi}_{y} \qquad  \text{for every $\xi \in \mathbb{S}^{n-1}$, for $\mathcal{H}^{n-1}$-a.e. $y \in \xi^\bot$},\\
 \label{e:neuju2}
 J_{\hat \omega^\xi_y} &= (J_{\omega})^{\xi}_{y} \qquad  \text{for $\mathcal{H}^{n-1}$-a.e. $\xi \in \mathbb{S}^{n-1}$, for $\mathcal{H}^{n-1}$-a.e. $y \in \xi^\bot$},\\
 \label{e:neuju3}
 J_{\hat \omega^\xi_y} &\subseteq (J_{\omega})^{\xi}_{y} \qquad  \text{for every $\xi \in \mathbb{S}^{n-1}$, for $\mathcal{H}^{n-1}$-a.e. $y \in \xi^\bot$}.
 \end{align}
In addition, the following relation between traces holds true for every $\xi \in \mathbb{S}^{n-1}$, for $\mathcal{H}^{n-1}$-a.e. $y \in \xi^\bot$, and for every $t \in (J_{\omega_\xi})^\xi_y$
 \begin{equation}
 \label{e:euju5}
 \aplim_{\substack{q \to p \\ q \in H^{\pm} (p)}} \omega_{\xi}(q)  = \aplim_{s \to t^{\pm\sigma(p)}}\hat{\omega}^{\xi}_{y}(s)\,, 
\end{equation}
whenever $p=\varphi_\xi(y+t\xi) \in U $ and $\nu_{\omega_\xi} \in \Gamma( U )$ is a Borel measurable orientation of $J_{\omega_{\xi}}$.
    \end{theorem}
    \begin{proof}
    Since rectifiability is a local property we reduce ourselves to work on a chart $(U,\psi)$. Notice that
    \[
    \langle \omega(q), d \, \exp_p[q](v) \rangle_{q}  = u(\psi(q)) \cdot  \overline{v}(\psi(q)), 
    \]
    where $u \colon \psi(U) \to \mathbb{R}^n$ is defined as in~\eqref{e:vpsi} and $\overline{v} \colon \psi(U) \to \mathbb{R}^n$ are the components of $d  \exp_p[q](v)$ w.r.t.~the basis $(g_i(q))_i$. From the continuity of $g_i(\cdot)$ and $g^i(\cdot)$ together with the facts that $\lim_{q \to p} \overline{v}(\psi(q)) = \overline{v}(\psi(p))$ and the map $v \mapsto \overline{v}(\psi(p))$ is an isomorphism between the vector spaces ${\rm T}_p{\rm M}$ and $\mathbb{R}^n$ (for every $p$), a standard geometric argument leads to 
    \[
    \psi(J_\omega \cap U) = J_u \cap \psi(U).
    \]
    The countably $(n-1)$-rectifiability of $J_\omega$ follows thus from Theorem~\ref{t:delnin}. 

    Now we consider the family of curvilinear projections $(\overline{P}_\xi)_{\xi \in \mathbb{S}^{n-1}}$ on $\psi(U)$ defined as $\overline{P}_\xi := P_\xi \circ \psi^{-1}$. Since by Proposition \ref{p:equivalent} we have $u \in GBD_F(\psi(U))$, the remaining part of theorem follows by a direct application of Theorem \ref{p:euju}.
\end{proof}

\subsection{Approximate symmetric gradient in $GBD(\rm M)$}
\label{s:appr-sym-M}

In this subsection we show that $\omega \in GBD(\rm M)$ admits an approximate symmetric gradient.

\begin{theorem}[Existence of the approximate symmetric gradient]
\label{t:apsymmanifold}
Let $({\rm M}, g)$ be an $n$-dimensional Riemannian manifold and let $\omega \in GBD(\rm M)$. Then for $\mathcal{H}^n$-a.e. $p \in \rm M$ there exists the approximate symmetric gradient $e(\omega)(p)$ and moreover 
\begin{equation}
\label{e:l1appsym}
\int_M \|e(\omega)(p)\|_{{\rm T}_p{\rm M} \otimes {\rm T}_p{\rm M}} \, \di \mathcal{H}^{n}(p) \leq \lambda(\rm M)\, .
\end{equation}

In addition, if $(P_{\xi})_{\xi \in \mathbb{S}^{n-1}}$ is a family of curvilinear projections on an open set~$U \subseteq \rm M$, then for $\mathcal{H}^{n-1}$-a.e.~$\xi \in \mathbb{S}^{n-1}$ it holds true
\begin{equation}
    \label{e:apsym1.4manifold}
    \nabla \hat{\omega}_y^\xi (t) = (e(\omega))^{\xi}_{y}(\dot{\varphi}_{\xi} (y + t\xi)) \quad \text{for $\mathcal{H}^{n-1}$-a.e. }y \in \xi^\bot, \text{ for a.e. } t \in U_y^\xi\,.
\end{equation}
% There exists $q \colon \Omega \times \mathbb{R}^n \to \mathbb{R}$ with $q(x,tz)=t^2q(x,z)$ for every $(x,z,t) \in \Omega \times \mathbb{R}^n \times \mathbb{R}$ such that if we define $q_\xi(x):=q(x,\xi(x))$ then for every $\xi \in \mathbb{S}^{n-1}$ we have $q_\xi \in L^1(\Omega)$ and
%\begin{equation}
%\label{e:apsym1.2}
%    \nabla \hat{u}_y^\xi (t) = (q_\xi)_y^\xi (t) \qquad \text{for $\mathcal{H}^{n-1}$-a.e. }y \in \xi^\bot, \text{ for a.e. } t \in \Omega_y^\xi.
%\end{equation}
\end{theorem}

\begin{proof}
    Let $(U,\psi)$ be a chart of $\rm M$. By Proposition~\ref{p:equivalent}, the function $u\colon \psi(U) \to \R^{n}$ defined in~\eqref{e:M3} belongs to~$GBD_{F}(\psi(U))$ with~$F$ given by~\eqref{e:chris}. Moreover, if $\lambda \in \mathcal{M}_{b} ^{+} (\rm M)$ is the measure appearing in Definition~\ref{d:GBD_M} of $GBD(\rm M)$, defining~$ \overline{\lambda}$ as in~\eqref{e:lambda-bar} we have that for every $V \subseteq \psi(U)$ open, every $\xi^{\bot}\in \mathbb{S}^{n-1}$, and every curvilinear projection~$P \colon V \to \xi^{\bot}$
    \begin{align*}
    \int_{\xi^{\bot}}  \big(\big| | {\rm D} \hat{u}^{\xi}_{y} | ( B^{\xi}_{y} \setminus J^{1}_{\hat{u}^{\xi}_{y}}) + \mathcal{H}^{0} (  B^{\xi}_{y} \cap J^{1}_{\hat{u}^{\xi}_{y}} ) \big) \, \di \mathcal{H}^{n-1} (y) \leq \| \dot{\varphi} \|^{2}_{L^{\infty}} {\rm Lip}(P ; V)^{n-1} \overline{\lambda}(B)\,,
    \end{align*}
    for every $B \in \mathcal{B} (V)$. By Remark~\ref{r:lambda-bar}, for $\epsilon>0$ fixed we may further assume that 
    \begin{align}
    \label{e:eps}
    \overline{\lambda} \leq (1 + \epsilon) \psi_{\sharp}\lambda\,, \qquad {\rm Lip} (\psi; U ) < 1 + \epsilon \,, \qquad  {\rm Lip} (\psi^{-1}; \psi(U)) < 1 + \epsilon\,.
    \end{align}
    
     For $x_0 \in \psi(U)$ consider $\phi_{x_0}(\cdot)$ and $\chi_{x_0}(\cdot)$ the vector fields defined in \eqref{e:retr12} and \eqref{e:chi1.1.1}, respectively, and let $\text{d}(\cdot,x_0)$ be the function defined in \eqref{e:chi1.2}. Notice that, because of the identity $\text{d}_{\rm M}(q,p)= |v_{q}|_{p}$ we can rewrite \eqref{e:apsymmanifold1} as
    \begin{equation}
    \label{e:apsymmanifold97}
    \aplim_{q \to p} \bigg|\frac{\langle\omega(q), d\,\text{exp}_p[q](\overline{v}_q)\rangle_{q} - \langle \omega(p), \overline{v}_q\rangle_{p}}{ |v_{q}|_{p} } - e(\omega)(p)( \overline{v}_q)\bigg|=0,
    \end{equation}
    where $\overline{v}_q := v_q/|v_{q}|_{p}$.
 In addition we have
 \[
 \begin{split}
 \overline{v}_q &= \frac{1}{c_p(q)} \sum_{i=1}^n \phi_{\psi(p)}(\psi(q))_i \, g_i(p) \\
d\, \text{exp}_p[q](\overline{v_q}) &= \frac{1}{c_p(q)} \sum_{i=1}^n \chi_{\psi(p)}(\psi(q))_i \, g_i(q),
 \end{split}
 \]
 where we used that the renormalization constant $c_p(q)$ is the same because of the fact that the Riemannian norm of the velocity field of geodesics are constant in time. We therefore infer that
    \[
    \begin{split}
    \langle \omega(p), \overline{v}_q \rangle_{p} &= \frac{1}{c_p(q)} u(\psi(p)) \cdot \phi_{\psi(p)}(\psi(q)) \\
    \langle \omega(q), d\, \text{exp}_p[q](\overline{v_q}) \rangle_{q} &= \frac{1}{c_p(q)} u(\psi(q)) \cdot \chi_{\psi(p)}(\psi(q)).
    \end{split}
    \]
In addition the definition of $\text{d}(x_0,\cdot)$ gives $\text{d}(\psi(p),\psi(q))= |w_q|$ where $w_q \in \mathbb{R}^n$ is such that $\text{exp}_{\psi(p)}(w_q)=\psi(q)$. Therefore, the geodesic which starts at $p$ with initial velocity $\sum_{i=1}^n (w_q)_i \, g_i(p) \in {\rm T}_p{\rm M}$ reaches at time $t=1$ the point $q$. This means that $v_q= \sum_{i=1}^n (w_q)_i \, g_i(p)$. We also have $\phi_{\psi(p)}(\psi(q))= w_q/|w_q|$. Therefore
\[
\begin{split}
 \overline{v}_q &= \frac{1}{c_p(q)} \sum_{i=1}^n \phi_{\psi(p)}(\psi(q))_i \, g_i(p) =\frac{1}{c_p(q)|w_q|} \sum_{i=1}^n (w_q)_i \, g_i(p) = \frac{v_q}{c_p(q)|w_q|}\,.
 \end{split}
\]
Hence, we get that
\[
\frac{\text{d}_{\rm {M}}(p,q)}{\text{d}(\psi(p),\psi(q))}=\frac{ |v_q|_{p}}{|w_q|} = c_p(q)\,.
\]
In particular,
\[
\frac{\langle\omega(q), d\,\text{exp}_p[q](\overline{v}_q)\rangle_{q} - \langle \omega(p), \overline{v}_q\rangle_{p}}{ |v_q|_{p} } = \frac{u(\psi(q))\cdot \chi_{\psi(p)}(\psi(q)) - u(\psi(p))\cdot \phi_{\psi(p)}(\psi(q))}{c_p(q)^2\text{d}(\psi(p),\psi(q))}
\]
We already know from \eqref{e:apsym3},\eqref{e:apsymq}, and \eqref{e:apsym32} that the following holds true
\begin{align}
\label{e:apsymmanifold98}
\aplim_{q \to p} \bigg|&\frac{u(\psi(q))\cdot \chi_{\psi(p)}(\psi(q)) - u(\psi(p))\cdot \phi_{\psi(p)}(\psi(q))}{c_p(q)^2\text{d}(\psi(p),\psi(q))}\\
\nonumber
& \ \ \ \ \ \ \ \ \ \ \ \ \ \ \ \ \ \  - \frac{e(u)(\psi(p))\phi_{\psi(p)}(\psi(q)) \cdot \phi_{\psi(p)}(\psi(q))}{c_p(q)^2}\bigg|=0,
\end{align}
as soon as we provide a strictly positive lower bound for the function $c_p(\cdot)$ in a neighborhood of $p$. But this follows from the fact that 
\[
c_p(q)=\frac{ |v_q|_{p}}{|w_q|} \geq  \inf_{ \xi \in \mathbb{S}^{n-1}} \sum_{i,j=1}^n \xi_i\xi_j (g_i(p) \cdot g_j(p)) > 0,
\]
since $(g_i \cdot g_j)_{ij}$ is a positive definite matrix.

Defining $e(\omega)(p) \in {\rm T}_p{\rm M} \otimes {\rm T}_p{\rm M}$ as
\begin{equation}
\label{e:apsymmanifold99}
e(\omega)(p)(v) := e(u)(\psi(p))L_p(v)\cdot L_p(v) \qquad \text{for every } v \in {\rm T}_p{\rm M}\,, 
\end{equation}
where $L_p \colon {\rm T}_p{\rm M} \to \mathbb{R}^n$ is the linear map defined as $L_p(v):= \sum_{i=1}^n \langle g^i(p),v \rangle e_i$, we verify that
\[
\begin{split}
e(\omega)(p)(v_q) &:= \RRR |v_q|_{p}^{2} \EEE\frac{e(u)(\psi(p))\phi_{\psi(p)}(\psi(q)) \cdot \phi_{\psi(p)}(\psi(q))}{c_p(q)^2}\\
&= |w_q|^2 e(u)(\psi(p))\phi_{\psi(p)}(\psi(q)) \cdot \phi_{\psi(p)}(\psi(q))) \\
&= e(u)(\psi(p))w_q \cdot w_q = e(u)(\psi(p))L_p(v_q)\cdot L_p(v_q).
\end{split}
\]
By combining the above equalities with \eqref{e:apsymmanifold98} we finally obtain the validity of \eqref{e:apsymmanifold97}.

In order to prove that $\|e(\omega)\|_{{\rm T}_p{\rm M} \otimes {\rm T}_p{\rm M}} \in L^1({\rm M})$ we infer from \eqref{e:apsymmanifold99}
\begin{align}
\label{e:Lp}
\|e(\omega)(p)\|_{{\rm T}_p{\rm M} \otimes {\rm T}_p{\rm M}} = \sup_{v \in {\rm T}_p{\rm M}} \frac{|e(\omega)(p)(v)|}{|v|_{p}^{2} } &= \sup_{v \in {\rm T}_p{\rm M}} \frac{|e(u)(\psi(p))L_p(v)\cdot L_p(v)|}{|v|_{p}^{2}  } \\
&\leq \|L_p\|_{\mathcal{L}({\rm T}_p{\rm M} ;\mathbb{R}^n)}^2 |e(u)(\psi(p))| \nonumber
\end{align}
In view of~\eqref{e:eps}, we further have that for $p \in U$ we have that 
\begin{align}
\label{e:Lp-2}
\|L_p\|_{\mathcal{L}(T_pM;\mathbb{R}^n)} & \leq \sup_{\substack{v \in {\rm T}_{p} {\rm M}\\ |v|_{p} \leq 1}} \, | L_{p} v | \leq ( 1 + \epsilon) \sqrt{\sum_{i=1}^{n} |v_{i}|^{2}} \leq (1 + \epsilon)^{2}
\end{align}
by definition of~$g^{i}$, $g_{i}$, and of~$g(p)(v) = \sum_{i, j=1}^{n} v_{i} v_{j} g_{i} (p) \cdot g_{j} (p)$.

 By~\eqref{e:apsym1.3.99} of Theorem~\ref{t:apsym} and by~\eqref{e:eps} we have that 
\begin{align}
\label{e:eu_M}
    \int_{U}&  \ |e(u)(\psi(p))| \, \di \mathcal{H}^n(p) = \int_{\psi(U)} |e(u)(x)| \, J\psi^{-1}(x) \, \di x  
    \\
    &
    \leq \text{Lip}(\psi^{-1};\psi(U))^n  \int_{\psi(U)} |e(u)(x)|\, \di x  \leq  \text{Lip}(\psi^{-1};\psi(U))^n \overline{\lambda} (\psi(U)) \nonumber
    \\
    &
    \leq (1 + \epsilon)^{n+1} \psi_{\sharp} \lambda (\psi(U)) = (1 + \epsilon)^{n+1}\lambda(U)\,. \nonumber
    \end{align}
    Repeating the covering argument of~\eqref{e:phi-P3}--\eqref{e:phi-99} on ${\rm M}$, we infer from~\eqref{e:Lp}--\eqref{e:eu_M} and from the arbitrariness of~$\epsilon>0$ the validity of \eqref{e:l1appsym}. This concludes the proof of the theorem.\EEE
    \end{proof}

    \section*{Acknowledgments}
The work of the authors was partially funded by the Austrian Science Fund (\textbf{FWF}) through the project P35359-N. S.A. was also supported by the FWF project ESP-61. E.T. further acknowledges the support of the FWF projects Y1292 and F65. Finally, the authors acknowledge the warm hospitality of ESI, Vienna during the workshop \emph{Between Regularity and Defects: Variational and Geometrical Methods in Materials Science}, where part of this research was carried out.

\appendix

\bibliographystyle{siam}
\bibliography{biblio}

%\begin{thebibliography}{99}

%\bibitem{afp} AMBROSIO, L., FUSCO, N., and PALLARA, A.: {\em Functions of bounded variation and free discontinuity problems}.

%\bibitem{dal} DAL MASO, G.: \emph{Generalised functions of bounded deformation.} Journal of the European Mathematical Society 15.5, 1943-1997, (2013).

%\bibitem{del} DEL NIN, G.: \emph{Rectifiability of the jump set of locally integrable functions.} arXiv preprint arXiv:2001.04675 (2020).

%\bibitem{eva} EVANS, L. C., and RONALD F. G.: \emph{Measure theory and fine properties of functions}. Routledge, 2018.

%\bibitem{fed} FEDERER, H.: \emph{Geometric measure theory}. Springer, 1996.

%\bibitem{hov} HOVILA, R., et al.: \emph{Besicovitch-Federer projection theorem and geodesic flows on Riemann surfaces.} Geometriae Dedicata 161.1, 51-61, (2012).

%\end{thebibliography}

\end{document}